\documentclass[11pt,twoside]{amsart} % amsart / amsbook
\usepackage[a4paper, left=28mm, right=28mm, top=28mm, bottom=34mm]{geometry}

\usepackage{mypreamble}
\newcommand{\basek}{\Bbbk} % base ring
\newcommand{\Modk}{\Mod(\basek)}
\newcommand{\Chk}{\Ch(\basek)}
\newcommand{\Homcpx}{\Hom^\bullet} % Hom complex in Ch(k)
\newcommand{\dgCat}{\catname{dgCat}}
\newcommand{\dgCatk}{\catname{dgCat}_\basek}
\newcommand{\HodgCat}{\catname{HodgCat}}
\newcommand{\intHom}{\mathbb{R}\underline{\mathrm{Hom}}} % internal Hom of hodgCat
\newcommand{\HotwodgCat}{\catname{Ho}_2\catname{dgCat}} % homotopy 2-category
\newcommand{\CAT}{\catname{CAT}} % cat of large categories
\newcommand{\dg}{\mathrm{dg}}
\newcommand{\Comdg}{C_\dg}
\newcommand{\Comdgk}{\Comdg(\basek)}
\newcommand{\Com}{C} % dg module
\newcommand{\Kom}{K}
\newcommand{\Dom}{D}
\newcommand{\Perf}{\catname{Perf}}
\DeclareMathOperator{\Cone}{Cone}
\DeclareMathOperator{\Cocone}{Cocone}
\newcommand{\Acyc}{\catname{Acyc}} % acyclic complex
\newcommand{\hproj}{\catname{h}\text{-}\catname{proj}}
\newcommand{\hinj}{\catname{h}\text{-}\catname{inj}}
\DeclareMathOperator{\hp}{\mathsf{p}} % h-projective replacement
\DeclareMathOperator{\hi}{\mathsf{i}} % h-injective replacement
\newcommand{\cofQ}{Q} % cofibrant replacement of dg categories
\newcommand{\cofresol}{q} % cofibrant resolution

\newcommand{\familyS}{\mathrm{S}} % family of objects
\newcommand{\tria}[1]{\catname{tria}(#1)} % smallest tria full subcat
\newcommand{\thick}[1]{\catname{thick}(#1)} % smallest thick subcat
\newcommand{\Loc}[1]{\catname{Loc}(#1)} % smallest localizing subcat
\newcommand{\cpt}{\text{cpt}}

\newcommand{\Bimod}{\catname{Bimod}}
\newcommand{\DBimod}{\catname{DBimod}}
\newcommand{\DBimodhf}{\DBimod_\text{hf}}
\newcommand{\DBimodhp}{\DBimod_\text{hp}}
\newcommand{\rqr}{\text{rqr}}
\newcommand{\DBimodrqr}{\DBimod^\rqr}
\newcommand{\procomp}{\odot} % composition for profunctor
\newcommand{\Dprocomp}{\procomp^\mathbb{L}} % derived procomposition
\newcommand{\bbL}{\mathbb{L}}
\newcommand{\bbR}{\mathbb{R}}
\newcommand{\bbLT}{\bbL T}
\newcommand{\bbRH}{\bbR H}
\newcommand{\bbLtil}{\widetilde{\bbL}}
\newcommand{\Dddag}{{\bbR\ddag}} % derived daggar
\newcommand{\DIsbellL}{\bbL\mathcal{L}} % derived Isbell dual L
\newcommand{\DIsbellR}{\bbR\mathcal{R}} % derived Isbell dual R
\DeclareMathOperator{\DRan}{DRan} % derived Ran
\DeclareMathOperator{\DRift}{DRift} % derived Rift

\newcommand{\qucl}[1]{\overline{#1}} % the full subcat of quasi-representable modules
\newcommand{\Htil}{\widetilde{H}^0}

\newcommand{\Winit}{{W_\text{init}}}
\newcommand{\Vterm}{{V_\text{term}}}
\newcommand{\Wcoprod}{{W_\text{coprod}}}
\newcommand{\Vprod}{{V_\text{prod}}}

\newcommand{\Wcoshift}[1]{{W_{#1\text{-coshift}}}}
\newcommand{\Vshift}[1]{{V_{#1\text{-shift}}}}
\newcommand{\Wcone}{{W_\text{cone}}}
\newcommand{\Vcocone}{{V_\text{cocone}}}

\newcommand{\localization}{\delta}

\newcommand{\bicatK}{\mathcal{K}}
\newcommand{\bicatL}{\mathcal{L}}
\newcommand{\bicatM}{\mathcal{M}}

\newcommand{\moncatV}{\mathcal{V}}

\newcommand{\Prof}{\catname{Prof}}
\newcommand{\VCat}{\moncatV\text{-}\Cat}
\newcommand{\VProf}{\moncatV\text{-}\Prof}
\newcommand{\Span}{\catname{Span}}

\newcommand{\co}{\text{co}}
\newcommand{\coop}{\text{coop}}

\newcommand{\IsbellL}{\mathcal{L}}
\newcommand{\IsbellR}{\mathcal{R}}
\newcommand{\equipsubAst}{\ast}
\newcommand{\equipAst}{(\mplaceholder)_{\equipsubAst}}
\newcommand{\coequipAst}{(\mplaceholder)^{\equipsubAst}}
\newcommand{\equip}{\equipAst}
\newcommand{\coequip}{\coequipAst}

\newcommand{\equipsubStar}{\star}
\newcommand{\equipStar}{(\mplaceholder)_{\equipsubStar}}

\newcommand{\equipsubBullet}{\bullet}
\newcommand{\equipBullet}{(\mplaceholder)_{\equipsubBullet}}

\makeatletter
\def\slashedarrowfill@#1#2#3#4#5{%
	$\m@th\thickmuskip0mu\medmuskip\thickmuskip\thinmuskip\thickmuskip
	\relax#5#1\mkern-7mu%
	\cleaders\hbox{$#5\mkern-2mu#2\mkern-2mu$}\hfill
	\mathclap{#3}\mathclap{#2}%
	\cleaders\hbox{$#5\mkern-2mu#2\mkern-2mu$}\hfill
	\mkern-7mu#4$%
}
\def\rightslashedarrowfill@{%
	\slashedarrowfill@\relbar\relbar\mapstochar\rightarrow}
\newcommand\xslashedrightarrow[2][]{%
	\ext@arrow 0055{\rightslashedarrowfill@}{#1}{#2}}
\makeatother
\def\slashedrightarrow{\xslashedrightarrow{}}
\DeclareMathAlphabet{\mathbbold}{U}{bbold}{m}{n}
\newcommand{\warrcat}{\mathbbold{2}} % warking arrow category {0<1}

\newcommand{\mymemo}[1]{\textcolor{orange}{#1}} % memo用

\usepackage{todonotes}
\setuptodonotes{inline}

\title[A formal category theoretic approach]{A formal category theoretic approach to the homotopy theory of dg categories}
\author[Y.\ Imamura]{Yuki Imamura}
\date{\today}%%%%%%%%%%%%%%%%%%%
\address[Y.\ Imamura]{%
	Research Institute for Mathematical Sciences, Kyoto University, Kyoto 606-8502, Japan
}
\email{u287972b@alumni.osaka-u.ac.jp}

\begin{document}

\begin{abstract}
	We introduce a bicategory that refines the localization of the category of dg categories with respect to quasi-equivalences and investigate its properties via formal category theory.
	Concretely, we first introduce the bicategory of dg categories \(\mathsf{DBimod}\), whose Hom categories are given by the derived categories of dg bimodules, and then define the desired bicategory as the sub-bicategory $\mathsf{DBimod}^\text{rqr}$ consisting of right quasi-representable dg bimodules.

	The first half of the paper is devoted to the study of adjunctions and equivalences in these bicategories.
	We then show that the embedding $\mathsf{DBimod}^\text{rqr} \hookrightarrow \mathsf{DBimod}$ forms a proarrow equipment in the sense of Richard J. Wood, which provides a framework for formal category theory and enables us to define (weighted) (co)limits in an abstract setting. 
	From this proarrow equipment, we derive the notion of homotopical (co)limits in dg categories, including homotopical shifts and cones, which in turn allows us to give a formal characterization of pretriangulated dg categories.
	As an application, we provide a conceptual proof of the fact that pretriangulatedness is preserved under the gluing procedure, and we establish reflection results concerning adjoints and colimits.
\end{abstract}

\maketitle

\tableofcontents
%% tableofcontents には資料データである adrees email keywords subjclass などが必要

\setcounter{section}{-1}
\section{Introduction}

%% theoremのカウンタをアルファベットにする
\renewcommand{\thetheorem}{\Alph{theorem}}

Let $\basek$ be a commutative ring with unit (we will assume from \cref{section:the_homotopy_category_theory_of_dg_categories_over_a_field} onwards that $\basek$ is a field for simplicity).

A \emph{differential graded category} (or \emph{dg category}) over $\basek$ is an enriched category over the symmetric monoidal closed category $\Chk$ of cochain complexes of $\basek$-modules.
If a dg category $\catA$ is \emph{pretriangulated}, then its \emph{homotopy category} $H^0(\catA)$ (the preadditive category obtained by taking $0$-th cohomology of Hom complexes of $\catA$) has a triangulated structure. Hence (pretriangulated) dg categories are viewed as an enhancement of triangulated categories, and used in algebraic geometry and representation theory (\cite{Bondal-Kapranov:1991}, \cite{Keller:2006arXiv}).

Besides the usual notion of equivalences as enriched categories, dg categories possess a weaker notion of equivalences called quasi-equivalences.
A dg functor $F\colon \catA\to \catB$ is called a \emph{quasi-equivalence} if all $F_{A,A'}\colon \catA(A,A') \to \catB(FA,FA')$ are quasi-isomorphisms for $A,A'\in\catA$ and if the induced functor $H^0(F)\colon H^0(\catA)\to H^0(\catB)$ is an equivalence of categories. Quasi-equivalences are considered to be the right notion of equivalences between dg categories (\cite[p.\ 617]{Toen:2007}).
In \cite{Tabuada:2005quillen_model_dgCat}, Tabuada established the homotopy theory of dg categories by constructing a model structure on the category $\dgCat$ of small dg categories whose weak equivalences are the quasi-equivalences. From this we get the localization $\HodgCat$ of $\dgCat$ with respect to quasi-equivalences, which is called the \emph{homotopy category of dg categories}.

In addition to the model structure, the category $\dgCat$ has the structure of a symmetric monoidal closed category. Unfortunately the monoidal structure on $\dgCat$ is not compatible with the model structure in the sense that it is not a monoidal model category. Nevertheless To\"en proved in \cite{Toen:2007} that the homotopy category $\HodgCat$ has a symmetric monoidal closed structure $(\otimes^\bbL,\basek,\intHom)$. The internal Hom $\intHom(\catA,\catB)$ can be defined as the dg category $\Dom_\dg(\catA,\catB)^\rqr$ of \emph{right quasi-representable dg bimodules} (see \cite[Theorem 6.1]{Toen:2007}).

Furthermore, $\dgCat$ is a $2$-category as well. Does the 2-category structure on $\dgCat$ descend to the homotopy category $\HodgCat$? The $2$-structure on $\dgCat$ is also incompatible with the model structure. Yet, the following proposition hints at a $2$-categorical structure on $\HodgCat$.

\begin{proposition*}[{\cite[Corollary 4.8]{Toen:2007}; see also \cite{Canonaco-Stellari:2015Internal_Homs}}]
	There exists a bijection
	\[ \Hom_\HodgCat(\catA,\catB) \cong H^0\big( \intHom(\catA,\catB) \big) /{\cong}, \]
	where the right hand side denotes the set of isomorphism classes of the objects.
\end{proposition*}

This result suggests that a $2$-category with $H^0 \big( \intHom(\catA,\catB) \big)$ as Hom-categories would serve as a refinement of $\HodgCat$. Once such a refining $2$-category is obtained, it is expected that we can develop a homotopical version of category theory for dg categories, just as we use the $2$-category structure on $\dgCat$ to develop category theory of dg categories (as enriched categories).

The purpose of this paper is to construct such a $2$-category in terms of dg bimodules and to study its properties through the lens of formal category theory.
Specifically, rather than a mere $2$-category, we construct it as the bicategory $\DBimodrqr$ of \emph{right quasi-representable} dg bimodules, which is a sub-bicategory of the derived bicategory $\DBimod$ of dg bimodules.
These bicategories have been previously considered by several authors, including
%It is worth noting that these bicategories have already been considered by several authors: 
\cite{Keller:1998invariance},
\cite{Keller:2005triangulated_orbit},
\cite{Lunts-Orlov:2010uniqueness_of_enhancements},
\cite{Johnson:2014azumaya},
\cite{Genovese:2017},
\cite{Anno-Logvinenko:2017spherical},
\cite{Orlov:2020finite-dimensional},
\cite{Yekutieli:2020Derived_categories},
\cite{Bodzenta-Bondal:2022},
\cite{Campbell-Ponto:2023riemann-roch_in_monoidal_2-cat}, 
among others.
While the construction is not novel, we will include a detailed account for the sake of completeness.
The Hom category $\DBimodrqr(\catA,\catB)=\Dom(\catA,\catB)^\rqr$ is equivalent to $H^0\big( \intHom(\catA,\catB) \big)$, and thus we call the bicategory $\DBimodrqr$ the \emph{homotopy bicategory of dg categories}.
%the bicategory $\DBimodrqr$ is an ``incarnation'' of a 2-category refinement of HodgCat.

In the second half of the paper, we present an approach to the homotopy theory of dg categories from the viewpoint of proarrow equipments. Proarrow equipments are known to provide a framework for formal category theory (\cite{Wood:1982proarrow1,Wood:1985proarrow2}); see \cref{section:proarrow_equipment} for an overview. We can observe that the bicategory $\DBimodrqr$ together with $\DBimod$ forms a proarrow equipment $\equipStar\colon \DBimodrqr\hookrightarrow\DBimod$, which enables the development of category theory within the bicategory $\DBimodrqr$. Applying the general theory of proarrow equipment, we introduce the notion of homotopical (co)limits in a dg category, which are respected by quasi-equivalences.
We could therefore refer to the resulting category theory in $\DBimodrqr$ as the \textit{homotopy category theory} of dg categories.

\vspace{1em}\noindent\textbf{Results of this paper.}
A \emph{dg bimodule} $X\colon \catA \slashedrightarrow\catB$ between dg categories is defined as a dg functor $X\colon \catB^\op\otimes\catA \to \Comdgk$, where $\Comdgk$ is the dg category of complexes of $\basek$-modules. A dg bimodule $X\colon \basek \slashedrightarrow\catB$ from the unit dg category $\basek$ is nothing but a (right) dg $\catB$-module.
All dg categories and the categories $\Com(\catA,\catB)$ of dg bimodules assemble into a bicategory $\Bimod$ whose composition $\procomp$ is given by the tensor product of dg bimodules.

The bicategory $\DBimod$ is a derived version of $\Bimod$.
Let $\Dom(\catA,\catB)$ be the derived category of dg bimodules. The \emph{derived composition}
\[ \Dprocomp\colon \Dom(\catB,\catC) \times \Dom(\catA,\catB) \to \Dom(\catA,\catC) \]
is defined as the total left derived functor of the composition $\procomp$. In general, the derived composition is not associative or unital. However, it is associative and unital when we restrict to \emph{locally h-projective} (or \emph{locally h-flat}) dg categories.
Namely we have the bicategory $\DBimodhp$ (resp.\ $\DBimodhf$) such that
\begin{itemize}
	\item the objects are locally h-projective (resp.\ locally h-flat) small dg categories;
	\item the Hom categories $\DBimod(\catA,\catB)$ are the derived categories $\Dom(\catA,\catB)$ of dg bimodules.
\end{itemize}
We remark that if the base ring $\basek$ is a field, then all dg categories are locally h-projective (and locally h-flat). We define $\DBimod\coloneqq \DBimodhp$ and refer to it as the \emph{derived bicategory of dg bimodules}.

In \cref{section:adjunctions_in_DBimod}, we study adjunctions in $\DBimod$ and characterize when a dg bimodule is left adjoint there, generalizing \cite[Corollary 6.6]{Genovese:2017}.

\begin{theorem}[{\cref{thm:right_compact_bimodules_are_right_adjoint}}]\label{theorem-N}
	Let $\catA,\catB$ be locally h-projective small dg categories and $X\colon \catA\slashedrightarrow\catB$ be a dg bimodule. Then the following are equivalent.
	\begin{enumerate}[label=\equivitem]
		\item The dg bimodule $X$ has a right adjoint in the bicategory $\DBimod$.
		\item The derived Hom functor $\bbRH_X\colon \Dom(\catB)\to\Dom(\catA)$ preserves coproducts.
		\item $X$ is right compact.
	\end{enumerate}
\end{theorem}

One can find a partial result of \cref{theorem-N} in \cite[Proposition 2.11]{Anno-Logvinenko:2017spherical}. We also provide a characterization of equivalences in $\DBimod$.

\begin{theorem}[{\cref{thm:equivalence_in_DBimod}}]\label{theorem-M}
	Let $\catA,\catB$ be locally h-projective dg categories and
	$X\colon \catA\slashedrightarrow\catB$ be a dg bimodule. Then the following are equivalent.
	\begin{enumerate}[label=\equivitem]
		\item $X$ is an equivalence in the bicategory $\DBimod$.
		\item The derived tensor functor $\bbLT_X\colon \Dom(\catA)\to\Dom(\catB)$ is an equivalence of categories.
	\end{enumerate}
\end{theorem}

In \cref{section:quasi-functors}, we introduce the bicategory $\DBimodrqr$ of quasi-functors. We say that a dg bimodule $M\in \Dom(\catB)$ is \emph{quasi-representable} if $M\cong \catB(\mplaceholder,B)$ in $\Dom(\catB)$ for some $B\in \catB$. Let $\qucl{\catB}\subseteq\Dom(\catB)$ be the full subcategory of quasi-representables.
A dg bimodule $X\colon \catA\slashedrightarrow\catB$ is called \emph{right quasi-representable}, or a \emph{quasi-functor}, if for all $A\in \catA$ the dg $\catB$-modules $X(\mplaceholder,A)$ are quasi-representable, or equivalently if the derived tensor functor $\bbLT_X\colon \Dom(\catA)\to\Dom(\catB)$ preserves quasi-representables. Let $\Dom(\catA,\catB)^\rqr \subseteq \Dom(\catA,\catB)$ be the full subcategory of quasi-functors, and define $\DBimodrqr\subseteq \DBimod$ as the sub-bicategory of quasi-functors. We think of $\DBimodrqr$ as a bicategory refining the homotopy category $\HodgCat$, and call it the \emph{homotopy bicategory of dg categories}. We observe equivalences of $\DBimodrqr$, as follows.

\begin{theorem}[{\cref{thm:equivalence_in_DBimodrqr}}]\label{theorem-L}
	Let $\catA,\catB$ be locally h-projective dg categories and $X\colon \catA\slashedrightarrow\catB$ be a quasi-functor. Then the following are equivalent.
	\begin{enumerate}[label=\equivitem]
		\item $X$ is an equivalence in the bicategory $\DBimodrqr$ of quasi-functors.
		\item Both the left derived functor $\bbLT_X\colon \Dom(\catA)\to\Dom(\catB)$ and the restricted one $\res{\bbLT_X}{\qucl{\catA}}\colon \qucl{\catA}\to \qucl{\catB}$ are equivalences of categories.
	\end{enumerate}
	Note that a dg bimodule satisfying the condition (ii) is called a \emph{quasi-equivalence} in \cite[\S7.2]{Keller:1994Deriving}.
\end{theorem}

For a dg functor $F\colon \catA\to \catB$, the associated dg bimodule $F_\equipsubAst=\catB(\mplaceholder,F \mplaceholder)\colon \catA\slashedrightarrow\catB$ is clearly a quasi-functor. When $F$ is a quasi-equivalence, $F_\equipsubAst$ is an equivalence in $\DBimodrqr$; see \cite[\S7.2, Example]{Keller:1994Deriving}.
The next proposition states that the converse holds, which is also mentioned at the end of \cite[\S2.3]{Genovese-Lowen-VandenBergh:2022derived_deformation:arXiv} without proof.

\begin{proposition}[{\cref{prop:quasi-equivalence_is_just_equivalence_in_DBimodrqr}}]\label{Theorem-K}
	Let $\catA,\catB$ be locally h-projective dg categories and $F\colon \catA\to \catB$ be a dg functor. For the associated quasi-functor $F_\equipsubAst\colon\catA\slashedrightarrow\catB$, the following are equivalent.
	\begin{enumerate}[label=\equivitem]
		\item $F$ is a quasi-equivalence (in the sense of \cref{def:quasi-equivalence}).
		\item $F_\equipsubAst$ is an equivalence in the bicategory $\DBimodrqr$ of quasi-functors.
		\item Both the left derived functor $\bbLT_{F_\equipsubAst}\colon \Dom(\catA)\to\Dom(\catB)$ and the restricted one $\res{\bbLT_{F_\equipsubAst}}{\qucl{\catA}}\colon \qucl{\catA}\to \qucl{\catB}$ are equivalences of categories.
	\end{enumerate}
\end{proposition}

An adjunction in the bicategory $\DBimodrqr$ is referred to as an \emph{adjunction of quasi-functors}; see \cite[\S7]{Genovese:2017}. In \cref{subsection:adjunction_of_quasi-functors}, we investigate the relation between adjoint quasi-functors and adjoint dg functors, which has not been explored in \cite{Genovese:2017}. By showing that there exists a normal lax functor $\Gamma\colon \Bimod\to\DBimod$ and it induces a pseudo-functor $\gamma\colon \dgCat\to\DBimodrqr$, we obtain the following:

\begin{proposition}[{\cref{prop:adj_of_dg_functors_becomes_those_of_quasi-functors}}]\label{proposition-H}
	For an adjunction $F\dashv G$ of dg functors, we have an adjunction $F_\equipsubAst \dashv G_\equipsubAst$ of quasi-funcors.
\end{proposition}

\Cref{section:the_homotopy_category_theory_of_dg_categories_over_a_field} is the main section of this paper. In this section, we develop the homotopy category theory of dg categories (over a field). Now let us write the inclusion $\DBimodrqr\hookrightarrow \DBimod$ by $\equipStar$.
By \cite[Corollary 6.6]{Genovese:2017} or \cref{theorem-N}, quasi-functors have right adjoints in $\DBimod$. Summarizing, we find that the pseudo-functor $\equipStar$ satisfies:
\begin{enumerate}
	\item $\equipStar$ is an identity on objects;
	\item $\equipStar$ is locally fully faithful;
	\item for every quasi-functor $f$, the dg bimodule $f=f_\equipsubStar$ has a right adjoint in $\DBimod$.
\end{enumerate}
Such a pseudo-functor is called a \emph{proarrow equipment}. This notion was introduced in the field of formal category theory; basic references are \cite{Wood:1982proarrow1,Wood:1985proarrow2} (see also \cref{section:proarrow_equipment} for a review). 
Formal category theory seeks to axiomatize ordinary category theory within an abstract $2$-category, and a proarrow equipment offers a formal framework for developing categorical reasoning.
For example, the pseudo-functor $\equipAst\colon\dgCat \rightarrow \Bimod$, sending a dg functor $F\colon \catA\to\catB$ to the dg bimodule $F_\equipsubAst=\catB(\mplaceholder,F\mplaceholder)$, forms a proarrow equipment, which supports $\Chk$-enriched category theory.  
The observation that $\equipStar\colon \DBimodrqr\hookrightarrow\DBimod$ is a proarrow equipment suggests that we can develop a version of category theory within $\DBimodrqr$, which may be regarded as a homotopical version of dg category theory.

In a proarrow equipment, we can define several notions in category theory, such as (weighted) (co)limits, Cauchy completeness, and pointwise Kan extensions.
In \cref{subsection:homotopical_(co)limits_in_a_dg_category}, applying the notion of (co)limits to the proarrow equipment $\equipStar\colon \DBimodrqr\hookrightarrow\DBimod$, we introduce the notion of homotopical (co)limits in a dg category (see \cref{def:homotopical_colimits_in_a_dg_category}). As special (co)limits, we obtain homotopical 
initial objects%(\cref{example:h-initial_object})
, coproducts%(\cref{example:h-coproduct})
, shifts%(\cref{example:h-tensor})
, and cocones%(\cref{example:h-cocone})
. In particular, we have a formal characterization of pretriangulated dg categories as those admitting certain homotopical limits.

\begin{corollary}[{\cref{cor:pretriangulated_iff_homotopy_cone_and_homotopy_shift}}]\label{corollary-J}
	A dg category $\catA$ has homotopical shifts and cocones if and only if it is pretriangulated in the sense that $H^0(\catA)$ is closed (up to isomorphism) under shifts and cones of $\Dom(\catA)$.
\end{corollary}

In \cref{subsection:homotopy_Cauchy_complete_dg_categories}, we discuss preservation of (co)limits in our setting. 
We emphasize that quasi-equivalences respect homotopical (co)limits, since they become equivalences of $\DBimodrqr$ by \cref{Theorem-K}.
Using \cref{theorem-N}, we verify that initial and terminal objects, finite (co)products, shifts, and (co)cones are \emph{absolute}, i.e., preserved by any quasi-functor.
We also define a dg category $\catA$ to be \emph{h-Cauchy complete} when all left adjoint dg bimodules in $\DBimod$ with codomain $\catA$ are right quasi-representable. From \cref{theorem-N}, we find that $\catA$ is h-Cauchy complete if and only if the inclusion $H^0(\catA)\hookrightarrow \Perf(\catA)\coloneqq \Dom(\catA)^\cpt$ is an equivalence of categories (\cref{prop:characterization_of_h-Cauchy_complete}).

In \cref{subsection:application}, we see an application of the formal description of being pretriangulated to the gluing of dg categories along a dg bimodule.
%to some constructions of dg categories which have a bicategorical universality.
%, such as %quotients and gluings.
For example, it is shown in \cite[Lemma 4.3]{Kuznetsov-Lunts:2015} that the gluing of pretriangulated dg categories is again pretriangulated. \Cref{prop:gluing_has_right-adjoint-weighted_limits} gives a slight generalization and an alternative, more conceptual proof of this result.

We show that adjunctions of quasi-functors induce ordinary adjunctions between the corresponding $H^0$-categories, and that if a dg category has h-coproducts, its $H^0$-category admits coproducts in the ordinary sense.
Moreover, assuming the existence of h-shifts, we prove that both left adjoints and h-coproducts are reflected under passage to $H^0$-categories, as follows. 
This marks a key difference from the strict theory of dg categories, where different assumptions are required, and highlights the advantage of our homotopical framework.

\begin{theorem}[{\cref{thm:reflection_theorem_of_left_adjoint_for_quasi-functor,thm:reflection_theorem_of_h-coproduct}}]
	\label{theorem-AA}
	Let $\catA,\catB$ be dg categories with h-shifts.
	\begin{enumerate}
		\item Let $f\colon \catA\to\catB$ be a quasi-functor. Then $f$ admits a left adjoint as a quasi-functor if and only if the induced functor $\Htil(f)=\res{\bbLT_{f_\equipsubStar}}{\qucl{\catA}} \colon \qucl{\catA}\to \qucl{\catB}$ has a left adjoint.
		\item Let $\{A_j\}_{j\in J}$ be a family of objects of $\catA$. Then $\catA$ admits an h-coproduct of $\{A_j\}_j$ if and only if $H^0(\catA)$ has a coproduct of $\{A_j\}_j$.
	\end{enumerate}
\end{theorem}

As an immediate consequence of \cref{theorem-AA}, we deduce the following result.

\begin{corollary}[{\cref{cor:cocompleteness_for_pretri_dg_cat,cor:cocontinuity_for_quasi-functor_bw_pretri_dg_cat}}]
	\label{corollary-AB}
	Let $\catA,\catB$ be pretriangulated dg categories.
	\begin{enumerate}
		\item $\catA$ has all h-colimits if and only if $H^0(\catA)$ has all coproducts.
		\item A quasi-functor $f\colon \catA\to \catB$ preserves all h-colimits if and only if the induced functor $\Htil(f)=\res{\bbLT_{f_\equipsubStar}}{\qucl{\catA}}\colon \qucl{\catA}\to\qucl{\catB}$ preserves all coproducts.
	\end{enumerate}
\end{corollary}

In the literature, the notions of cocompleteness for dg categories and cocontinuity for quasi-functors have been considered --- for example \cite[Definition 6.1]{Porta:2010} and \cite[Remark 3.9, Definition 2.5]{Lowen-RamosGonzalez:2022tensor_product_of_well_generated}. 
In light of \cref{corollary-AB}, our framework helps clarify the conceptual basis of the terminology used in these works.

\vspace{1em}\noindent\textbf{Outline.}
The paper is organized as follows.
\Cref{section:preliminaries_on_triangulated_categories} and \cref{section:Basic_dg_category_theory} are devoted to preliminaries on triangulated categories and dg categories. In \cref{subsection:(co)end_calculus} we summarize results on ends and coends which will be freely used in the sections that follow.

From \cref{section:dg_bimodules_and_their_bicategorical_structures} to \cref{section:quasi-functors}, we construct and study the bicategories $\DBimod$ and $\DBimodrqr$.
We investigate adjunctions, equivalences, right Kan extensions, and right Kan liftings there. We also construct a pseudo-functor $\bbLtil\colon \DBimod \to \CAT$ which assigns dg categories $\catA$ to the derived categories $\Dom(\catA)$ and dg bimodules $X$ to the left derived functors $\bbLT_X$ of tensoring with $X$.
We also study adjunctions of quasi-functors and their relation to adjunctions of dg functors, which \cite{Genovese:2017} does not mention.

%In \cref{subsection:the_tensor_and_Hom_functors_associated_with_dg_bimodules}, we consider the derived tensor functor $\bbLT_X$ and the derived Hom functor $\bbRH_X$ associated with dg bimodules. We see that assignments $\catA\mapsto\Dom(\catA)$, $X\mapsto \bbLT_X$ form a pseudo functor $\bbLtil\colon \DBimod \to \CAT$ which is locally conservative.

In \cref{section:the_homotopy_category_theory_of_dg_categories_over_a_field}, we develop the homotopy category theory of dg categories (over a field), including homotopical (co)limits, absoluteness of (co)limits, and homotopy Cauchy completeness. We apply them to the gluing of dg categories. 
\cref{subsection:reflection_of_adjoints_and_colimits} establishes the reflection theorems for left adjoints and h-coproducts under the $H^0$-construction.

Finally, in \cref{section:universality_of_htpy_cat_theory}, we point out that the pair $(\Gamma,\gamma)$ behaves like a ``morphism of proarrow equipments'' in a sense and that the proarrow equipment $\equipStar\colon \DBimodrqr\hookrightarrow\DBimod$ has a universal property similar to that of localizations (\cref{thm:universality_2}).

\Cref{section:proarrow_equipment} contains definitions and basic results on proarrow equipments. It mostly follows \cite{Wood:1982proarrow1} except for \cref{subsection:absolute_limits_and_Cauchy_completeness}, which is folklore.

\begin{acknowledgements*}
	The author would like to express his gratitude to his supervisor, Shinnosuke Okawa, for a lot of advice through regular seminars.
	The author is also very grateful to Hisashi Aratake for inviting him to the world of category theory and giving him opportunities for learning topics on enriched categories and $2$-categories, including proarrow equipments.

	This research was conducted while the author was affiliated with the University of Osaka, and some of the results are part of the author's PhD dissertation.
	This work was supported by JST SPRING Grant Number JPMJSP2138 and JSPS KAKENHI Grant Number JP23KJ1488.
\end{acknowledgements*}

%% theoremのカウンタをもとに戻す
\setcounter{theorem}{0}%
\renewcommand{\thetheorem}{\thesection.\arabic{theorem}}

%% We tacitly assume that all categories discussed in this paper are locally small. 
%% All (co)limits are indexed by small categories.

\section{Preliminaries on triangulated categories}\label{section:preliminaries_on_triangulated_categories}

Let $\catT$ be a triangulated category.
In this section, we assume that all subcategories of triangulated categories are strictly full; i.e. full and closed under isomorphisms\footnote{A subcategory closed under isomorphisms is also called \emph{replete}.}.
A subcategory $\catS$ of $\catT$ is called a \emph{triangulated subcategory} if $\catS$ is closed under shifts and cones in $\catT$. A triangulated subcategory $\catS$ of $\catT$ is \emph{thick} if it is closed under direct summands. When $\catT$ has small coproducts, a triangulated subcategory $\catS$ is called \emph{localizing} if it is closed under coproducts. We remark that localizing subcategories are thick.

For a set $\familyS$ of objects of $\catT$, let $\tria{\familyS}$ denote the smallest triangulated subcategory of $\catT$ containing $\familyS$, and $\thick{\familyS}$ the smallest thick subcategory containing $\familyS$. Also, let $\Loc{\familyS}$ denote the smallest localizing subcategory containing $\familyS$, when $\catT$ has small coproducts.
We have the sequence of inclusion relations $\tria{\familyS}\subseteq \thick{\familyS}\subseteq \Loc{\familyS}$.

For a set $\familyS$ of objects of $\catT$, we define the \emph{right orthogonal subcategory} $\familyS^\perp$ of $\familyS$ to be the full subcategory of $\catT$ spanned by those objects $X$ which $\Hom_\catT(C[n],X)=0$ for all $C\in\familyS$ and all $n \in \Z$.

\begin{definition}[\cite{Bondal-van_den_Bergh:2003generators}]
	Let $\catT$ be a triangulated category with coproducts.
	\begin{enumerate}
		\item An object $C\in\catT$ is called \emph{compact} if the functor $\Hom_\catT(C,\mplaceholder)\colon \catT\to\Ab$ preserves coproducts.
		\item A set $\familyS$ of objects of $\catT$ \emph{generates} $\catT$ if $\familyS^\perp=0$.
		\item $\catT$ is \emph{compactly generated} if it is generated by a set $\familyS$ of compact objects. The set $\familyS$ is called a \emph{generating set} of $\catT$.
	\end{enumerate}
\end{definition}

For example, the derived category $\Dom(R)$ of a ring $R$ is generated by the single compact generator $R$. More generally, we will see that the derived category of a dg category is also compactly generated by the representables (\cref{prop:Dom(catA)_is_compactly_generated}).

%The following is a famous result due to Neeman.

\begin{proposition}[\cite{Neeman:1992connection, Neeman:1996brown_representability, Neeman:2001triangulated}]\label{prop:Brown_representability_theorem}
	Let $\catT$ be a compactly generated triangulated category with $\familyS$ a generating set.
	\begin{enumerate}
		\item The full subcategory $\catT^\cpt$ consisting of all compact objects of $\catT$ coincides with $\thick{\familyS}$.
		\item The equality $\Loc{\familyS}=\catT$ holds.
	\end{enumerate}
\end{proposition}

\begin{proposition}[{cf.~\cite[Lemma\ 4.2]{Keller:1994Deriving}}]\label{prop:lemma_for_compactly_generated_tria_cat}
	Let $\catT,\catT'$ be triangulated categories with coproducts, $F,G\colon \catT\to\catT'$ triangulated functors preserving coproducts, and $\alpha\colon F\Rightarrow G$ a natural transformation. If $\catT$ is compactly generated by a set $\familyS$ of objects, then the following hold.
	\begin{enumerate}
		\item The restricted functor $\res{F}{\tria{\familyS}}\colon \tria{\familyS}\to \catT'$ is fully faithful if and only if for all $C,C'\in \familyS$ and all $n \in \Z$,
		      \[ F\colon \Hom_\catT(C,C'[n]) \to \Hom_{\catT'}(FC,FC'[n]) \]
		      is bijective.
		\item If (1) holds and if for each $C\in \familyS$ the object $FC$ of $\catT'$ is compact, then $F\colon \catT\to \catT'$ is fully faithful.
		\item $F$ is an equivalence if and only if $F$ is fully faithful and the set $\{FC\}_{C\in \familyS}$ of objects compactly generates $\catT'$.
		\item $\alpha$ is an isomorphism if and only if $\alpha_C\colon FC\to GC$ is invertible for each $C\in \familyS$.
	\end{enumerate}
\end{proposition}

\begin{comment}
Let $\catT,\catT'$ be compactly generated triangulated categories with $\familyS,\familyS'$ generating sets, respectively, and $F\colon \catT\to\catT'$ a triangulated functor preserving coproducts. If the restriction $\res{F}{\Z S}\colon \Z\familyS\to \Z\familyS'$ is an equivalence, then $F$ is an equivalence. Here $\Z\familyS\coloneqq \{C[n] \mid C\in \familyS,\> n \in \Z\}$.\mymemo{使ってない？から不要かも}
\end{comment}

\begin{proposition}\label{prop:condition_for_preserving_compact}
	Let $\catT,\catT'$ be triangulated categories with coproducts and $F\colon \catT\to\catT'$ a triangulated functor. Suppose that $\catT$ is compactly generated by a set $\familyS$ of objects. If $F(C)\in \catT'$ is compact for all $C\in \familyS$, then $F$ preserves compact objects.
\end{proposition}

\begin{proof}
	Let $\catT_0\subseteq\catT$ be the full subcategory spanned by those objects $X$ that $F(X)$ is compact. Then $\catT_0$ is a thick subcategory containing $\familyS$. Hence \cref{prop:Brown_representability_theorem} shows that $\catT_0$ contains $\thick{\familyS}=\catT^\cpt$, which completes the proof.
\end{proof}

\begin{proposition}[{\cite[Lemma 2.10 (2)]{Kuznetsov-Lunts:2015}}]\label{prop:left-adj_preserves_compact_iff_right-adj_preserves_coproduct}
	Let $\catT,\catT'$ be triangulated categories with coproducts and $F\colon \catT\to\catT'$ a triangulated functor. Suppose that $\catT$ is compactly generated by a set $\familyS$ of objects, and that $F$ has a right adjoint $G$. Then $F$ preserves compact objects if and only if $G$ preserves coproducts.
\end{proposition}

\section{Basic dg category theory}\label{section:Basic_dg_category_theory}

Let $\basek$ be a commutative base ring with unit.
This section collects some definitions and basic results concerning dg categories.
The reader is referred to \cite{Keller:1994Deriving}, \cite{Keller:2006arXiv}, \cite{Drinfeld:2004dg_quotients}, and \cite{Toen:2011lectures} for dg category theory and \cite{Kelly:1982Basic} for general enriched category theory. Also \cite[Section 3]{Kuznetsov-Lunts:2015} is helpful.

\subsection{Preliminaries on dg categories}\label{subsection:preliminaries_on_dg_categories}

Let $\Chk$ denote the category of cochain complexes of $\basek$-modules; we use cohomological notation. The category $\Chk$ is a symmetric monoidal closed category via the tensor product $\otimes_\basek$ of complexes and the Hom complex $\Homcpx$, and furthermore has the following model category structure which is compatible with the monoidal structure.

\begin{proposition}[The projective model structure on $\Chk$; see {\cite[\S 2.3]{Hovey:1999}}]\label{thm:projective_model_structure_on_Ch(k)}
	The category $\Chk$ of cochain complexes of $\basek$-modules has a model category structure whose weak equivalences are quasi-isomorphisms and whose fibrations are surjections,
	%\[\wequ = \Qism \coloneqq \{\text{quasi-isomorphisms}\}, \quad \Fib = \{\text{surjections}\}\]
	which is called the \emph{projective model structure} on $\Chk$. This model category is %cofibrantly generated and 
	monoidal model with respect to the tensor product of complexes.

	All objects are fibrant.
	Note that when $\basek$ is a field, all objects are cofibrant.
\end{proposition}

A \emph{dg category} over $\basek$ is just a $\Chk$-enriched category. In this paper, all dg categories are assumed to be over $\basek$.
All cochain complexes of $\basek$-modules form the dg category $\Comdgk$ whose Hom objects are given by the Hom complex $\Homcpx$. We identify the base ring $\basek$ with the dg category with one single object and with the unique Hom complex $\basek$.
We call a dg category $\catA$ \emph{small} when the collection $\obj(\catA)$ of objects is a small set. Let $\dgCat=\dgCatk$ denote the category of small dg categories over $\basek$ (we often omit the subscript $\basek$).

We can associate two ordinary categories with a dg category.
For a dg category $\catA$, let $Z^0(\catA)$ be the \emph{underlying category} (or the \emph{$0$-th cocycle category}) with the same objects whose Hom sets are $Z^0(\catA(A,B))$, and $H^0(\catA)$ be the \emph{homotopy category} (or the \emph{$0$-th cohomology category}) with the same objects whose Hom sets are $H^0(\catA(A,B))$. In the case of the dg category $\Comdgk$ of complexes of $\basek$-modules, its underlying category $Z^0(\Comdgk)$ is just the enriching category $\Chk$.

\begin{definition}\label{def:quasi-equivalence}
	Let $F\colon \catA\to\catB$ be a dg functor between dg categories.
	\begin{enumerate}
		\item $F$ is \emph{quasi-fully faithful} if for all $A,A'\in \catA$, the morphism $F_{AA'}\colon \catA(A,A')\to \catB(FA,FA')$ of complexes is a quasi-isomorphism.
		\item $F$ is \emph{quasi-essentially surjective} if the functor $H^0(F)\colon H^0(\catA)\to H^0(\catB)$ between homotopy categories is essentially surjective.
		      %すなわち任意の対象$B\in\catB$に対し$H^0(\catB)$において$B$が$FA$と同型になるような対象$A\in\catA$が存在するときをいう．
		\item $F$ is a \emph{quasi-equivalence} if it is quasi-fully faithful and quasi-essentially surjective.
	\end{enumerate}
\end{definition}

For a quasi-equivalence $F$, the induced functor $H^0(F)$ is an equivalence of categories.

\begin{comment}

\begin{definition}
	A dg functor $F\colon \catA\to\catB$ is called a \emph{quasi-fibration} if the following two conditions are satisfied:
	\begin{enumerate}[label=\conditem]
		\item (locally fibration) For all $A,A'\in \catA$, the morphism $F_{AA'}\colon \catA(A,A')\to \catB(FA,FA')$ of complexes is a fibration with respect to the projective model structure on $\Chk$, i.e.\ a surjection.
		\item (isofibration) The induced functor $H^0(F)\colon H^0(\catA)\to H^0(\catB)$ is an isofibration: that is, for all $A'\in H^0(\catA)$ and all isomorphisms $v\colon B\xrightarrow{\cong} FA'$ of $H^0(\catB)$, there exists an isomorphism  $u\colon A\xrightarrow{\cong} A'$ of $H^0(\catA)$ such that $H^0(F)(u)=v$.
	\end{enumerate}
\end{definition}

\end{comment}

\begin{theorem}[The Tabuada model structure on $\dgCat$ \cite{Tabuada:2005quillen_model_dgCat}]\label{thm:projective_model_structure_on_dgCatk}
	The category $\dgCat$ of small dg categories (over $\basek$) has a model category structure whose weak equivalences are quasi-equivalences. Let $\HodgCat$ denote the homotopy category of $\dgCat$.
	% and whose fibrations are quasi-fibrations.
	%\[\wequ = \{\text{擬同値}\}, \quad \Fib = \{\text{quasi-fibration}\}\]
	%This model category is cofibrantly generated.
	%すべてのdg圏はfibrantである．
\end{theorem}

\begin{remark}
	The category $\dgCatk$ of small dg categories becomes symmetric monoidal closed via the tensor product $\otimes$ of $\Chk$-enriched categories and the dg functor category $\Fun_\dg(\mplaceholder,\mplaceholder)$. However, it is not a monoidal model category because the tensor products of cofibrant dg categories are not necessarily cofibrant (\cite[p.~631]{Toen:2007}).
\end{remark}

\begin{proposition}\label{prop:cofibrant_replacement_of_dg_cat}
	In the Tabuada model category $\dgCatk$, the following hold.
	\begin{enumerate}
		\item For all dg categories $\catA$, there exists an identity-on-object quasi-equivalence $\cofresol_\catA \colon \cofQ(\catA)\to\catA$ with $\cofQ(\catA)$ a cofibrant dg category. We call it a \emph{cofibrant resolution} of $\catA$.
		\item For a cofibrant dg category $\catA$, each Hom complex $\catA(A,B)$ is cofibrant with respect to the projective model structure on $\Chk$.
	\end{enumerate}
\end{proposition}
%\item When $\basek$ is a field, all dg categories are cofibrant. \mymemo{（reference？）}
% Is this a mistake?

\begin{proof}
	See \cite[Proposition 2.3]{Toen:2007} or the excellent thesis~\cite[Lemma 3.6, 3.7]{Belmans:2013master}.
\end{proof}

Recall that a complex $X$ of $\basek$-modules is defined to be \emph{h-projective} if $\Homcpx(X,N)$ is acyclic for every acyclic complex $N\in \Chk$. (We will soon generalize the notion of h-projectiveness to dg modules.)

Since cofibrant objects in $\Chk$ are h-projective% to be added with suitable reference
, every cofibrant dg category has h-projective Hom complexes by \cref{prop:cofibrant_replacement_of_dg_cat} above. Such a dg category will be called \emph{locally h-projective}. Note that those dg categories whose Hom complexes are h-projective are often called h-projective simply, but we do not choose this terminology here.
We also remark that, when $\basek$ is a field, every dg category is locally h-projective.

\subsection{Dg modules}\label{subsection:dg_bimodules}

For a small dg category $\catA$, a \emph{(right) dg $\catA$-module} is a dg functor $M\colon \catA^\op\to \Comdgk$. Let $\Comdg(\catA)=\Fun_\dg(\catA^\op,\Comdgk)$ denote the dg category of dg $\catA$-modules. Put $\Com(\catA)=Z^0(\Comdg(\catA))$ and $\Kom(\catA)=H^0(\Comdg(\catA))$. A right dg $\catA^\op$-module, or a dg functor $\catA\to\Comdgk$, is called a \emph{left dg $\catA$-module}. In the following, a right dg module is referred to as just a dg module unless otherwise stated.

For a dg $\catA$-module $M$ and an integer $n \in \Z$, we define the \emph{shift} $M[n]$ as the dg module given by $M[n](A)= M(A)[n]$ for all $A\in \catA$. Here $M(A)[n]$ denotes the usual shift of the complex $M(A)$. For a morphism $\theta\colon M\to N$ in $\Com(\catA)=Z^0(\Comdg(\catA))$, we define the \emph{cone} $\Cone(\theta)$ by $\Cone(\theta)(A) = \Cone(\theta_A)$, where $\Cone(\theta_A)$ is the usual mapping cone of the morphism $\theta_A\colon M(A)\to N(A)$ of complexes.

Given a morphism $\theta\colon M\to N$ in $\Com(\catA)$, there is a canonical triangle $M \to N \to \Cone(\theta) \to M[1]$. With distinguished triangles as those isomorphic to canonical ones, the homotopy category $\Kom(\catA)=H^0(\Comdg(\catA))$ of dg modules is a triangulated category.
Since we have the dg Yoneda embedding $\yoneda\colon \catA \hookrightarrow \Comdg(\catA)$, any homotopy category $H^0(\catA)$ can be embedded to the triangulated category $\Kom(\catA)$.

\begin{definition}
	Let $\catA$ be a small dg category and $M\in\Comdg(\catA)$ a dg $\catA$-module.
	\begin{enumerate}
		\item $M$ is called \emph{acyclic} if for all $A\in\catA$, the complex $N(A)\in\Chk$ is acyclic. Let $\Acyc(\catA)\subseteq \Comdg(\catA)$ be the full subcategory consisting of acyclic dg modules.
		\item $M$ is called \emph{h-projective} if for all acyclic dg modules $N\in \Acyc(\catA)$, the complex $\Comdg(\catA)(M,N)$ is also acyclic; or equivalently if $H^0(\Comdg(\catA)(M,N))=\Kom(\catA)(M,N)\cong 0$ for all acyclic dg modules $N\in \Acyc(\catA)$. Let $\hproj(\catA)\subseteq \Comdg(\catA)$ be the full subcategory consisting of h-projective dg modules.
		\item $M$ is called \emph{h-injective} if for all acyclic dg modules $N\in \Acyc(\catA)$, the complex $\Comdg(\catA)(N,M)$ is also acyclic; or equivalently if $H^0(\Comdg(\catA)(N,M))=\Kom(\catA)(N,M)\cong 0$ for all acyclic dg modules $N\in \Acyc(\catA)$. Let $\hinj(\catA)\subseteq \Comdg(\catA)$ be the full subcategory consisting of h-injective dg modules.
	\end{enumerate}
\end{definition}

The dg Yoneda lemma implies that the representable dg functors $\catA(\mplaceholder,A)\colon \catA^\op\to \Comdgk$ are h-projective dg $\catA$-modules. Therefore the Yoneda embedding $\yoneda \colon \catA \to \Comdg(\catA)$ factors through $\hproj(\catA)$.

%dg圏$\catA$に対して$\Comdg(\catA)$も$\hproj(\catA)$もpretriangulatedなdg圏で，そのホモトピー圏はすべての余積を持つ．

A morphism $\theta\colon M\to N$ of dg modules over a small dg category $\catA$, or a morphism in $\Com(\catA)=Z^0(\Comdg(\catA))$, is called a \emph{quasi-isomorphism} if for each $A\in\catA$, the morphism $\theta_A\colon M(A)\to N(A)$ in $\Chk$ is a quasi-isomorphism. Note that a morphism $\theta\colon M\to N$ of dg modules is a quasi-isomorphism if and only if its cone $\Cone(\theta)$ is acyclic.

\begin{lemma}[{\cite[Section\ 3]{Keller:1994Deriving}}]\label{lem:existence_of_h-proj_resolution}
	Let $M\in\Comdg(\catA)$ be a dg module over a small dg category $\catA$.
	\begin{enumerate}
		\item $M$ admits an h-projective resolution: that is, there exists a quasi-isomorphism $\hp(M)\to M$ in $\Com(\catA)$ with $\hp(M)$ h-projective.
		\item $M$ admits an h-injective resolution: that is, there exists a quasi-isomorphism $M\to \hi(M)$ in $\Com(\catA)$ with $\hi(M)$ h-injective.
	\end{enumerate}
\end{lemma}

We define the \emph{derived category} $\Dom(\catA)$ of a small dg category $\catA$ as the localization of $\Kom(\catA)$ with respect to quasi-isomorphisms. This is just the Verdier quotient of $\Kom(\catA)$ by the triangulated subcategory $H^0(\Acyc(\catA))$ of acyclic dg modules, and we see from \cref{lem:existence_of_h-proj_resolution} that there are equivalences of triangulated categories
\[ \Dom(\catA) \simeq \Kom(\catA)/H^0(\Acyc(\catA)) \simeq H^0(\hproj(\catA)) \simeq H^0(\hinj(\catA)), \]
and inclusions $H^0(\catA) \subseteq \Dom(\catA) \subseteq \Kom(\catA)$.
Moreover the localization functor $\localization\colon\Kom(\catA)\to\Dom(\catA)$ has a left adjoint $\hp\colon \Dom(\catA)\to\Kom(\catA)$ and a right adjoint $\hi\colon \Dom(\catA)\to\Kom(\catA)$.

\begin{proposition}\label{prop:Dom(catA)_is_compactly_generated}
	For a small dg category $\catA$, the derived category $\Dom(\catA)$ is a triangulated category compactly generated by the set $\{\catA(\mplaceholder,A)\}_{A\in\catA}$.
\end{proposition}

We refer to compact objects in $\Dom(\catA)$ as \emph{perfect} dg modules. It follows from \cref{prop:Brown_representability_theorem} that $\Perf(\catA)\coloneqq \Dom(\catA)^\cpt=\thick{\{\catA(\mplaceholder,A)\mid A\in\catA\}}$.

We call a dg category $\catA$ \emph{pretriangulated} if $H^0(\catA)$ is closed (up to isomorphism) under shifts and cones of $\Dom(\catA)$. The homotopy category of a pretriangulated dg category has the structure of a triangulated category induced by that of $\Dom(\catA)$.

%% tensor product of dg modules

Let $\catA$ be a small dg category and $M\in \Comdg(\catA)$ be a (right) dg $\catA$-module. For a left dg $\catA$-module $N\in \Comdg(\catA^\op)$, the \emph{tensor product $M\otimes_\catA N$ over $\catA$} is defined to be the coequalizer of the parallel morphisms
\[\begin{tikzcd}
		\coprod_{A,B \in \catA} M(A)\otimes_\basek \catA(B,A) \otimes_\basek N(B) \arrow[shift left=0.7ex]{r}{} \arrow[shift right=0.7ex]{r}[swap]{} & \coprod_{A\in \catA} M(A)\otimes_\basek N(A)
	\end{tikzcd}\]
in $\Chk$ (This is the same as the coend of $M(-)\otimes_\basek N(-)$; see \cref{exa:tensor_product_is_coend}).
The tensor product satisfies
\begin{equation}
	\Homcpx(M\otimes_\catA N, X) \cong \Comdg(\catA)(M,\Homcpx(N(\mplaceholder),X)) \label{eqn:example_of_tensor_Hom_adjunction}
\end{equation}
for all $X\in \Chk$, which we check later in \cref{exa:example_of_tensor_Hom_adjunction}. The tensor product over $\catA$ forms a dg functor
\[ \mplaceholder\otimes_\catA \mplaceholder\colon \Comdg(\catA)\otimes \Comdg(\catA^\op) \to \Comdgk. \]

\begin{definition}
	A dg module $M\in \Comdg(\catA)$ is called \emph{h-flat} if for all acyclic left dg modules $N\in \Comdg(\catA^\op)$, the tensor product $M\otimes_\catA N$ is also acyclic.
\end{definition}

\begin{lemma}\label{lem:homology_of_dual_complex}
	Let $E\in \Modk$ be an injective $\basek$-module. We identify $E$ as the complex $E\in \Chk$ concentrated in degree $0$.
	\begin{enumerate}
		\item For a complex $X\in \Chk$, we have $H^n(\Homcpx(X,E)) \cong \Hom_\basek(H^n(X),E)$.
		\item Suppose that $E\in\Modk$ is an injective cogenerator, i.e.\ that for every $\basek$-module $W\in \Modk$, $W=0$ if (and only if) $\Hom_\basek(W,E)=0$. Then a complex $X$ is acyclic if and only if the Hom complex $\Homcpx(X,E)$ is acyclic.
	\end{enumerate}
\end{lemma}

\begin{proof}
	(1) Since $E$ is a complex concentrated in degree $0$, $\Homcpx(X,E)$ is obtained by applying the functor $\Hom_\basek(\mplaceholder,E)$ to the each degree term of $X$. If $E$ is injective, then $\Hom_\basek(\mplaceholder,E)$ is exact, and so commutes with taking $n$-th cohomology.

	(2) This follows from (1).
\end{proof}

\begin{proposition}\label{prop:h-projective_implies_h-flat}
	Every h-projective dg module $M\in \Comdg(\catA)$ is h-flat.
\end{proposition}

\begin{proof}
	Take an injective cogenerator $E$ of $\Modk$. Remark that
	\[ \Homcpx(M\otimes_\catA N, E) \cong \Comdg(\catA)(M,\Homcpx(N(\mplaceholder),E)) \]
	for a left dg module $N\in \Comdg(\catA^\op)$. If $N$ is acyclic, then so is $\Homcpx(N(\mplaceholder),E)$ by \cref{lem:homology_of_dual_complex}. H-projectiveness of $M$ implies that $\Comdg(\catA)(M,\Homcpx(N(\mplaceholder),E))$, and hence $\Homcpx(M\otimes_\catA N, E)$, is also acyclic. Therefore we see from \cref{lem:homology_of_dual_complex} that $M\otimes_\catA N$ is acylic, which shows $M$ is h-flat.
\end{proof}

\subsection{(Co)end calculus}\label{subsection:(co)end_calculus}

Before proceeding to the next section, we now introduce the notion of (co)ends and collect some facts on (co)ends.
(Co)ends are a kind of (co)limits, and provide helpful tools to compute homs and tensors of dg modules.
We refer the reader to Loregian's book~\cite{Loregian:2021coend_calculus} and Kelly's book~\cite{Kelly:1982Basic} for details.
For the dg case, \cite[Section 3]{Genovese:2017} is also a good reference.

Here we define (co)ends of dg functors valued in the dg category $\Comdgk$ of complexes, because only such (co)ends will appear in this paper.

\begin{definition}
	Let $\catA$ be a small dg category and $T\colon \catA^\op\otimes \catA \to \Comdgk$ be a dg functor. An \emph{end} $\int_{A\in\catA} T(A,A)$ of $T$ is defined as the equalizer of
	\[\begin{tikzcd}
			\prod_{A\in \catA} T(A,A) \arrow[shift left=0.7ex]{r}{\rho} \arrow[shift right=0.7ex]{r}[swap]{\sigma} & \prod_{A,B \in \catA} {\Homcpx}\big(\catA(A,B),T(A,B)\big),
		\end{tikzcd}\]
	where $\rho$ and $\sigma$ are induced by the morphisms
	\begin{align*}
		 & \rho_{AB} \colon T(A,A) \to {\Homcpx}\big(\catA(A,B),T(A,B)\big), \quad x \mapsto (a \mapsto T(A,a)(x)),   \\
		 & \sigma_{AB} \colon T(B,B) \to {\Homcpx}\big(\catA(A,B),T(A,B)\big), \quad x \mapsto (a \mapsto T(a,B)(x)).
	\end{align*}

	Dually, a \emph{coend} $\int^{A\in\catA} T(A,A)$ of $T$ is defined as the coequalizer of
	\[\begin{tikzcd}
			\coprod_{A,B \in \catA} \catA(B,A) \otimes_\basek T(A,B) \arrow[shift left=0.7ex]{r}{\rho'} \arrow[shift right=0.7ex]{r}[swap]{\sigma'} & \coprod_{A\in \catA} T(A,A),
		\end{tikzcd}\]
	where $\rho'$ and $\sigma'$ are induced by the morphisms
	\begin{align*}
		 & \rho'_{AB} \colon \catA(B,A)\otimes_\basek T(A,B) \to T(A,A),\quad a\otimes x \mapsto T(A,a)(x)       \\
		 & \sigma'_{AB} \colon \catA(B,A) \otimes_\basek T(A,B) \to T(B,B), \quad a \otimes x \mapsto T(a,B)(x).
	\end{align*}
\end{definition}

\begin{example}\label{exa:tensor_product_is_coend}
	Let $M\in \Comdg(\catA)$ be a right dg $\catA$-module and $N\in \Comdg(\catA^\op)$ be a left dg $\catA$-module.
	The tensor product $M \otimes_\catA N$ over $\catA$, introduced in the previous subsection, is the coend
	\[ M \otimes_\catA N = \int^{A\in \catA} M(A)\otimes_\basek N(A). \]
\end{example}

\begin{proposition}\label{prop:Homcpx_of_functor_cat_is_end}
	Let $\catA$ be a small dg category and $\Fun_\dg(\catA,\Comdgk)$ denote the dg functor category from $\catA$ to $\Comdgk$. For dg functors $F,G\colon \catA\to\Comdgk$, the Hom complex $\Fun_\dg(\catA,\Comdgk)(F,G)$ is isomorphic to the end of $\Homcpx(F(-),G(-))$:
	\[ \Fun_\dg(\catA,\Comdgk)(F,G) \cong \int_{A} \Homcpx(F(A),G(A)). \]
\end{proposition}

\begin{proposition}\label{prop:Homcpx_preserves_ends}
	Let $\catA$ be a small dg category and $T\colon \catA^\op\otimes \catA \to \Comdgk$ be a dg functor. For a complex $X\in \Comdgk$, we have natural isomorphisms
	\begin{align*}
		{\Homcpx}\mleft( X, \int_A T(A,A) \mright) & \cong \int_A \Homcpx(X,T(A,A)),      \\
		{\Homcpx}\mleft( \int^A T(A,A),X \mright)  & \cong \int_A \Homcpx(T(A,A),X),      \\
		X \otimes_\basek \int^A T(A,A)             & \cong \int^A X\otimes_\basek T(A,A).
	\end{align*}
\end{proposition}

\begin{proposition}[(co)Yoneda lemma]\label{prop:(co)Yoneda_lemma}
	Let $\catA$ be a small dg category.
	\begin{enumerate}
		\item For a dg functor $F\colon \catA \to\Comdgk$, we have natural isomorphisms
		      \[ F(B) \cong \int_A \Homcpx(\catA(B,A),F(A)) \cong \int^A \catA(A,B) \otimes_\basek F(A) \]
		      for all $B\in\catA$.
		\item For a dg functor $G\colon \catA^\op \to\Comdgk$, we have natural isomorphisms
		      \[ G(B) \cong \int_A \Homcpx(\catA(A,B),G(A)) \cong \int^A \catA(B,A) \otimes_\basek G(A) \]
		      for all $B\in\catA$.
	\end{enumerate}
\end{proposition}

\begin{proposition}[Fubini theorem]\label{prop:Fubini_theorem_for_end}
	Let $\catA,\catB$ be small dg categories and $T\colon \catA^\op\otimes\catB^\op\otimes \catA\otimes\catB \to \Comdgk$ be a dg functor. Then we have
	\[ \int_{(A,B)\in \catA\otimes\catB} T(A,B,A,B) \cong \int_{A\in \catA} \int_{B \in \catB} T(A,B,A,B) \cong \int_{B \in \catB} \int_{A\in \catA} T(A,B,A,B), \]
	\[ \int^{(A,B)\in \catA\otimes\catB} T(A,B,A,B) \cong \int^{A\in \catA} \int^{B \in \catB} T(A,B,A,B) \cong \int^{B \in \catB} \int^{A\in \catA} T(A,B,A,B). \]
\end{proposition}

\begin{example}\label{exa:example_of_tensor_Hom_adjunction}
	Let $M\in \Comdg(\catA)$ be a right dg $\catA$-module, $N\in \Comdg(\catA^\op)$ be a left dg $\catA$-module, and $X\in \Chk$ be a complex. Then we can prove \eqref{eqn:example_of_tensor_Hom_adjunction} by (co)end calculus, as follows:
	\begin{alignat*}{2}
		\Homcpx(M\otimes_\catA N, X)
		 & = \Homcpx\left( \int^{A\in \catA} M(A)\otimes_\basek N(A),X \right) & \quad & \text{by \cref{exa:tensor_product_is_coend},}       \\
		 & \cong \int_{A\in \catA}\Homcpx(M(A)\otimes_\basek N(A),X)           &       & \text{by \cref{prop:Homcpx_preserves_ends},}        \\
		 & \cong \int_{A\in \catA}\Homcpx(M(A), \Homcpx(N(A),X))               &       & \text{by the adjunction,}                           \\
		 & \cong \Comdg(\catA)(M,\Homcpx(N(\mplaceholder),X))                  &       & \text{by \cref{prop:Homcpx_of_functor_cat_is_end}.}
	\end{alignat*}
\end{example}

%\cref{prop:Homcpx_of_functor_cat_is_end}
%\cref{prop:Homcpx_preserves_ends}
%\cref{prop:(co)Yoneda_lemma}
%\cref{prop:Fubini_theorem_for_end}

\section{Dg bimodules and their bicategorical structures}\label{section:dg_bimodules_and_their_bicategorical_structures}

\subsection{Bicategories of dg bimodules}\label{subsection:the_bicategory_of_dg_bimodules}

In this subsection, we construct two bicategories $\Bimod$ and $\DBimod$ whose 1-morphisms are dg bimodules.
For details on bicategories, we refer to \cite{Benabou:1967bicategories} and \cite{Johnson-Yau:2021}.

Let $\catA$, $\catB$ be small dg categories. A \emph{dg bimodule} (or a \emph{dg profunctor}) $X\colon \catA\slashedrightarrow\catB$ from $\catA$ to $\catB$ is a dg functor $X\colon \catB^\op\otimes\catA \to \Comdgk$. This is the same as a right dg $(\catB\otimes\catA^\op)$-module. Let $\Comdg(\catA,\catB)\coloneqq \Comdg(\catB\otimes\catA^\op)=\Fun_\dg(\catB^\op\otimes\catA,\Comdgk)$ denote the dg category of dg bimodules from $\catA$ to $\catB$, and put
\[ \Com(\catA,\catB)=\Com(\catA^\op\otimes\catB),\quad \Kom(\catA,\catB)=\Kom(\catA^\op\otimes\catB),\quad \Dom(\catA,\catB)=\Dom(\catA^\op\otimes\catB). \]
Regarding the base ring $\basek$ as a dg category with a single object, we have $\basek^\op\otimes\catA\cong\catA$ and $\catA^\op\otimes\basek\cong \catA^\op$. Hence, a dg bimodule $M\colon \basek\slashedrightarrow\catA$ from $\basek$ to $\catA$ is precisely a right dg $\catA$-module $M\colon \catA^\op\to\Comdgk$ and a dg bimodule $N\colon \catA\slashedrightarrow\basek$ from $\catA$ to $\basek$ is precisely a left dg $\catA$-module $N\colon \catA\to\Comdgk$.

A dg bimodule $X\colon\catA\slashedrightarrow\catB$ is called \emph{h-projective} if it is h-projective as a right dg $(\catB\otimes\catA^\op)$-module. On the other hand, one can get from a dg bimodule $X\colon \catA\slashedrightarrow\catB$ a right dg $\catB$-module $X(\mplaceholder,A)$ for $A\in \catA$ and a left dg $\catA$-module $X(B,\mplaceholder)$ for $B\in \catB$. We call $X$ \emph{right h-projective} if each $X(\mplaceholder,A)$ is h-projective as a right dg $\catB$-module, and \emph{left h-projective} if each $X(B,\mplaceholder)$ is h-projective as a left dg $\catA$-module. Similarly we define h-injective (h-flat, compact, etc.) dg bimodules and right or left h-injective (h-flat, compact, etc.) ones in the same way.

For two dg bimodules $X\colon \catA\slashedrightarrow\catB$ and $Y\colon \catB \slashedrightarrow \catC$, their \emph{tensor product over $\catB$} is defined to be the dg bimodule $X \otimes_\catB Y\colon \catC^\op\otimes\catA \to\Comdgk$ given by
\[ (X \otimes_\catB Y)(C,A) = X(\mplaceholder,A) \otimes_\catB Y(C,\mplaceholder) = \int^{B\in\catB} X(B,A) \otimes_\basek Y(C,B). \]
Defining the composition $\procomp$ of dg bimodules by $Y\procomp X\coloneqq X \otimes_\catB Y$ (Notice the order!), we obtain a bicategory $\Bimod$ whose objects are all small dg categories and whose Hom categories $\Bimod(\catA,\catB)$ are the categories $\Com(\catA,\catB)$ of dg bimodules.
Its identities $I_\catA\colon \catA\slashedrightarrow\catA$ are the dg Hom functors $\catA(\mplaceholder,\mplaceholder)\colon \catA^\op\otimes\catA\to \Comdgk$.

%The bicategory $\Bimod$ is closed
Let us introduce other constructions of dg bimodules. For dg bimodules $X\colon \catA \slashedrightarrow \catB$ and $Z\colon \catA\slashedrightarrow\catC$, the dg bimodule $X^\ddag Z\colon \catB\slashedrightarrow\catC$ is defined by setting
\[ X^\ddag Z(C,B) = \Fun_\dg(\catA,\Comdgk)(X(B,\mplaceholder), Z(C,\mplaceholder)) = \int_{A\in\catA} {\Homcpx}\big(X(B,A),Z(C,A)\big). \]
We also define the bimodule $Y_\ddag Z\colon \catA\slashedrightarrow\catB$ given by
\[ Y_\ddag Z (B,A) = \Fun_\dg(\catC^\op,\Comdgk)(Y(\mplaceholder,B),Z(\mplaceholder,A)) = \int_{C\in\catC} {\Homcpx}\big(Y(C,B),Z(C,A)\big) \]
for dg bimodules $Y\colon \catB\slashedrightarrow\catC$ and $Z\colon \catA\slashedrightarrow\catC$.
The following is well-known, which we can check by (co)end calculus.

\begin{proposition}\label{prop:Bimod_is_closed_bicategory}
	Let $X\colon \catA\slashedrightarrow\catB$, $Y\colon \catB\slashedrightarrow\catC$, and $Z\colon \catA\slashedrightarrow\catC$ be dg bimodules between small dg categories. Then we have natural isomorphisms of complexes
	\begin{align*}
		\Comdg(\catA,\catC)(Y\odot X, Z) & \cong \Comdg(\catB,\catC)(Y,X^\ddag Z), \\
		\Comdg(\catA,\catC)(Y\odot X,Z)  & \cong \Comdg(\catA,\catB)(X,Y_\ddag Z).
	\end{align*}
	This means that $X^\ddag Z$ forms the right Kan extension $\Ran_X Z$ and $Y_\ddag Z$ forms the right Kan lifting $\Rift_Y Z$ in $\Bimod$; in other words, $\Bimod$ is a closed bicategory.
\end{proposition}

The composition $\procomp$ of dg bimodules has dg functoriality and gives rise to a dg functor
\[ \procomp\colon \Comdg(\catB,\catC) \otimes \Comdg(\catA,\catB) \to \Comdg(\catA,\catC), \]
which induces the functor
\[ \procomp\colon \Kom(\catB,\catC) \times \Kom(\catA,\catB) \to \Kom(\catA,\catC). \]
Then we define the \emph{derived composition} $\Dprocomp$ by the right Kan extension
\[
	\begin{tikzcd}
		\Kom(\catB,\catC) \times \Kom(\catA,\catB) \arrow{r}{\procomp}[name=B1,below,pos=0.45]{} \arrow{d}[swap]{\localization\times\localization} & \Kom(\catA,\catC) \arrow{d}{\localization} \\
		\Dom(\catB,\catC) \times \Dom(\catA,\catB) \arrow[dashed]{r}{\Dprocomp}[name=A1,above,pos=0.45]{} & \Dom(\catA,\catC)\rlap{.}
		\arrow[Rightarrow,from=A1,to=B1,shorten <=2ex,shorten >=1.5ex," "]
	\end{tikzcd}
\]
Since the localization functor $\localization\colon \Kom(\catA,\catB)\to \Dom(\catA,\catB)$ has a left adjoint $\hp\colon \Dom(\catA,\catB)\to\Kom(\catA,\catB)$, we see that the derived composition $\mplaceholder\Dprocomp\mplaceholder$ is given by the composites
\[ \Dom(\catB,\catC) \times \Dom(\catA,\catB) \xrightarrow{\hp\times \hp} \Kom(\catB,\catC) \times \Kom(\catA,\catB) \xrightarrow{\procomp} \Kom(\catA,\catC) \xrightarrow{\localization} \Dom(\catA,\catC). \]

\begin{lemma}\label{lem:precomp_with_right_h-flat_preserves_qism}
	Consider small dg categories $\catA$, $\catB$, and $\catC$.
	\begin{enumerate}
		\item If a dg bimodule $X\colon \catA\slashedrightarrow\catB$ is right h-flat, then the precomposition
		      \[ \mplaceholder\procomp X = X\otimes_\catB \mplaceholder \colon \Comdg(\catB,\catC) \to \Comdg(\catA,\catC) \]
		      preserves quasi-isomorphisms.
		      %次で示す\cref{lem:h-proj_implies_right_h-proj}により，これは$\catA$がlocally h-projectiveかつ$X$がh-projectiveならば成り立つ．
		\item If a dg bimodule $Y\colon \catB\slashedrightarrow\catC$ is left h-flat, then the postcomposition
		      \[ Y\procomp \mplaceholder = \mplaceholder\otimes_\catB Y\colon \Comdg(\catA,\catB) \to \Comdg(\catA,\catC) \]
		      preserves quasi-isomorphisms.
		      %次で示す\cref{lem:h-proj_implies_right_h-proj}により，これは$\catC$がlocally h-projectiveかつ$Y$がh-projectiveならば成り立つ．
	\end{enumerate}
\end{lemma}

\begin{proof}
	(1) It suffices to show that if $N\in \Comdg(\catB,\catC)$ is an acyclic dg bimodule, then so is $N\procomp X$. Given an acyclic $N$, then $N(C,\mplaceholder)$ is an acyclic left dg $\catB$-module. Therefore the assumption that $X$ is right h-flat implies that
	\[ N\procomp X (C,A) = X(\mplaceholder,A) \otimes_\catB N(C,\mplaceholder) \]
	is acyclic for $A\in \catA$ and $C\in \catC$. The same for (2).
\end{proof}

The following propositions are basic relations between properties of bimodules as two-sided modules and as one-sided ones. Recall that a dg category is \emph{locally h-projective} (respectively, \emph{locally h-flat}) if its Hom complexes are h-projective (respectively, h-flat) in $\Comdgk$.

\begin{lemma}[dg category version of {\cite[Lemma 14.3.11]{Yekutieli:2020Derived_categories}}]\label{lem:h-flat_with_locally_h-flat_domain_is_right_h-flat}
	Consider small dg categories $\catA$, $\catB$, and $\catC$.
	\begin{enumerate}
		\item If $\catA$ is locally h-flat and if a dg bimodule $X\colon \catA\slashedrightarrow\catB$ is h-flat, then $X$ is right h-flat.
		\item If $\catC$ is locally h-flat and if a dg bimodule $Y\colon \catB\slashedrightarrow\catC$ is h-flat, then $Y$ is left h-flat.
	\end{enumerate}
\end{lemma}

\begin{proof}
	(1) We want to show that the right dg $\catB$-module $X(\mplaceholder,A)$ is h-flat for every $A\in \catA$. Given a left dg $\catB$-module $N\colon \catB\to \Comdgk$, we have isomorphisms
	\begin{alignat*}{2}
		X(\mplaceholder,A) \otimes_\catB N
		 & = \int^{B' \in \catB} X(B',A)\otimes N(B')                                                         & \quad & \text{(Definition)}                         \\
		 & \cong \int^{A'\in \catA} \catA(A',A)\otimes \int^{B' \in \catB} X(B',A')\otimes N(B')              &       & \text{(\cref{prop:(co)Yoneda_lemma})}       \\
		 & \cong \int^{A'\in \catA} \int^{B' \in \catB} \catA(A',A)\otimes X(B',A')\otimes N(B')              &       & \text{(\cref{prop:Homcpx_preserves_ends})}  \\
		 & \cong \int^{(B',A')\in \catB\otimes\catA^\op} X(B',A')\otimes \big(N(B')\otimes\catA(A',A) \big)   &       & \text{(\cref{prop:Fubini_theorem_for_end})} \\
		 & \cong X \otimes_{\catB\otimes\catA^\op} \big( N(\mplaceholder)\otimes\catA(\mplaceholder,A) \big).
	\end{alignat*}
	Suppose $N$ is acyclic. Then $N(\mplaceholder)\otimes\catA(\mplaceholder,A)$ is acyclic by the locally h-flatness of $\catA$. Since $X$ is h-projective as a dg ($\catB\otimes\catA^\op$)-module, the above isomorphisms imply that $X(\mplaceholder,A) \otimes_\catB N$ is acyclic as well. Thus $X$ is right h-flat.

	(2) It follows in a similar manner, because there is an isomorphism
	\[ N \otimes_\catB Y(C,\mplaceholder) \cong Y \otimes_{\catC\otimes\catB^\op} \big(\catC(C,\mplaceholder)\otimes N(\mplaceholder)\big) \]
	for a right dg $\catB$-module $N\colon \catB^\op\to \Comdgk$ and an object $C\in \catC$.
\end{proof}

\begin{comment}
上の(2)の証明：
各$C\in \catC$に対して，left dg $\catB$-module $Y(C,\mplaceholder)$がh-flatであることを示そう．right dg $\catB$-module $N\colon \catB^\op\to \Comdgk$に対して，
\begin{align*}
	N \otimes_\catB Y(C,\mplaceholder)
	 & \cong \int^{B' \in \catB} N(B') \otimes Y(C,B')                                                   \\
	 & \cong \int^{C'\in \catC} \catC(C,C')\otimes \int^{B' \in \catB} N(B') \otimes Y(C',B')            \\
	 & \cong \int^{C'\in \catC} \int^{B' \in \catB} \catC(C,C')\otimes N(B') \otimes Y(C',B')            \\
	 & \cong \int^{(C',B')\in \catC\otimes\catB^\op} Y(C',B') \otimes \big(\catC(C,C')\otimes N(B')\big) \\
	 & \cong Y \otimes_{\catC\otimes\catB^\op} \big(\catC(C,\mplaceholder)\otimes N(\mplaceholder)\big)
\end{align*}
である．$N$がacyclicのとき，$\catC$がlocally h-flatであることから$\catC(C,\mplaceholder)\otimes N(\mplaceholder)$もacyclicで，よって$Y$がh-projective dg ($\catC\otimes\catB^\op$)-moduleより$N \otimes_\catB Y(C,\mplaceholder)$もacyclicとなる．よって$Y$はleft h-flatである．
\end{comment}

\begin{lemma}[dg category version of {\cite[Lemma 14.3.11]{Yekutieli:2020Derived_categories}}]
	Consider small dg categories $\catA$, $\catB$, and $\catC$.%\mymemo{（使ってないからいらない？）}
	\begin{enumerate}
		\item If $\catA$ is locally h-flat and if a dg bimodule $X\colon \catA\slashedrightarrow\catB$ is h-injective, then $X$ is right h-injective.
		\item If $\catC$ is locally h-flat and if a dg bimodule $Y\colon \catB\slashedrightarrow\catC$ is h-injective, then $Y$ is left h-injective.
	\end{enumerate}
\end{lemma}

\begin{proof}
	(1) We want to show that the right dg $\catB$-module $X(\mplaceholder,A)$ is h-injective for $A\in \catA$. Given a right dg $\catB$-module $N\colon \catB^\op\to \Comdgk$, we have
	\begin{alignat*}{2}
		 & \Comdg(\catB)(N,X(\mplaceholder,A))                                                                     & \quad &                                                    \\
		 & \cong \int_{B'\in \catB} \Homcpx(N(B'),X(B',A))                                                         & \quad & \text{(\cref{prop:Homcpx_of_functor_cat_is_end})}  \\
		 & \cong \int_{A'\in\catA} \Homcpx\mleft( \catA(A,A'), \int_{B'\in \catB} \Homcpx(N(B'),X(B',A')) \mright) &       & \text{(\cref{prop:(co)Yoneda_lemma})}              \\
		 & \cong \int_{A'\in\catA}\int_{B'\in \catB} \Homcpx\big( \catA(A,A'), \Homcpx(N(B'),X(B',A')) \big)       &       & \text{(\cref{prop:Homcpx_preserves_ends})}         \\
		 & \cong \int_{A'\in\catA}\int_{B'\in \catB} \Homcpx\left( \catA(A,A')\otimes N(B'), X(B',A') \right)      &       & \text{(adjunction)}                                \\
		 & \cong \int_{(B',A')\in\catB\otimes\catA^\op} \Homcpx\left( N(B')\otimes\catA(A,A'), X(B',A') \right)    &       & \text{(\cref{prop:Fubini_theorem_for_end})}        \\
		 & \cong \Comdg(\catB\otimes\catA^\op)( N(\mplaceholder)\otimes\catA(A,\mplaceholder), X)                  &       & \text{(\cref{prop:Homcpx_of_functor_cat_is_end}).}
	\end{alignat*}
	Suppose $N$ is acyclic. Then $N(\mplaceholder)\otimes\catA(A,\mplaceholder)$ is acyclic by the locally h-flatness of $\catA$. Since $X$ is h-injective as a dg ($\catB\otimes\catA^\op$)-module, $\Comdg(\catB)(N,X(\mplaceholder,A))$ is acyclic as well. Thus $X$ is right h-injective.

	(2) Similarly.
\end{proof}

\begin{lemma}[{\cite[Lemma 3.4]{Canonaco-Stellari:2015Internal_Homs}}]\label{lem:h-proj_implies_right_h-proj}
	% or Genovese:2017 Lemma 5.5
	Let $\catA$ and $\catB$ be small dg categories, and $X\colon \catA\slashedrightarrow\catB$ a dg bimodule.
	\begin{enumerate}
		\item If $\catA$ is locally h-projective and if $X$ is h-projective, then $X$ is right h-projective.
		\item If $\catB$ is locally h-projective and if $X$ is h-projective, then $X$ is left h-projective.
	\end{enumerate}
\end{lemma}

\begin{proof}
	(1) Let us verify that the dg $\catB$-module $X(\mplaceholder,A)\in \Comdg(\catB)$ is h-projective for $A\in \catA$.
	Given a dg bimodule $N \in \Comdg(\catB)$, we have
	\begin{alignat*}{2}
		 & \Comdg(\catB)(X(\mplaceholder,A),N)                                                                       & \quad &                                                    \\
		 & \cong \int_{B'\in\catB} \Homcpx(X(B',A), N(B'))                                                           & \quad & \text{(\cref{prop:Homcpx_of_functor_cat_is_end})}  \\
		 & \cong \int_{A'\in\catA} {\Homcpx}\mleft( \catA(A',A), \int_{B'\in\catB} \Homcpx(X(B',A'), N(B')) \mright) &       & \text{(\cref{prop:(co)Yoneda_lemma})}              \\
		 & \cong \int_{A'\in\catA} \int_{B'\in\catB} {\Homcpx}\big( \catA(A',A), \Homcpx(X(B',A'), N(B')) \big)      &       & \text{(\cref{prop:Homcpx_preserves_ends})}         \\
		 & \cong \int_{A'\in\catA} \int_{B'\in\catB} {\Homcpx}\big(X(B',A'), \Homcpx(\catA(A',A), N(B')) \big)       &       & \text{(adjunction)}                                \\
		 & \cong \int_{(B',A')\in \catB\otimes\catA^\op} {\Homcpx}\big(X(B',A'), \Homcpx(\catA(A',A), N(B')) \big)   &       & \text{(\cref{prop:Fubini_theorem_for_end})}        \\
		 & \cong \Comdg(\catB\otimes\catA^\op)\big(X,\Homcpx(\catA(-,A), N(-))\big)                                  &       & \text{(\cref{prop:Homcpx_of_functor_cat_is_end})}.
	\end{alignat*}
	If $N$ is acyclic, then $\Homcpx(\catA(-,A), N(-))$ is so by the assumption that $\catA$ is locally h-projective. Since $X$ is h-injective as a dg ($\catB\otimes\catA^\op$)-module, the above calculation shows $\Comdg(\catB)(X(\mplaceholder,A),N)$ is acyclic as well. Thus $X$ is right h-projective.

	(2) Similarly.
\end{proof}

\begin{lemma}[{\cite[Lemma 6.1]{Genovese:2017}, \cite[Proposition 2.5]{Anno-Logvinenko:2017spherical}}]\label{lem:tensor_of_h-proj_is_h-proj}
	Let $X\colon \catA\slashedrightarrow\catB$ and $Y\colon \catB\slashedrightarrow\catC$ be dg bimodules.
	\begin{enumerate}
		\item If $X$ is h-projective and $Y$ is right h-projective, then $Y\procomp X$ is h-projective.
		\item If $X$ is left h-projective and $Y$ is h-projective, then $Y\procomp X$ is h-projective.
	\end{enumerate}
	In particular, \cref{lem:h-proj_implies_right_h-proj} shows that if $\catB$ is locally h-projective and if $X$ and $Y$ are h-projective, then $Y\procomp X$ is also h-projective.
\end{lemma}

\begin{proof}
	(1) Given an acyclic dg bimodule $N \in \Comdg(\catA,\catC)$, we have
	\[ \Comdg(\catA,\catC)(Y\procomp X, N) \cong \Comdg(\catA,\catB)(X,Y_\ddag N), \]
	where the dg bimodule $Y_\ddag N$ is defined as
	\begin{align*}
		Y_\ddag N (B,A) & = \Comdg(\catC)(Y(\mplaceholder,B),N(\mplaceholder,A)).
	\end{align*}
	By the right h-projectiveness of $Y$, we can see that $Y_\ddag N (B,A)$ is acyclic. Since $X\in \Comdg(\catA,\catB)$ is h-projective, $\Comdg(\catA,\catC)(Y\procomp X, N)$ turns out to be acyclic. Thus $Y\procomp X$ is h-projective.

	(2) Similarly.
\end{proof}

\begin{comment}
上の(2)の証明：
acyclic dg bimodule $N \in \Comdg(\catA,\catC)$に対して
\[ \Comdg(\catA,\catC)(Y\odot X, N) \cong \Comdg(\catB,\catC)(Y,X^\ddag N) \]
である．ここで
\begin{align*}
	X^\ddag N (C,B) & = \Comdg(\catA^\op)(X(B,\mplaceholder),N(C,\mplaceholder))
\end{align*}
である．
$X$がleft h-projectiveより各$X(B,\mplaceholder)$がh-projective dg $\catA^\op$-moduleで，$N(C,\mplaceholder)$はまたacyclicだから$X^\ddag N (C,B)$もacyclicとなる．よって$Y\in \Comdg(\catB,\catC)$がh-projectiveであることから，$\Comdg(\catA,\catC)(Y\procomp X, N)$がacyclicとわかる．
\end{comment}

Although the derived composition $\Dprocomp$ does not have associativity in general, we can prove the following result.
%これは，dg bimoduleの合成$\procomp$がh-projectiveであることを保たないからである．

\begin{proposition}
	%cf. Yekutieli Proposition 14.3.13
	Let $X\colon \catA\slashedrightarrow\catB$, $Y\colon \catB\slashedrightarrow\catC$, and $Z\colon \catC\slashedrightarrow\catD$ be dg bimodules. If either (1) $\catB$, $\catC$ are locally h-projective or (2) $\catA$, $\catD$ are locally h-flat, then there is a natural isomorphism
	\[ (Z\Dprocomp Y)\Dprocomp X \cong Z\Dprocomp (Y\Dprocomp X) \]
	in $\Dom(\catA,\catD)$.
\end{proposition}

\begin{proof}
	(1) Suppose that $\catB$ and $\catC$ are locally h-projective. Then by \cref{lem:tensor_of_h-proj_is_h-proj} we may assume $\hp(Z\Dprocomp Y)=\hp(Z)\procomp \hp(Y)$ and $\hp(Y\Dprocomp X)=\hp(Y)\procomp \hp(X)$. Thus we have
	\begin{align*}
		(Z\Dprocomp Y)\Dprocomp X & \cong \hp(Z \Dprocomp Y) \procomp \hp(X)              \\
		                          & \cong \big(\hp(Z)\procomp \hp(Y)\big) \procomp \hp(X) \\
		                          & \cong \hp(Z) \procomp \big(\hp(Y)\procomp \hp(X)\big) \\
		                          & \cong \hp(Z) \procomp \hp(Y\Dprocomp X)               \\
		                          & \cong Z\Dprocomp (Y\Dprocomp X).
	\end{align*}

	(2) If $\catA$ and $\catD$ are locally h-flat, then \cref{prop:h-projective_implies_h-flat} and \cref{lem:h-flat_with_locally_h-flat_domain_is_right_h-flat} shows that $\hp(X)$ is right h-flat and $\hp(Z)$ is left h-flat. Thus we have by \cref{lem:precomp_with_right_h-flat_preserves_qism}
	\begin{align*}
		(Z\Dprocomp Y)\Dprocomp X & \cong \hp(Z \Dprocomp Y) \procomp \hp(X)               \\
		                          & \cong (Z \Dprocomp Y) \procomp \hp(X)                  \\
		                          & \cong \big(\hp(Z) \procomp \hp(Y)\big) \procomp \hp(X) \\
		                          & \cong \hp(Z) \procomp \big(\hp(Y)\procomp \hp(X)\big)  \\
		                          & \cong \hp(Z) \procomp (Y\Dprocomp X)                   \\
		                          & \cong \hp(Z) \procomp \hp(Y\Dprocomp X)                \\
		                          & \cong Z\Dprocomp (Y\Dprocomp X).\qedhere
	\end{align*}
\end{proof}

\begin{proposition}
	Let $X\colon\catA\slashedrightarrow\catB$ be dg bimodule.
	\begin{enumerate}
		\item If $\catB$ is locally h-flat, then there is a natural isomorphism $X\Dprocomp I_\catA \cong X$ in $\Dom(\catA,\catB)$.
		\item If $\catA$ is locally h-flat, then there is a natural isomorphism $I_\catB \Dprocomp X \cong X$ in $\Dom(\catA,\catB)$.
	\end{enumerate}
\end{proposition}

\begin{proof}
	(1) If $\catB$ is locally h-flat, then $\hp(X)$ is left h-flat. Hence we have
	\[ X\Dprocomp I_\catA = \hp(X) \procomp \hp(I_\catA) \cong \hp(X) \procomp I_\catA \cong \hp(X) \cong X \]
	by \cref{lem:precomp_with_right_h-flat_preserves_qism}.
	(2) Similarly.
	% (2) $\catA$がlocally h-flatのとき，$\hp(X)$はright h-flatだから
	%\[ I_\catB \Dprocomp X =\hp(I_\catB)\procomp \hp(X)\cong I_\catB \procomp \hp(X) \cong \hp(X) \cong X \] となる．
\end{proof}

According to the above propositions, we can define $\DBimodhf$ as the bicategory such that
\begin{itemize}
	\item the objects are locally h-flat small dg categories;
	\item the Hom categories $\DBimodhf(\catA,\catB)$ are the derived categories $\Dom(\catA,\catB)$ of dg bimodules;
	\item the compositions are the derived compositions.
\end{itemize}
Likewise, we can define the bicategory $\DBimodhp$ whose objects are locally h-projective small dg categories. This bicategory is a full sub-bicategory of $\DBimodhf$.

Recall that every dg category $\catA$ has a locally h-projective resolution $\cofQ(\catA)$.
We can use this assignment to get a new bicategory $\cofQ^*\DBimodhp$ whose objects are all small dg categories and whose Hom categories are $\Dom(\cofQ(\catA),\cofQ(\catB))$, which is the same as the bicategory written as $\catname{DBimod}$ in \cite[p.\ 651]{Genovese:2017}.

In the following, we will mainly take care of the bicategory $\DBimodhp$ of locally h-projective small dg categories, and simply put $\DBimod\coloneqq \DBimodhp$.

\subsection{The tensor and Hom functors associated with dg bimodules}\label{subsection:the_tensor_and_Hom_functors_associated_with_dg_bimodules}

For a dg bimodule $X\colon \catA\slashedrightarrow\catB$ between small dg categories, we have a dg adjunction
\begin{equation*}
	T_X \colon
	\begin{tikzcd}
		\Comdg(\catA) \arrow[shift left=1ex]{r}{} \arrow[phantom]{r}[sloped]{\scriptstyle\perp} &
		\Comdg(\catB) \arrow[shift left=1ex]{l}{}
	\end{tikzcd}
	\colon H_X
\end{equation*}
where
\begin{align*}
	T_X & = X\procomp\mplaceholder = \mplaceholder \otimes_\catA X, \text{ and }                   \\
	H_X & = X_\ddag(\mplaceholder), \text{ i.e. } H_X(N)(A) = \Comdg(\catB)(X(\mplaceholder,A),N).
\end{align*}
The notation is borrowed from Keller's paper~\cite{Keller:1994Deriving}. The assignments $\catA \mapsto \Com(\catA)$ and $X\mapsto T_X$ form a pseudo-functor
\[ T\colon \Bimod \to \CAT, \]
where $\CAT$ denotes the $2$-category of not necessarily small categories.

Applying the $0$-th cohomology category functor $H^0$ to the above dg adjunction, we have an ordinary adjunction
\begin{equation*}
	T_X \colon
	\begin{tikzcd}
		\Kom(\catA) \arrow[shift left=1ex]{r}{} \arrow[phantom]{r}[sloped]{\scriptstyle\perp} &
		\Kom(\catB) \arrow[shift left=1ex]{l}{}
	\end{tikzcd}
	\colon H_X
\end{equation*}
(we use the same notation $T_X,H_X$.). Then we define the left derived functor of $T_X$ as the right Kan extension
\[
	\begin{tikzcd}
		\Kom(\catA) \arrow{d}[swap]{\localization} \arrow{r}{T_X}[name=B1,below,pos=0.5]{} & \Kom(\catB) \arrow{d}{\localization} \\
		\Dom(\catA) \arrow[dashed]{r}{\bbLT_X}[name=A1,above,pos=0.5]{} & \Dom(\catB)\rlap{.}
		\arrow[Rightarrow,from=A1,to=B1,shorten <=2ex,shorten >=1.5ex," "]
	\end{tikzcd}
\]
Since the localization $\localization\colon \Kom(\catA)\to \Dom(\catA)$ has a left adjoint $\hp\colon \Dom(\catA)\to\Kom(\catA)$, the left derived functor $\bbLT_X$ is given by the composites
\[ \Dom(\catA) \xrightarrow{\hp} \Kom(\catA) \xrightarrow{T_X} \Kom(\catB) \xrightarrow{\localization} \Dom(\catB), \]
and becomes an absolute Kan extension. Also we define the right derived functor of $H_X$ as the left Kan extension
\[
	\begin{tikzcd}
		\Kom(\catA) \arrow{d}[swap]{\localization} & \Kom(\catB) \arrow{d}{\localization} \arrow{l}[swap]{H_X}[name=A1,above,pos=0.5]{} \\
		\Dom(\catA) & \Dom(\catB)\rlap{.} \arrow[dashed]{l}[swap]{\bbRH_X}[name=B1,below,pos=0.5]{}
		\arrow[Rightarrow,from=A1,to=B1,shorten <=2ex,shorten >=3.5ex," "]
	\end{tikzcd}
\]
Since the localization $\localization\colon \Kom(\catB)\to \Dom(\catB)$ has a right adjoint $\hi \colon \Dom(\catB)\to\Kom(\catB)$, the right derived functor $\bbRH_X$ is given by the composites
\[ \Dom(\catB) \xrightarrow{\hi} \Kom(\catB) \xrightarrow{H_X} \Kom(\catA) \xrightarrow{\localization} \Dom(\catA), \]
and becomes an absolute Kan extension.
The following result is fundamental.

\begin{proposition}
	Let $X\colon \catA\slashedrightarrow\catB$ be a dg bimodule.
	\begin{enumerate}
		\item The derived functors $\bbLT_X$ and $\bbRH_X$ are triangulated functors.
		\item We have an adjunction
		      \begin{equation*}
			      \bbLT_X \colon
			      \begin{tikzcd}
				      \Dom(\catA) \arrow[shift left=1ex]{r}{} \arrow[phantom]{r}[sloped]{\scriptstyle\perp} &
				      \Dom(\catB) \arrow[shift left=1ex]{l}{}
			      \end{tikzcd}
			      \colon \bbRH_X\rlap{.}
		      \end{equation*}
	\end{enumerate}
\end{proposition}

The next proposition is extracted from \cite{Keller:1994Deriving}. It seems to require the additional assumption of right h-projectiveness, which is not included in the original paper.

\begin{proposition}[{\cite[\S6.1, Lemma (a)]{Keller:1994Deriving}}]\label{prop:bbLT_X_is_equiv}
	Let $X\colon \catA\slashedrightarrow\catB$ be a dg bimodule. If $X$ is right h-projective, then the following are equivalent.
	\begin{enumerate}[label=\equivitem]
		\item The left derived functor $\bbLT_X\colon \Dom(\catA)\to\Dom(\catB)$ is an equivalence of categories.
		\item The morphism of complexes
		      \[ \catA(A,A') \to \Comdg(\catB)(X(\mplaceholder,A),X(\mplaceholder,A')) \]
		      is a quasi-isomorphism for all $A,A'\in \catA$, and the set $\{X(\mplaceholder,A)\}_{A\in \catA}$ of objects compactly generates $\Dom(\catB)$.
	\end{enumerate}
\end{proposition}

\begin{proof}
	Since the representable $\catA(\mplaceholder,A)$ is h-projective, we have $\bbLT_X(\catA(\mplaceholder,A))= T_X(\catA(\mplaceholder,A))=X(\mplaceholder,A)$ and
	\begin{align*}
		\Hom_{\Dom(\catA)}(\catA(-,A),\catA(-,A')[n])&\cong \Hom_{\Kom(\catA)}(\catA(-,A),\catA(-,A')[n]) \\
		&\cong H^n(\Comdg(\catA)(\catA(-,A),\catA(-,A')))\\
		&\cong H^n(\catA(A,A')).
	\end{align*}
	Because $X(-,A)$ is h-projective, we also have 
	\begin{align*}
		\Hom_{\Dom(\catB)}(X(-,A),X(-,A')[n])&\cong \Hom_{\Kom(\catB)}(X(-,A),X(-,A')[n])\\
		&\cong H^n(\Comdg(\catB)(X(-,A),X(-,A'))).
	\end{align*}
	Thus the assertion follows from \cref{prop:Dom(catA)_is_compactly_generated} and \cref{prop:lemma_for_compactly_generated_tria_cat}.
\end{proof}

\begin{corollary}\label{cor:Morita_equivalence_functor_criterion}
	For a dg functor $F\colon \catA\to\catB$ between small dg categories, the following are equivalent.
	\begin{enumerate}[label=\equivitem]
		\item The left derived functor $\bbLT_{F_*}\colon \Dom(\catA)\to\Dom(\catB)$ is an equivalence of categories.
		\item $F$ is quasi-fully faithful and the set $\{\catB(\mplaceholder,FA)\}_{A\in \catA}$ of objects compactly generates $\Dom(\catB)$.
	\end{enumerate}
\end{corollary}

\begin{proof}
	It follows from \cref{prop:bbLT_X_is_equiv}.
\end{proof}

Assigning to a dg bimodule $X\colon \catA\slashedrightarrow\catB$ the left derived functor $\bbLT_X\colon \Dom(\catA)\to\Dom(\catB)$ gives rise to a functor
$\bbL=\bbL_{\catA,\catB} \colon \Com(\catA,\catB) \to \Fun(\Dom(\catA),\Dom(\catB))$.

\begin{proposition}[{\cite[\S6.1, Lemma (b)]{Keller:1994Deriving}}]\label{prop:qism_induces_iso_btw_derived_func}
	A morphism $\mu\colon X\to Y$ of dg bimodules in $\Com(\catA,\catB)$ is a quasi-isomorphism if and only if the natural transformation $\bbLT_\mu\colon \bbLT_X \Rightarrow \bbLT_Y$ is an isomorphism.
\end{proposition}

\begin{proof}
	By \cref{prop:Dom(catA)_is_compactly_generated} and \cref{prop:lemma_for_compactly_generated_tria_cat}~(4), we can see that $\bbLT_\mu$ is an isomorphism if and only if each component $\mu_A\colon X(\mplaceholder,A) \to Y(\mplaceholder,A)$ is invertible in $\Dom(\catB)$. The latter is equivalent to the condition that each $\mu_A$ is a quasi-isomorphism, which means $\mu$ itself is a quasi-isomorphism.
\end{proof}

\Cref{prop:qism_induces_iso_btw_derived_func} implies that the functor $\bbL_{\catA,\catB}$ factors as
\[
	\begin{tikzcd}
		\Com(\catA,\catB) \arrow{d}[swap]{\localization} \arrow{r}{\bbL_{\catA,\catB}} & \Fun(\Dom(\catA),\Dom(\catB)) \\
		\Dom(\catA,\catB) \arrow[dashed]{ru}[swap]{\bbLtil_{\catA,\catB}}
	\end{tikzcd}
\]
by the universality of localization.

\begin{corollary}[See also {\cite[Lemma C.5]{Bodzenta-Bondal:2022}}]\label{cor:bbLtil_is_locally_conservative}
	Let $\catA$ and $\catB$ be small dg categories. Then the induced functor $\bbLtil=\bbLtil_{\catA,\catB}\colon \Dom(\catA,\catB) \to \Fun(\Dom(\catA),\Dom(\catB))$ is conservative: namely, a morphism $\alpha\colon X\Rightarrow Y$ in $\Dom(\catA,\catB)$ is invertible if and only if $\bbLtil(\alpha)$ is a natural isomorphism.
\end{corollary}

\begin{proof}
	Every morphism $\alpha\colon X\Rightarrow Y$ in $\Dom(\catA,\catB)$ can be written as the composite $\mu \circ (p_X)^{-1}$ in which $\mu\colon \hp(X)\to Y$ is some morphism of dg bimodules and $p_X\colon\hp(X)\to X$ is an h-projective resolution. Then we have $\bbLtil(\alpha)=\bbLtil(\mu)\circ \bbLtil(p_X)^{-1} =\bbL(\mu)\circ \bbL(p_X)^{-1}$. Since $\bbL(p_X)$ is invertible by \cref{prop:qism_induces_iso_btw_derived_func}, $\bbLtil(\alpha)$ is an isomorphism if and only if $\bbL(\mu)=\bbLT_\mu$ is so. Again by \cref{prop:qism_induces_iso_btw_derived_func} the latter is equivalent to $\mu$ being a quasi-isomorphism. Thus the assertion follows.
\end{proof}

\begin{proposition}\label{prop:bbLtilde_becomes_pseudo-functor}
	Let $X\colon \catA\slashedrightarrow\catB$ and $Y\colon \catB\slashedrightarrow\catC$ be dg bimodules.
	\begin{enumerate}
		\item There is an isomorphism $\bbLtil(I_\catA)=\bbL(I_\catA)\cong \id_{\Dom(\catA)}$.
		\item If $\catC$ is locally h-flat, then we have $\bbLtil(Y)\circ \bbLtil(X) \cong \bbLtil(Y \Dprocomp X)$.
	\end{enumerate}
\end{proposition}

\begin{proof}
	(1) It follows from the calculation  $\bbLT_{I_\catA}=\localization(\hp(\mplaceholder)\otimes_\catA I_\catA) \cong \localization\circ \hp(\mplaceholder)\cong \id_{\Dom(\catA)}$.

	(2) If $\catC$ is locally h-flat, then $\hp Y$ is left h-flat by \cref{lem:h-flat_with_locally_h-flat_domain_is_right_h-flat}. Hence we have
	\begin{align*}
		\bbLtil(Y)\circ \bbLtil(X)
		 & = \bbLT_Y\circ \bbLT_X                                               \\
		 & \cong \bbLT_{\hp Y} \circ \bbLT_{\hp X}                              \\
		 & \cong \hp(\hp(\mplaceholder)\otimes_\catA \hp X) \otimes_\catB \hp Y \\
		 & \cong (\hp(\mplaceholder)\otimes_\catA \hp X) \otimes_\catB \hp Y    \\
		 & \cong \hp(\mplaceholder)\otimes_\catA (\hp X \otimes_\catB \hp Y)    \\
		 & \cong \hp(\mplaceholder)\otimes_\catA (Y\Dprocomp X)                 \\
		 & \cong \bbLT_{Y\Dprocomp X} = \bbLtil(Y\Dprocomp X).\qedhere
	\end{align*}
\end{proof}

In the previous section, we observed that locally h-projective (resp.\ locally h-flat) small dg categories together with the derived categories $\Dom(\catA,\catB)$ form a bicategory $\DBimodhp$ (resp.\ $\DBimodhf$). It follows from \cref{prop:bbLtilde_becomes_pseudo-functor} that the mapping $\catA\mapsto \Dom(\catA)$ and the functors $\bbLtil_{\catA,\catB}\colon \Dom(\catA,\catB) \to \Fun(\Dom(\catA),\Dom(\catB))$ assemble into a pseudo-functor
\[ \text{$\bbLtil\colon \DBimodhp\to \CAT$} \quad \text{(resp. $\bbLtil\colon \DBimodhf\to \CAT$)}. \]

\section{Adjunctions in $\DBimod$}\label{section:adjunctions_in_DBimod}

In the sequel, we will pay main attention to the bicategory $\DBimodhp$ of locally h-projective small dg categories, writing $\DBimod\coloneqq\DBimodhp$, unless stated otherwise.
In this section, we investigate right extensions, right liftings, and adjunctions in the bicategory $\DBimod$.

\begin{proposition}\label{prop:DBimod_is_closed_bicategory}
	Let $X\colon \catA\slashedrightarrow\catB$, $Y\colon \catB\slashedrightarrow\catC$, and $Z\colon \catA\slashedrightarrow\catC$ be dg bimodules between small dg categories. Then there exist dg bimodules $X^{\Dddag} Z$, $Y_{\Dddag} Z$ and natural isomorphisms of categories
	\begin{align*}
		\Dom(\catA,\catC)(Y\Dprocomp X, Z) & \cong \Dom(\catB,\catC)(Y,X^{\Dddag} Z), \\
		\Dom(\catA,\catC)(Y\Dprocomp X,Z)  & \cong \Dom(\catA,\catB)(X,Y_{\Dddag} Z).
	\end{align*}
	If $\catB$ is locally h-projective, we can take $X^{\Dddag} Z=(\hp X)^\ddag Z$ and $Y_{\Dddag} Z=(\hp Y)_{\ddag} Z$.

	In particular, when $\catA$, $\catB$, and $\catC$ are locally h-projective (or locally h-flat), we have $X^{\Dddag} Z=\DRan_X Z$ and $Y_{\Dddag} Z =\DRift_Y Z$ in the bicategory $\DBimodhp$ (or $\DBimodhf$); hence $\DBimodhp$ (and $\DBimodhf$) is closed. Here $\DRan$ and $\DRift$ denote right Kan extension and lifting in the bicategory $\DBimodhp$ (or $\DBimodhf$), respectively.
\end{proposition}

\begin{proof}
	We have isomorphisms
	\begin{align*}
		\Dom(\catA,\catC)(Y\Dprocomp X, Z)
		 & \cong \Kom(\catA,\catC)(Y\Dprocomp X, \hi Z)                          \\
		 & \cong \Kom(\catA,\catC)(\hp Y\procomp\hp X, \hi Z)                    \\
		 & = H^0\Big( \Comdg(\catA,\catC)(\hp Y\procomp\hp X, \hi Z) \Big)       \\
		 & \cong H^0\Big( \Comdg(\catB,\catC)(\hp Y,(\hp X)^\ddag (\hi Z)) \Big) \\
		 & = \Kom(\catB,\catC)(\hp Y,(\hp X)^\ddag (\hi Z))                      \\
		 & \cong \Dom(\catB,\catC)(Y,(\hp X)^\ddag (\hi Z)).
	\end{align*}
	Therefore we can take $X^{\Dddag} Z=(\hp X)^\ddag (\hi Z)$. Similarly $Y_{\Dddag} Z=(\hp Y)_{\ddag} (\hi Z)$.

	If $\catB$ is locally h-projective, then $\hp(Y\Dprocomp X)=\hp Y\procomp \hp X$ by \cref{lem:tensor_of_h-proj_is_h-proj}. Hence we have $X^{\Dddag} Z=(\hp X)^\ddag Z$ by similar arguments.

	The last assertion is due to \cref{prop:closed_bicat_has_Ran_and_Rift}.
\end{proof}

As seen in \cref{prop:Bimod_is_closed_bicategory} the bicategory $\Bimod$ is closed, and so we have
\[ \Comdg(\catB,\catA)(Y,X_\ddag I_\catB) \cong \Comdg(\catB,\catB)(X\procomp Y,I_\catB) \cong \Comdg(\catA,\catB)(X,Y^\ddag I_\catB) \]
for small dg categories $\catA,\catB$, which means that the dg functors
\begin{align*}
	\IsbellL & = (\mplaceholder)_\ddag I_\catB \colon \Comdg(\catA,\catB) \to \Comdg(\catB,\catA)^\op, \\
	\IsbellR & = (\mplaceholder)^\ddag I_\catB \colon \Comdg(\catB,\catA)^\op \to \Comdg(\catA,\catB)
\end{align*}
form a dg adjunction $\IsbellL \dashv \IsbellR$. This adjunction is called \emph{dg Isbell duality} (\cite{Genovese:2017}).

Likewise the bicategory $\DBimodhp$ (or $\DBimodhf$) is also closed, and so we have
\[ \Dom(\catB,\catA)(Y,X_{\Dddag} I_\catB) \cong \Dom(\catB,\catB)(X\Dprocomp Y,I_\catB) \cong \Dom(\catA,\catB)(X,Y^{\Dddag} I_\catB) \]
for locally h-projective (or locally h-flat) small dg categories $\catA,\catB$, which means that the functors
\begin{align*}
	\DIsbellL & = (\mplaceholder)_{\Dddag} I_\catB \colon \Dom(\catA,\catB) \to \Dom(\catB,\catA)^\op, \\
	\DIsbellR & = (\mplaceholder)^{\Dddag} I_\catB \colon \Dom(\catB,\catA)^\op \to \Dom(\catA,\catB)
\end{align*}
form an adjunction $\DIsbellL \dashv \DIsbellR$. This adjunction is called \emph{derived Isbell duality} (\cite{Genovese:2017}).

\begin{lemma}\label{lem:H_X_preserves_acyclic_modules}
	Let $X\colon \catA\slashedrightarrow\catB$ be a dg bimodule. Suppose $X$ is right h-projective. Then the functor $H_X\colon \Comdg(\catB)\to\Comdg(\catA)$ preserves quasi-isomorphisms. In particular it induces $H_X\colon \Dom(\catA)\to\Dom(\catB)$ by the universality of localization, and the right and left derived functors $\bbRH_X,\bbL H_X$ are both isomorphic to it.
\end{lemma}

\begin{proof}
	It suffices to observe that $H_X$ preserves acyclicity. Given a dg module $N\in \Comdg(\catB)$ and an object $A\in \catA$, we have by definition
	\[ H_X N(A) = \Comdg(\catB)(X(\mplaceholder,A),N). \]
	If $N$ is acyclic, then $H_X N(A)$ is also acyclic as 	$X$ is right h-projective, which completes the proof.
\end{proof}

Let $X\colon\catA\slashedrightarrow\catB$ be a dg bimodule. The dg bimodule $\IsbellL(X)$ is the right Kan lifting in $\Bimod$
\[
	\begin{tikzcd}
		& \catA \arrow{d}{X}[sloped,marking]{\mapstochar}
		\arrow[Rightarrow,d,"",shift right=3.4ex,shorten <=1.7ex,shorten >=-0.3ex] \\
		\catB \arrow[bend left=15]{ru}{\IsbellL(X)}[sloped,marking]{\mapstochar} \arrow{r}[swap]{I_\catB}[sloped,marking]{\mapstochar}
		& \catB\rlap{.}
	\end{tikzcd}
\]
Applying the pseudo-functor $T\colon \Bimod\to \CAT$ to this diagram, we have a natural transformation $T_X\circ T_{\IsbellL(X)} \Rightarrow \Id_{\Com(\catB)}$, which corresponds to a natural transformation $\nu\colon T_{\IsbellL(X)}\Rightarrow H_X$ by the adjunction $T_X\dashv H_X$. We finally get a transformation $\bbL\nu\colon \bbLT_{\IsbellL(X)} \Rightarrow \bbL H_X\colon \Dom(\catB)\to\Dom(\catA)$ between left derived functors.

\begin{proposition}[{\cite[\S6.2, Lemma]{Keller:1994Deriving}}]\label{prop:H_X_preserves_coproducts}
	Let $X\colon \catA\slashedrightarrow\catB$ be a dg bimodule and suppose $X$ is right h-projective. Notice that $\bbL H_X \cong H_X$ by \cref{lem:H_X_preserves_acyclic_modules}. Then:
	\begin{enumerate}
		\item For the representable dg module $\catA(\mplaceholder,A)\in \Dom(\catA)$, the component $\bbL\nu_{\catA(\mplaceholder,A)}\colon \bbLT_X(\catA(\mplaceholder,A))\to H_X(\catA(\mplaceholder,A))$ is invertible.
		\item The following are equivalent:
		      \begin{enumerate}[label=\equiviitem]
			      \item the transformation $\bbL\nu\colon \bbLT_{\IsbellL(X)} \Rightarrow H_X$ is an isomorphism;
			      \item the triangulated functor $H_X\colon \Dom(\catB)\to\Dom(\catA)$ preserves coproducts;
			      \item $X$ is right compact, i.e.\  $X(\mplaceholder,A)\in \Dom(\catB)$ is compact for each $A\in\catA$.
		      \end{enumerate}
		\item If the left derived functor $\bbLT_X\colon \Dom(\catA)\to\Dom(\catB)$ is an equivalence of categories, then we have $(\bbLT_X)^{-1}\cong \bbLT_{\IsbellL(X)}$.
	\end{enumerate}
\end{proposition}

\begin{proof}
	(1) Since the representable $\catA(\mplaceholder,A)$ is h-projective, we have
	\begin{align*}
		\bbLT_{\IsbellL(X)}(\catA(\mplaceholder,A)) & \cong T_{\IsbellL(X)}(\catA(\mplaceholder,A)) \cong \IsbellL(X)(?,A)                      \\
		                                            & = \Comdg(\catA)(X(\mplaceholder,?),\catA(\mplaceholder,A)) = H_X(\catA(\mplaceholder,A)).
	\end{align*}

	(2) (a) $\Leftrightarrow$ (b): It follows from \cref{prop:Dom(catA)_is_compactly_generated} and \cref{prop:lemma_for_compactly_generated_tria_cat}~(4) together with (1) above.
	(b) $\Leftrightarrow$ (c): The condition that $X$ is right compact means that the left derived functor $\bbLT_X\colon \Dom(\catA) \to \Dom(\catB)$ sends the compact generators $\catA(\mplaceholder,A)$ to compact objects. Therefore, it follows from \cref{prop:condition_for_preserving_compact} and \cref{prop:left-adj_preserves_compact_iff_right-adj_preserves_coproduct}.

	(3) If $\bbLT_X$ is an equivalence, then so is $H_X$ because of the adjunction $\bbLT_X\dashv H_X$. In particular $H_X$ preserves coproducts. Therefore (2) implies $\bbLT_{\IsbellL(X)} \cong H_X \cong (\bbLT_X)^{-1}$.
\end{proof}

\begin{lemma}\label{lem:bbLT_of_Rift}
	Let $X\colon \catA\slashedrightarrow\catB$ be a dg bimodule and suppose that $X$ is right h-projective. Notice that $H_X$ preserves quasi-isomorphisms by \cref{lem:H_X_preserves_acyclic_modules} and we can take $\bbL H_X=H_X$. If the functor $H_X\colon \Dom(\catB)\to\Dom(\catA)$ preserves coproducts, then there exists a natural isomorphism
	\[ \bbLT_{X_\ddag Z} \cong H_X \circ \bbLT_Z \]
	for a dg bimodule $Z\colon \catC\slashedrightarrow\catB$.
\end{lemma}

\begin{proof}
	Since the dg bimodule $X_\ddag Z$ is the right Kan lifting $\Rift_X Z$ of Z along X in the bicategory $\Bimod$, we have a $2$-cell $X \procomp (X_\ddag Z) \Rightarrow Z$ in $\Bimod$ and hence a natural transformation $T_X \circ T_{X_\ddag Z} \Rightarrow T_Z\colon \Com(\catC)\to\Com(\catB)$. By the adjunction $T_X \dashv H_X$, we get $T_{X_\ddag Z} \Rightarrow H_X\circ T_Z\colon \Com(\catC)\to \Com(\catA)$, which induces a natural transformation $\bbLT_{X_\ddag Z} \Rightarrow H_X\circ \bbLT_Z\colon \Dom(\catC)\to \Dom(\catA)$ between left derived functors. Now we have
	\begin{align*}
		\bbLT_{X_\ddag Z}(\catC(\mplaceholder,C))
		 & = T_{X_\ddag Z}(\catC(\mplaceholder,C)) = (X_\ddag Z)(?,C)                             \\
		 & = \Comdg(\catB)\big(X(\mplaceholder,?),Z(\mplaceholder,C)\big)                         \\
		 & = H_X\big(Z(\mplaceholder,C)\big)                                                      \\
		 & = H_X\big(T_Z(\catC(\mplaceholder,C))\big) = H_X \circ \bbLT_Z(\catC(\mplaceholder,C))
	\end{align*}
	for each $C\in \catC$. Therefore $\bbLT_{X_\ddag Z} \Rightarrow H_X\circ \bbLT_Z$ is invertible on representables. By the assumption that $H_X$ preserves coproducts, we see $\bbLT_{X_\ddag Z} \cong H_X\circ \bbLT_Z$ from \cref{prop:lemma_for_compactly_generated_tria_cat}.
\end{proof}

We can characterize left adjoints in $\DBimod$ as follows.

\begin{theorem}\label{thm:right_compact_bimodules_are_right_adjoint}
	Let $\catA,\catB$ be locally h-projective small dg categories and $X\colon \catA\slashedrightarrow\catB$ be a dg bimodule. Then the following are equivalent.
	\begin{enumerate}[label=\equivitem]
		\item The dg bimodule $X$ has a right adjoint in the bicategory $\DBimod$.
		\item The right derived functor $\bbRH_X\colon \Dom(\catB)\to\Dom(\catA)$ preserves coproducts.
		\item $X$ is right compact.
	\end{enumerate}
	In this case, $\DIsbellL(X)$ is the right adjoint of $X$ in $\DBimod$.
\end{theorem}

Note that this is a generalization of \cite[Corollary 6.6]{Genovese:2017} (see \cref{cor:quasi-functor_has_right_adjoint}).
One can find a partial result of this theorem in \cite[Proposition 2.11]{Anno-Logvinenko:2017spherical}.

\begin{proof}
	In the derived category $\Dom(\catA,\catB)$, every dg bimodule is isomorphic to an h-projective one, and so we may assume that $X$ is h-projective. Then it follows from \cref{lem:h-proj_implies_right_h-proj} that $X$ is right h-projective as $\catB$ is locally h-projective, and from \cref{lem:H_X_preserves_acyclic_modules} that $\bbRH_X\cong H_X$. Therefore \cref{prop:H_X_preserves_coproducts}~(2) leads to the equivalence of (ii) and (iii).

	(i) $\Rightarrow$ (ii): If $X$ has a right adjoint $Y$ in the bicategory $\DBimod$, then applying the pseudo-functor $\bbLtil\colon \DBimod\to \CAT$ leads to an adjunction $\bbLT_X \dashv \bbLT_Y$. Since $\bbLT_X \dashv \bbRH_X$, we have $\bbLT_Y\cong \bbRH_X$. In particular, $\bbRH_X$ preserves coproducts.

	(ii) $\Rightarrow$ (i): In the bicategory $\DBimod$, there exists a right Kan lifting
	\[
		\begin{tikzcd}
			& \catA \arrow{d}{X}[sloped,marking]{\mapstochar}
			\arrow[Rightarrow,d,"",shift right=3.0ex,shorten <=1.9ex,shorten >=-0.3ex] \\
			\catB \arrow[bend left=15]{ru}{\DIsbellL(X)=X_\Dddag I_\catB}[sloped,marking]{\mapstochar} \arrow{r}[swap]{I_\catB}[sloped,marking]{\mapstochar}
			& \catB\rlap{.}
		\end{tikzcd}
	\]
	It is sufficient by \cref{prop:adjunction_as_Kan_extension/lifting} to prove that this Kan lifting is absolute. Take a locally h-projective small dg category $\catC$ and a dg bimodule $Z\colon \catC\slashedrightarrow\catB$. Composing $Z$ with the above diagram, we use the universality of the right Kan lifting $\DRift_X Z=X_\Dddag Z$ to obtain a $2$-cell in $\DBimod$
	\[ n\colon \DIsbellL(X) \Dprocomp Z \Rightarrow \DRift_X Z=X_\Dddag Z \]
	satisfying
	\[
		\begin{tikzcd}
			\catB \arrow{r}{\DIsbellL(X)}[sloped,marking]{\mapstochar} \arrow[equal]{rd}
			& \catA \arrow{d}{X}[sloped,marking]{\mapstochar}
			\arrow[Rightarrow,d,"",shift right=2.8ex,shorten <=-0.4ex,shorten >=2.3ex] \\
			\catC \arrow{u}{Z}[sloped,marking]{\mapstochar} \arrow{r}[below]{Z}[sloped,marking]{\mapstochar}
			& \catB
		\end{tikzcd}
		\quad=\quad
		\begin{tikzcd}
			\catB \arrow{r}{\DIsbellL(X)}[sloped,marking]{\mapstochar}
			& \catA \arrow{d}{X}[sloped,marking]{\mapstochar}
			\arrow[Rightarrow,d,"",shift right=2.8ex,shorten <=2.3ex,shorten >=-0.4ex] \\
			\catC \arrow{u}{Z}[sloped,marking]{\mapstochar} \arrow{r}[below]{Z}[sloped,marking]{\mapstochar} \arrow{ru}[sloped,marking]{\mapstochar}[name=B2,above,pos=0.4]{}%[pos=0.2,outer sep=0pt,description]{{\tiny X_\Dddag Z}}
			\arrow[Rightarrow,from=1-1,to=B2,dotted,shorten <=-1pt,shorten >=-1pt,"n",pos=0.35]
			& \catB\rlap{.}
		\end{tikzcd}
	\]
	It remains to show that $n$ is invertible. Because the pseudo-functor $\bbLtil\colon \DBimod\to\CAT$ is locally conservative as seen in  \cref{cor:bbLtil_is_locally_conservative}, we only need to check that
	\[ \bbLtil(n) \colon \bbLT_{\DIsbellL(X)} \circ \bbLT_Z \to \bbLT_{X_\Dddag Z} \]
	is an isomorphism.
	Now that $H_X$ preserves coproducts, \cref{lem:H_X_preserves_acyclic_modules} and \cref{lem:bbLT_of_Rift} imply that $\bbLT_{\IsbellL(X)} \cong H_X$ and $\bbLT_{X_\ddag Z} \cong H_X \circ \bbLT_Z$. Since $X$ is h-projective and $\catA$ is locally h-projective, we have $X_\Dddag Z = X_\ddag Z$ by \cref{prop:DBimod_is_closed_bicategory}. Thus we can see that $\bbLtil(n)$ is an isomorphism.
\end{proof}

By duality we have:

\begin{theorem}\label{thm:left_compact_bimodules_has_left_adjoint}
	Let $\catA,\catB$ be locally h-projective small dg categories and $Y\colon \catB\slashedrightarrow\catA$ be a dg bimodule. Then the following are equivalent.
	\begin{enumerate}[label=\equivitem]
		\item The dg bimodule $Y$ has a left adjoint in the bicategory $\DBimod$.
		\item $Y$ is left compact, i.e.\ $Y(A,\mplaceholder)\in\Dom(\catB^\op)$ is compact for each $A\in\catA$.
	\end{enumerate}
	In this case, $\DIsbellR(Y)$ is the left adjoint of $Y$ in $\DBimod$.
\end{theorem}

\begin{proof}
	Since $\catB^\op\otimes\catA \cong (\catA^\op)^\op\otimes\catB^\op$, we get $\Dom(\catA,\catB)\cong \Dom(\catB^\op,\catA^\op)$. This identification induces a $2$-equivalence
	\[ \DBimod \xrightarrow{\simeq} \DBimod^\op, \quad \catA \mapsto \catA^\op. \]
	Thus, by duality, it follows easily from \cref{thm:right_compact_bimodules_are_right_adjoint}.
\end{proof}

We can also get a characterization of equivalences in $\DBimod$.

\begin{theorem}\label{thm:equivalence_in_DBimod}
	Let $\catA,\catB$ be locally h-projective dg categories and
	$X\colon \catA\slashedrightarrow\catB$ be a dg bimodule. Then the following are equivalent.
	\begin{enumerate}[label=\equivitem]
		\item $X$ is an equivalence in the bicategory $\DBimod$.
		\item The left derived functor $\bbLT_X\colon \Dom(\catA)\to\Dom(\catB)$ is an equivalence of categories.
	\end{enumerate}
\end{theorem}

\begin{proof}
	(i) $\Rightarrow$ (ii): This is trivial since the pseudo-functor $\bbLtil\colon \DBimod\to \CAT$ preserves equivalences.

	(ii) $\Rightarrow$ (i): Notice that $\bbLT_X$ sends the representable dg module $\catA(\mplaceholder,A)$ to $X(\mplaceholder,A)$. If $\bbLT_X$ is an equivalence, then it preserves compactness, which implies that $X$ is right compact. Hence by \cref{thm:right_compact_bimodules_are_right_adjoint} $X$ has a right adjoint $\DIsbellL(X)$ in $\DBimod$. Let the unit and counit of this adjunction be written by $\eta$ and $\varepsilon$, respectively. It suffices to show that these $2$-cells in $\DBimod$ are isomorphisms. Applying the pseudo-functor $\bbLtil$, we have an adjunction $\bbLT_X=\bbLtil(X)\dashv \bbLtil(\DIsbellL(X))$ with $\bbLtil(\eta)$ unit and $\bbLtil(\varepsilon)$ counit. Now that $\bbLT_X$ is an equivalence, it follows that $\bbLtil(\eta)$ and $\bbLtil(\varepsilon)$ are invertible. Then $\eta$ and $\varepsilon$ are invertible as well, because $\bbLtil$ is locally conservative by \cref{cor:bbLtil_is_locally_conservative}. Thus $X$ is an equivalence in $\DBimod$.
\end{proof}

\section{Quasi-functors}\label{section:quasi-functors}

Recall that we have put $\DBimod\coloneqq\DBimodhp$.

\subsection{The bicategory of quasi-functors}\label{subsection:the_bicategory_of_quasi-functors}

Let $M$ be a right dg module over a small dg category $\catA$. We say that $M$ is \emph{quasi-representable} if it is isomorphic to some representable dg module in the derived category $\Dom(\catA)$. This is equivalent to the condition that there are an object $A\in \catA$ and a quasi-isomorphism $\catA(\mplaceholder,A) \to M$ between dg modules. Let $\qucl{\catA}\subseteq \Dom(\catA)$ denote the full subcategory consisting of all quasi-representable dg modules. As the representables are h-projective, we have an equivalence $H^0(\catA)\simeq \qucl{\catA}$.

We call a dg bimodule $X\colon \catA\slashedrightarrow\catB$ \emph{right quasi-representable}, or a \emph{quasi-functor}, if the dg $\catB$-module $X(\mplaceholder,A)$ is quasi-representable for every $A\in\catA$. Then the left derived functor $\bbL T_X\colon \Dom(\catA) \to \Dom(\catB)$ is restricted to a functor
\[ \res{\bbLT_X}{\qucl{\catA}}\colon \qucl{\catA}\to \qucl{\catB}. \]
Analogously we define \emph{left quasi-representable} dg bimodules as dg bimodules $X$ for which $X(B,\mplaceholder)$ is a quasi-representable dg $\catA^\op$-module for every $B\in \catB$.

Now, quasi-representable dg modules are compact and so quasi-functors are right compact. Therefore \cref{thm:right_compact_bimodules_are_right_adjoint} induces the following.

\begin{corollary}[{\cite[Corollary 6.6]{Genovese:2017}}]\label{cor:quasi-functor_has_right_adjoint}
	A quasi-functor $X\colon \catA\slashedrightarrow\catB$ between locally h-projective small dg categories has a right adjoint $\DIsbellL(X)$ in $\DBimodhp$.
\end{corollary}

This is a homotopical analog of the well-known fact that the dg bimodule $F_\equipsubAst=\catB(\mplaceholder,F \mplaceholder)$ associated with a dg functor $F\colon \catA\to\catB$ is left adjoint to the dg bimodule $\IsbellL(F_\equipsubAst)=F^\equipsubAst=\catB(F\mplaceholder,\mplaceholder)$ in the bicategory $\Bimod$; see \cite[Proposition 7.9.1]{Borceux:1994HoCA1} or \cite[Remark 5.2.1]{Loregian:2021coend_calculus} for the $\Set$-enriched case.

Moreover we characterize those dg bimodules that can be written as $\DIsbellL(X)$ for a quasi-functor $X$.

\begin{lemma}\label{lem:A_ast_is_left_adjoint_to_A^star}
	Let $\catA$ be a locally h-projective dg category and $A\in \catA$ an object, which we view as a dg functor $A\colon \basek\to\catA$. Then the associated dg bimodule $A_\equipsubAst\colon \basek\slashedrightarrow \catA$ is left adjoint to the dg bimodule $A^\equipsubAst\colon \catA\slashedrightarrow\basek$ in $\DBimod$.
\end{lemma}

\begin{proof}
	Under the isomorphism $\Comdg(\basek,\catA)\cong \Comdg(\catA)$, $A_\equipsubAst$ corresponds to $\catA(\mplaceholder,A)$, which is h-projective and compact. Thus, by \cref{thm:right_compact_bimodules_are_right_adjoint}, $A_\equipsubAst$ has $\DIsbellL(A_\equipsubAst)=\IsbellL(A_\equipsubAst)=A^\equipsubAst$ as a right adjoint.
\end{proof}

\begin{proposition}[{\cite[Corollary 5.9]{Genovese:2017}}]
	Let $\catA, \catB$ be locally h-projective dg categories.
	If a dg bimodule $X\colon\catA\slashedrightarrow\catB$ is right quasi-representable, then $\DIsbellL(X)$ is left quasi-representable and $\DIsbellR(\DIsbellL(X)) \cong X$.
	Dually, if a dg bimodule $Y\colon\catB\slashedrightarrow\catA$ is left quasi-representable, then $\DIsbellR(Y)$ is right quasi-representable and $\DIsbellL(\DIsbellR(Y)) \cong Y$.
\end{proposition}

\begin{proof}
	We may assume that $X$ is h-projective. Then $X\Dprocomp A_\equipsubAst \cong X\procomp A_\equipsubAst \cong X(\mplaceholder,A)$ in $\Dom(\catB)$ and $B^\equipsubAst\Dprocomp X\cong B^\equipsubAst\procomp X\cong X(B,\mplaceholder)$ in $\Dom(\catA^\op)$ since $A_\equipsubAst$ and $B^\equipsubAst$ are h-projective. If $X$ is right quasi-representable, we have $X\dashv \DIsbellL(X)$ in $\DBimod$ by \cref{thm:right_compact_bimodules_are_right_adjoint}. By \cref{lem:A_ast_is_left_adjoint_to_A^star} and composing adjunctions, we obtain $X\Dprocomp A_\equipsubAst \dashv A^\equipsubAst\Dprocomp\DIsbellL(X)$. Now $X\Dprocomp A_\equipsubAst\cong X(\mplaceholder,A)$ is quasi-representable, so it can be written as $B_\equipsubAst$ for some $B\in\catB$. Therefore, by \cref{lem:A_ast_is_left_adjoint_to_A^star} again, we have $B^\equipsubAst\cong A^\equipsubAst\Dprocomp\DIsbellL(X)\cong \DIsbellL(X)(A,\mplaceholder)$, which shows that $\DIsbellL(X)$ is left quasi-representable. 
	Then, we have $\DIsbellR(\DIsbellL(X))\dashv \DIsbellL(X)$ because of \cref{thm:left_compact_bimodules_has_left_adjoint}, and hence $\DIsbellR(\DIsbellL(X)) \cong X$ by the uniqueness of left adjoints.
	
	The dual statement follows similarly. 
\end{proof}

As a corollary of the previous proposition, we see that a dg bimodule $Y\colon \catB \slashedrightarrow \catA$ is left quasi-representable if and only if it is right adjoint to a quasi-functor in the bicategory $\DBimod$.

It is easily seen that the identities $I_\catA$ and the derived compositions $Y\Dprocomp X$ of quasi-functors are again quasi-functors.
We write by $\DBimodrqr$ the sub-bicategory of $\DBimod$ determined by quasi-functors. We give a characterization of equivalences in $\DBimodrqr$.

\begin{theorem}\label{thm:equivalence_in_DBimodrqr}
	Let $\catA$ and $\catB$ be locally h-projective dg categories and $X\colon \catA\slashedrightarrow\catB$ be a quasi-functor. Then the following are equivalent.
	\begin{enumerate}[label=\equivitem]
		\item $X$ is an equivalence in the bicategory $\DBimodrqr$ of quasi-functors.
		\item Both the left derived functor $\bbLT_X\colon \Dom(\catA)\to\Dom(\catB)$ and the restricted one $\res{\bbLT_X}{\qucl{\catA}}\colon \qucl{\catA}\to \qucl{\catB}$ are equivalences of categories.\footnotemark
	\end{enumerate}
\end{theorem}
\footnotetext{Note that a dg bimodule satisfying the condition (ii) is called a \emph{quasi-equivalence} in \cite[\S7.2]{Keller:1994Deriving}.}
\begin{proof}
	(i) $\Rightarrow$ (ii): The proof is easy and left to the reader.

	(ii) $\Rightarrow$ (i): Since $X$ is an equivalence in $\DBimod$ as well, $\DIsbellL(X)$ is its quasi-inverse.  The fact that $\bbLT_X$ induces an equivalence $\qucl{\catA}\xrightarrow{\simeq} \qucl{\catB}$ implies that $\DIsbellL(X)$ is also a quasi-functor. Hence $X$ is an equivalence of $\DBimodrqr$.
\end{proof}

For a dg functor $F\colon \catA\to \catB$, the associated dg bimodule $F_\equipsubAst\colon \catA\slashedrightarrow\catB$ is clearly a quasi-functor. When $F$ is a quasi-equivalence, $F_\equipsubAst$ is an equivalence in $\DBimodrqr$; see \cite[\S7.2, Example]{Keller:1994Deriving}. The next proposition says that the converse holds, which is also mentioned at the end of \cite[\S2.3]{Genovese-Lowen-VandenBergh:2022derived_deformation:arXiv} without proof.

\begin{proposition}\label{prop:quasi-equivalence_is_just_equivalence_in_DBimodrqr}
	Let $\catA$ and $\catB$ be locally h-projective dg categories and $F\colon \catA\to \catB$ be a dg functor. For the associated quasi-functor $F_\equipsubAst\colon\catA\slashedrightarrow\catB$, the following are equivalent.
	\begin{enumerate}[label=\equivitem]
		\item $F$ is a quasi-equivalence (in the sense of \cref{def:quasi-equivalence}).
		\item $F_\equipsubAst$ is an equivalence in the bicategory $\DBimodrqr$ of quasi-functors.
		\item Both the left derived functor $\bbLT_{F_\equipsubAst}\colon \Dom(\catA)\to\Dom(\catB)$ and the restricted one $\res{\bbLT_{F_\equipsubAst}}{\qucl{\catA}}\colon \qucl{\catA}\to \qucl{\catB}$ are equivalences of categories.
	\end{enumerate}
\end{proposition}

\begin{proof}
	(ii) $\Leftrightarrow$ (iii): It follows from \cref{thm:equivalence_in_DBimodrqr}.

	(i) $\Leftrightarrow$ (iii): Quasi-essential surjectivity of $F$ is equivalent to that of $\res{\bbLT_{F_\equipsubAst}}{\qucl{\catA}}\colon \qucl{\catA}\to\qucl{\catB}$. When $F$ is quasi-essentially surjective, $\{\catB(\mplaceholder,FA)\}_{A\in\catA}$ compactly generates $\Dom(\catB)$. Hence in such a case, we can see from \cref{cor:Morita_equivalence_functor_criterion} that $F$ is quasi-fully faithful if and only if $\bbLT_X$ is an equivalence.
\end{proof}

Let $X\colon \catA\slashedrightarrow\catB$ be a quasi-functor between locally h-projective small dg categories. We define $X$ to be \emph{fully faithful} if the unit of the adjunction $X \dashv \DIsbellL(X)$ is invertible.

\begin{proposition}\label{prop:fully_faithful_quasi-functor}
	Let $\catA$ and $\catB$ be locally h-projective dg categories and $X\colon \catA\slashedrightarrow\catB$ be a quasi-functor. Then the following are equivalent.
	\begin{enumerate}[label=\equivitem]
		\item $X$ is fully faithful as a quasi-functor.
		\item The left derived functor $\bbLT_{X}\colon \Dom(\catA)\to\Dom(\catB)$ is fully faithful.
	\end{enumerate}
\end{proposition}

\begin{proof}
	Let $\eta$ be the unit of $X \dashv \DIsbellL(X)$. Then $\bbLtil(\eta)$ is the unit of the adjunction $\bbLT_X=\bbLtil(X)\dashv \bbLtil(\DIsbellL(X))$. By \cref{cor:bbLtil_is_locally_conservative}, $\eta$ is invertible if and only if $\bbLtil(\eta)$ is a natural isomorphism, which implies that $X$ is fully faithful if and only if the left adjoint $\bbLT_X$ is fully faithful.
\end{proof}

\begin{proposition}
	Let $\catA$ and $\catB$ be locally h-projective dg categories and $F\colon \catA\to \catB$ be a dg functor. For the associated quasi-functor $F_\equipsubAst\colon\catA\slashedrightarrow\catB$, the following are equivalent.
	\begin{enumerate}[label=\equivitem]
		\item $F$ is quasi-fully faithful.
		\item $F_\equipsubAst$ is fully faithful as a quasi-functor.
		\item The left derived functor $\bbLT_{F_\equipsubAst}\colon \Dom(\catA)\to\Dom(\catB)$ is fully faithful.
	\end{enumerate}
\end{proposition}

\begin{proof}
	(ii) $\Leftrightarrow$ (iii): It follows from \cref{prop:fully_faithful_quasi-functor}.
	
	(iii) $\Rightarrow$ (i): Since the representables $\catA(-,A)$ are h-projective, we have
	\begin{align*}
		\Hom_{\Dom(\catA)}\big( \catA(-,A),\catA(-,A')[n] \big) &\cong \Hom_{\Kom(\catA)}\big( \catA(-,A),\catA(-,A')[n] \big) \\
		&\cong H^n\Big( \Comdg(\catA)\big( \catA(-,A),\catA(-,A') \big) \Big) \\
		&\cong H^n(\catA(A,A')).
	\end{align*}
	Thus since $\bbLT_{F_\equipsubAst}(\catA(-,A))=\catB(-,FA)$, the fully faithfulness of $\bbLT_{F_\equipsubAst}$ implies that all $H^n(\catA(A,A')) \to H^n(\catB(FA,FA'))$ are isomorphisms.
	
	(i) $\Rightarrow$ (iii): It follows from \cref{prop:lemma_for_compactly_generated_tria_cat} and the above isomorphisms.
\end{proof}

\subsection{Adjunctions of quasi-functors}\label{subsection:adjunction_of_quasi-functors}

In this subsection, we assume that all dg categories are locally h-projective, and $\dgCat$ and $\Bimod$ denote the full sub-(bi)categories of locally h-projective ones.

We refer to an adjunction in the bicategory $\DBimodrqr$ of quasi-functors as an \emph{adjunction of quasi-functors}; see \cite[\S7]{Genovese:2017}. If a quasi-functor $X$ is left adjoint to a quasi-functor $Y$, then we have an adjunction $X \dashv Y$ in the bicategory $\DBimod$. Therefore $X\dashv Y$ holds in $\DBimodrqr$ if and only if $\DIsbellL(X) \cong Y$ holds in $\DBimod$. Summarizing, we have:

\begin{proposition}[{\cite[Proposition 7.1]{Genovese:2017}}]\label{prop:characterization_of_adjointness_for_quasi-functors}
	A quasi-functor $X$ has a right adjoint in $\DBimodrqr$ if and only if the dg bimodule $\DIsbellL(X)$ is right quasi-representable.
	A quasi-functor $Y$ has a left adjoint in $\DBimodrqr$ if and only if $Y$ is left quasi-representable.
\end{proposition}

The notion of adjoint dg functors is defined by adjoint morphisms in the $2$-category $\dgCat$. In this subsection, we investigate the relationship between adjoint quasi-functors and adjoint dg functors which has not been explored in \cite{Genovese:2017}.

Recall that the diagonal dg bimodule $I_\catA$ is the identity in both $\Bimod$ and $\DBimod$. Also we recall that the composition $\Dprocomp$ of $\DBimod$ is given by the total right derived functor of the composition $\procomp$ of $\Bimod$, as follows.
\[
\begin{tikzcd}
	\Com(\catB,\catC) \times \Com(\catA,\catB) \arrow{r}{\procomp}[name=B2,below,pos=0.4]{} \arrow{d}[swap]{} & \Com(\catA,\catC) \arrow{d}{} \\
	\Kom(\catB,\catC) \times \Kom(\catA,\catB) \arrow{r}{\procomp}
	\arrow{r}[name=B1,below,pos=0.45]{}[name=A2,above,pos=0.4]{} \arrow{d}[swap]{} & \Kom(\catA,\catC) \arrow{d}{} \\
	\Dom(\catB,\catC) \times \Dom(\catA,\catB) \arrow[dashed]{r}{\Dprocomp}[name=A1,above,pos=0.45]{} & \Dom(\catA,\catC)\rlap{.}
	\arrow[Rightarrow,from=A1,to=B1,shorten <=2ex,shorten >=1.5ex," "]
	\arrow[phantom,from=A2,to=B2,"\circlearrowleft"]
\end{tikzcd}
\]
%Note that the commutative square in the above diagram becomes an absolute Kan extension.
Hence there is a natural morphism $\mu\colon Y\Dprocomp X =\hp Y \procomp \hp X \to Y\procomp X$ in $\Dom(\catA,\catC)$, which is given by using h-projective resolutions of dg bimodules.

\begin{lemma}
	The following data define a normal lax functor $\Gamma\colon \Bimod\to\DBimod$:
	\begin{itemize}
		\item the map $\obj(\Bimod)\to \obj(\DBimod)$ is the identity;
		\item for each small dg categories $\catA,\catB$, $\Gamma_{\catA,\catB}\colon\Com(\catA,\catB)\to \Dom(\catA,\catB)$ is defined to be the composite $\Com(\catA,\catB)\to \Kom(\catA,\catB) \xrightarrow{\localization}\Dom(\catA,\catB)$;
		\item the lax functoriality constraint $\Gamma(Y)\Dprocomp\Gamma(X) \Rightarrow \Gamma(Y\procomp X)$ is defined as the natural morphism $\mu\colon Y\Dprocomp X =\hp Y \procomp \hp X \to Y\procomp X$ obtained above;
		\item the lax unity constraint is the identity.
	\end{itemize}
\end{lemma}

\begin{proof}
	The proof is left to the reader.
\end{proof}

Moreover the lax functor $\Gamma$ clearly induces a lax functor $\gamma= \Gamma\rvert_{\dgCat} \colon \dgCat \to \DBimodrqr$, which sends a dg functor $F$ to the quasi-functor $F_\equipsubAst$.

In general the constraint $\mu\colon Y\Dprocomp X \to Y\procomp X$ is not invertible, but we have the following result.

\begin{lemma}\label{lem:when_the_constraint_is_invertible}
	Let $X\colon\catA\slashedrightarrow\catB$ and $Y\colon \catB\slashedrightarrow\catC$ be dg bimodules between (locally h-projective) dg categories. If either $X$ is right h-flat or $Y$ is left h-flat, then $\mu\colon Y\Dprocomp X \to Y\procomp X$ is invertible in $\Dom(\catA,\catC)$.
\end{lemma}

\begin{proof}
	We prove only the first case. Suppose $X$ is right h-flat. The constraint $\mu\colon Y\Dprocomp X \to Y\procomp X$ factors as the composite
	\[ Y\Dprocomp X = \hp Y\procomp \hp X \xrightarrow{\hp Y \procomp q_X} \hp Y \procomp X \xrightarrow{q_Y \procomp X} Y\procomp X, \]
	where $q_X$ and $q_Y$ are h-projective resolutions of $X$ and $Y$, respectively. 
	Since $\hp Y$ is h-projective and $\catC$ is locally h-projective, \cref{lem:h-proj_implies_right_h-proj} implies that $\hp Y$ is left h-projective and hence left h-flat. Therefore, by \cref{lem:precomp_with_right_h-flat_preserves_qism}, the first morphism is a quasi-isomorphism. Since $X$ is right h-flat, the second morphism is also a quasi-isomorphism again by \cref{lem:precomp_with_right_h-flat_preserves_qism}.
	Thus their composite is sent to an isomorphism in $\Dom(\catA,\catC)$.
\end{proof}

For a dg functor $F\colon\catA\to\catB$, the associated dg bimodule $F_\equipsubAst\colon \catA\slashedrightarrow\catB$ is right h-flat. Hence the above lemma shows that the constraint $\mu\colon Y\Dprocomp F_\equipsubAst \to Y\procomp F_\equipsubAst$ is invertible. From this we obtain the following.

\begin{corollary}
	The restricted lax functor $\gamma\colon \dgCat\to \DBimodrqr$ is a pseudo-functor.
\end{corollary}

\begin{proof}
	It follows from \cref{lem:when_the_constraint_is_invertible}.
\end{proof}

Since a pseudo-functor preserves adjunctions, we have:

\begin{proposition}\label{prop:adj_of_dg_functors_becomes_those_of_quasi-functors}
	For an adjunction $F\dashv G$ of dg functors, we have an adjunction $F_\equipsubAst \dashv G_\equipsubAst$ of quasi-funcors.
	%If a dg functors $F$ is left adjoint to a dg functor $G$, then $F_\equipsubAst$ is left adjoint to $G_\equipsubAst$ as quasi-functors.
\end{proposition}

In \cref{subsection:the_tensor_and_Hom_functors_associated_with_dg_bimodules}, we have constructed the pseudo-functor $\bbLtil\colon \DBimod\to \CAT$, which sends $X\colon 
\catA\slashedrightarrow\catB$ to $\bbLT_X\colon \Dom(\catA)\to\Dom(\catB)$. For a quasi-functor $X$, the functor $\bbLtil(X)=\bbLT_X$ restricts to a functor $\bbLT_X\rvert_{\qucl{\catA}}\colon \qucl{\catA}\to \qucl{\catB}$, where $\qucl{\catA}$ denotes the full subcategory of $\Dom(\catA)$ consisting of quasi-representable dg modules. Putting $\Htil(\catA)\coloneqq \qucl{\catA}$ and $\Htil(X)\coloneqq \bbLT_{X}\rvert_{\qucl{\catA}}\colon \Htil(\catA) \to \Htil(\catB)$ for quasi-functors $X\colon \catA\slashedrightarrow\catB$, we obtain a new pseudo-functor
\[ \Htil\colon \DBimodrqr \to \Cat. \]
Therefore we see that an adjunction $X\dashv Y\colon \catA\slashedrightarrow\catB$ of quasi-functors induces an adjunction $\Htil(X) \dashv \Htil(Y)\colon \Htil(\catA)\to\Htil(\catB)$ of ordinary functors. Since $\Htil(\catA)= \qucl{\catA}$ is equivalent to the homotopy category $H^0(\catA)$, we end up with an adjunction between homotopy categories.
In fact, we can observe that the equivalences $H^0(\catA)\simeq \Htil(\catA)$ give rise to a pseudonatural equivalence $H^0 \xRightarrow{\simeq} \Htil\circ \gamma$ from the homotopy category $2$-functor $H^0\colon \dgCat\to \Cat$.

Pseudo-functors preserve adjunctions, while lax functors do not necessarily do so. However one can find the following result.

\begin{proposition}[{The dual of \cite[Proposition 2.7 and 2.9]{Dawson-Pare-Pronk:2004}}]\label{lem:condition_lax_functor_preserves_adjunctions}
	Let $\Phi\colon\bicatK\to\bicatL$ be a normal lax functor between bicategories. Let $\varphi_{g,f}$ and $\varphi_A$ denote the lax constraints of $\Phi$. Suppose that we are given an adjunction $(f\dashv u\colon A \to B, \eta,\varepsilon)$ in $\bicatK$. 
	\begin{enumerate}
		\item The following are equivalent:
		\begin{enumerate}[label=\equivitem]
			\item There is a $2$-cell $\overline{\eta}\colon \id_{\Phi(A)} \rightarrow \Phi(u)\circ\Phi(f)$ in $\bicatL$ such that
			\[\begin{tikzcd}
				\Phi(u\circ f) & \Phi(\id_A) \arrow{l}[swap]{\Phi(\eta)} \\
				\Phi(u)\circ \Phi(f) \arrow{u}{\varphi_{u,f}} & \id_{\Phi(A)} \arrow{u}[swap]{\varphi_A} \arrow[dashed]{l}[swap]{\overline{\eta}}
			\end{tikzcd}\]
			commutes.
			
			\item The constraint $\varphi_{g,f}\colon \Phi(g)\circ\Phi(f) \to \Phi(g\circ f)$ is an isomorphism for any composable morphism $g$.
		\end{enumerate}
		
		\item When (1) holds, we have $\Phi(f)\dashv\Phi(u)$ with unit $\overline{\eta}$.
	\end{enumerate}
\end{proposition}

Note that (i) above holds when $\varphi_{u,f}$ is an isomorphism.
In \cite{Dawson-Pare-Pronk:2004} it is said that $\Phi$ \emph{preserves the adjunction} $f\dashv u$ if the condition (i) of \cref{lem:condition_lax_functor_preserves_adjunctions}~(1) is satisfied.

\begin{corollary}\label{cor:adjunction_of_dg_functors_becomes_those_of_quasi-functors}
	The lax functor $\Gamma\colon\Bimod\to\DBimod$ preserves adjunctions of the form $F_\equipsubAst \dashv F^\equipsubAst$ for some dg functor $F$.
\end{corollary}

\begin{proof}
	By \cref{lem:when_the_constraint_is_invertible} the constraint $F^\equipsubAst \Dprocomp F_\equipsubAst \to F^\equipsubAst \procomp F_\equipsubAst$ (precisely, $\Gamma(F^\equipsubAst) \Dprocomp \Gamma(F_\equipsubAst) \to \Gamma(F^\equipsubAst \procomp F_\equipsubAst)$) of $\Gamma$ is an isomorphism, because $F_\equipsubAst$ is right h-flat. Hence \cref{lem:condition_lax_functor_preserves_adjunctions} proves the assertion.
\end{proof}

In other words, for a dg functor $F\colon \catA\to\catB$, the associated dg bimodule $F_\equipsubAst\colon \catA\slashedrightarrow\catB$ is left adjoint to the dg bimodule $F^\equipsubAst\colon \catB \slashedrightarrow\catA$ in $\DBimod$, as well as in $\Bimod$.
\Cref{lem:A_ast_is_left_adjoint_to_A^star} is a special case of this observation.

\section{The homotopy category theory of dg categories over a field}\label{section:the_homotopy_category_theory_of_dg_categories_over_a_field}

In this section we assume for simplicity that $\basek$ is a field, over which all dg categories are locally h-projective, though the following will work for locally h-projective dg categories over a commutative ring.

We will use straight arrows $\to$ for quasi-functors instead of slashed arrows $\slashedrightarrow$ to describe it being a morphism of $\DBimodrqr$.
A quasi-functor $f\colon \catJ\to \catA$ is written by $f_\equipsubStar$ when regarded as a morphism of $\DBimod$, and let $f^\equipsubStar$ be its right adjoint $\DIsbellL(f_\equipsubStar)$. Let us write the inclusion $\DBimodrqr\hookrightarrow \DBimod$ by $\equipStar$.

Note that a dg bimodule is isomorphic to $f_\equipsubStar$ for some quasi-functor $f$ if and only if it is right quasi-representable and that a dg bimodule is isomorphic to $f^\equipsubStar$ for some quasi-functor $f$ if and only if it is left quasi-representable.
In the following, for a right quasi-representable dg bimodule $X$, we mean by a \emph{representing quasi-functor} of $X$ a quasi-functor $f$ such that $f_\equipsubStar\cong X$. For a left quasi-representable $X$, we call a quasi-functor $f$ such that $f^\equipsubStar\cong X$ as a \emph{corepresenting quasi-functor} of $X$.

Now we state the following theorem, which captures a fact we have already observed.

\begin{theorem}\label{thm:proarrow_equipment_for_the_HCT_of_dg_cat}
	 The inclusion pseudo-functor $\equipStar\colon \DBimodrqr\hookrightarrow\DBimod$ is a \emph{proarrow equipment} in the sense of Wood (see \cref{def:proarrow_equipment}). That is, it satisfies:
	\begin{enumerate}
		\item $\equipStar$ is an identity on objects;
		\item $\equipStar$ is locally fully faithful;
		\item for every quasi-functor $f$, the dg bimodule $f_\equipsubStar$ has a right adjoint in $\DBimod$ (\cref{cor:quasi-functor_has_right_adjoint}).
	\end{enumerate}
\end{theorem}

A proarrow equipment serves as a general framework for formal category theory; references on this subject are \cite{Wood:1982proarrow1,Wood:1985proarrow2} (We review it in \cref{section:proarrow_equipment}).
Formal category theory abstracts and generalizes the basic notions of category theory using 2-categorical and axiomatic methods.
For example, the pseudo-functor $\equipAst\colon\dgCat \rightarrow \Bimod$, sending dg functors $F\colon \catA\to\catB$ to the dg bimodules $F_\equipsubAst=\catB(\mplaceholder,F\mplaceholder)\colon \catA\slashedrightarrow\catB$, forms a proarrow equipment, which provides $\Chk$-enriched category theory.
The above theorem indicates that we can develop formal category theory in $\DBimodrqr$.
We may name category theory done in the proarrow equipment $\equipStar\colon \DBimodrqr\hookrightarrow\DBimod$ as the \emph{homotopy category theory of dg categories}.

In a proarrow equipment, we can define several notions in category theory, such as (co)limits, Cauchy completeness, and pointwise Kan extensions.
Let us apply them to the homotopy category theory of dg categories.
In the following, we translate these notions to fit our context. For the definitions in the general setting, refer to \cref{section:proarrow_equipment} or \cite{Wood:1982proarrow1}.

\subsection{Homotopical (co)limits in a dg category}\label{subsection:homotopical_(co)limits_in_a_dg_category}

In this subsection, we define homotopical (co)limits in a dg category, including homotopical initial and terminal objects, (co)products, shifts, and cones. They are not just conical but weighted ones, as weighted (co)limits are the right notion of (co)limits in enriched category theory.

\begin{definition}\label{def:homotopical_colimits_in_a_dg_category}
	For a dg bimodule $W\colon \catM\slashedrightarrow \catJ$ and a quasi-functor $f\colon \catJ\to \catA$, a \emph{$W$-weighted colimit} of $f$ is a quasi-functor $\colim^W f \colon \catM \to \catA$ together with a $2$-cell in $\DBimod$
	\[
		\begin{tikzcd}
			& \catM \arrow{d}{W}[sloped,marking]{\mapstochar}
			\arrow[Rightarrow,d,"\iota",pos=0.7,shift right=3.4ex,shorten <=1.7ex,shorten >=-0.3ex] \\
			\catA \arrow[bend left=15]{ru}{(\colim^W f)^\equipsubStar}[sloped,marking]{\mapstochar} \arrow{r}[swap]{f^\equipsubStar}[sloped,marking]{\mapstochar} & \catJ
		\end{tikzcd}
	\]
	such that $\iota$ is a right Kan lifting of $f^\equipsubStar$ along $W$. In other words, when the right Kan lifting $\DRift_{W} f^\equipsubStar$ in $\DBimod$ is left quasi-representable, we call its corepresenting quasi-functor as a \emph{colimit}.

	Dually, for a dg bimodule $V\colon \catJ\slashedrightarrow \catM$ and a quasi-functor $f\colon \catJ\to \catA$, a \emph{$V$-weighted limit} of $f$ is a quasi-functor $\lim^V f \colon \catM \to \catA$ together with a $2$-cell in $\DBimod$
	\[
		\begin{tikzcd}
			\catM \arrow[bend left=15]{rd}{(\lim^V f)_\equipsubStar}[sloped,marking]{\mapstochar}
			\arrow[Rightarrow,shift left=3.4ex,shorten <=1.7ex,shorten >=-0.3ex,"\pi",swap,pos=0.7]{d} & \\
			\catJ \arrow{u}{V}[sloped,marking]{\mapstochar} \arrow{r}[swap]{f_\equipsubStar}[sloped,marking]{\mapstochar} & \catA
		\end{tikzcd}
	\]
	such that $\pi$ is a right Kan extension of $f_\equipsubStar$ along $V$. In other words, when the right Kan extension $\DRan_{V} f_\equipsubStar$ in $\DBimod$ is right quasi-representable, we call its representing quasi-functor as a \emph{limit}.
\end{definition}

The notion of (co)limit in the sense of \cref{def:homotopical_colimits_in_a_dg_category} is often referred to as an \emph{h-(co)limit}, to distinguish it from strict enriched (co)limits.

\begin{remark}
	As observed in \cref{prop:DBimod_is_closed_bicategory}, the right Kan lifting and Kan extension in $\DBimod$ are given as
	\[ \DRift_Y Z = Y_{\Dddag} Z = (\hp Y)_{\ddag} Z, \quad \DRan_X Z = X^{\Dddag} Z=(\hp X)^\ddag Z. \]
	Hence, a $W$-weighted colimit of $f$ exists precisely when for all $m \in \catM$ the left dg $\catA$-module
	\begin{align*}
		(\DRift_W f^\equipsubStar)(m,\mplaceholder) & = ((\hp W)_\ddag f^\equipsubStar) (m,\mplaceholder)                                   \\
		                                            & = \Fun_\dg(\catJ^\op,\Comdgk)\big((\hp W)(?,m), f^\equipsubStar(?,\mplaceholder)\big)
	\end{align*}
	is quasi-corepresentable, or more explicitly, there are an object $(\colim^W f)(m)$ of $\catA$ and a quasi-isomorphism
	\[ \catA\big((\colim^W f)(m),\mplaceholder\big) \to \Fun_\dg(\catJ^\op,\Comdgk)\big((\hp W)(?,m), f^\equipsubStar(?,\mplaceholder)\big) \]
	of left dg $\catA$-modules. Dually, a $V$-weighted limit of $f$ exists precisely when for all $m \in \catM$ the right dg $\catA$-module
	\begin{align*}
		(\DRan_V f_\equipsubStar)(\mplaceholder,m) & = ((\hp V)^\ddag f_\equipsubStar) (\mplaceholder,m)                               \\
		                                           & = \Fun_\dg(\catJ,\Comdgk)\big((\hp V)(m,?), f_\equipsubStar(\mplaceholder,?)\big)
	\end{align*}
	is quasi-representable, or more explicitly, there are an object $(\lim^V f)(m)$ of $\catA$ and a quasi-isomorphism
	\[ \catA\big(\mplaceholder,(\lim\nolimits^V f)(m)\big) \to \Fun_\dg(\catJ,\Comdgk)\big((\hp V)(m,?), f_\equipsubStar(\mplaceholder,?)\big) \]
	of right dg $\catA$-modules.
\end{remark}

% 例:initial object, coproduct, shift (tensor), coneのweight

The base field $\basek$ is identified with the dg category with one object. We here give some examples of (co)limits which are module-weighted, instead of bimodule-weighted ones; in other words, we shall consider the case of $\catM=\basek$.

Let $\catA$ be a dg category. Let us first consider the homotopical version of initial objects and coproducts.

%%% initial object

\begin{example}[h-initial object]\label{example:h-initial_object}
	Consider $\catJ =\emptyset$, the empty dg category, and the quasi-functor $f=F_\equipsubAst\colon \emptyset\to\catA$ induced by the unique dg functor $F \colon \emptyset\to \catA$. Note that $f_\equipsubStar=F_\equipsubAst$ is h-projective since $\Comdg(\emptyset,\catA)\cong \Fun_\dg(\emptyset,\Comdgk)\cong 0$, and we have $f^\equipsubStar=\bbL\IsbellL(F_\equipsubAst)=\hp(F_\equipsubAst)_\ddag I_\catA= (F_\equipsubAst)_\ddag I_\catA \cong F^\equipsubAst$.

	Let $\Winit\colon \basek\slashedrightarrow \emptyset$ be the unique functor $\emptyset^\op\otimes \basek\cong \emptyset \to \Comdgk$, which is h-projective in $\Comdg(\emptyset,\basek)\cong \Fun_\dg(\emptyset,\Comdgk)\cong 0$.
	In this case, we have an isomorphism of left dg modules
	\[ (\DRift_{\Winit} f^\equipsubStar)(\mplaceholder) = \Fun_\dg(\emptyset^\op,\Comdgk)\big(\Winit(?), F^\equipsubAst(?,\mplaceholder)\big) \cong 0. \]
	Hence the colimit $\colim^{\Winit} f$ exists if and only if there is an object $i\in \catA$ such that $\catA(i,\mplaceholder)$ is quasi-isomorphic to $0$, or equivalently $\catA(i,\mplaceholder) \cong 0$ in $\Dom(\catA^\op)$. Such a colimit is called an \emph{h-initial object} of $\catA$.

	Similarly, let $\Vterm \colon \emptyset\slashedrightarrow\basek$ be the unique bimodule $\basek^\op\otimes\emptyset\cong \emptyset\to \Comdgk$. Then
	\[ (\DRan_\Vterm f_\equipsubStar)(\mplaceholder) \cong 0 \]
	and the limit $\lim^\Vterm f$ exists if and only if there is an object $t \in \catA$ such that $\catA(\mplaceholder,t)$ is quasi-isomorphic to $0$, or equivalently $\catA(\mplaceholder,t) \cong 0$ in $\Dom(\catA)$. Such a limit is called an \emph{h-terminal object} of $\catA$.
\end{example}

A dg category $\catA$ is said to have a \emph{contractible object} if $H^0(\catA)$ has a zero object.

\begin{proposition}
	Let $\catA$ be a dg category.
	\begin{enumerate}
		\item $\catA$ has an h-initial object if and only if it has a contractible object.
		\item $\catA$ has an h-terminal object if and only if it has a contractible object.
	\end{enumerate}
\end{proposition}

\begin{proof}
	(1) If $i\in\catA$ is h-initial, then we have $H^0(\catA)(i,\mplaceholder) = H^0(\catA(i,\mplaceholder))\cong 0$, which shows that $i$ is a zero object in $H^0(\catA)$. Conversely, suppose $\catA$ has a contractible object denoted by $i$. Then it is preserved by the additive functor $H^0(\yoneda)\colon H^0(\catA^\op)\to H^0(\Comdg(\catA^\op))=\Kom(\catA^\op)$, and so we have $\catA(i,\mplaceholder) \cong 0$ in $\Kom(\catA^\op)$, and hence in $\Dom(\catA^\op)$. (2) Similar.
\end{proof}

Let $S(\basek)$ denote $\basek$ viewed as a complex concentrated in degree $0$, which is called the \emph{sphere complex}.

%%% coproduct

\begin{example}[h-coproduct]\label{example:h-coproduct}
	Consider a set $J$ and identify it with the free dg category $\coprod_J \basek$ generated by the discrete category $J$.
	Giving a dg functor $F \colon J\to \catA$ is equivalent to giving a family $\{A_j\}_{j \in J}$ of objects. Let $f=F_\equipsubAst\colon J\to\catA$ be the quasi-functor induced by $F$. Note that $f_\equipsubStar=F_\equipsubAst$ is h-projective since it corresponds to $(\catA(\mplaceholder,A_j))_j$ under the isomorphisms $\Comdg(J,\catA)\cong \Fun_\dg(J,\Comdg(\catA))\cong \prod_J \Comdg(\catA)$, and hence we have $f^\equipsubStar=\bbL\IsbellL(F_\equipsubAst)\cong F^\equipsubAst$.

	Let $\Wcoprod\colon \basek\slashedrightarrow J$ be the dg functor $J^\op\otimes \basek\cong J \to \Comdgk$ corresponding to $(S(\basek))_{j} \in \prod_J\Comdgk$, which is h-projective.
	In this case, we have isomorphisms of left dg modules
	\begin{align*}
		(\DRift_\Wcoprod f^\equipsubStar)(\mplaceholder) & = \Fun_\dg(J^\op,\Comdgk)\big(\Wcoprod(?), F^\equipsubAst(?,\mplaceholder)\big)            \\
		                                          & \cong \mleft(\prod\nolimits_j\Comdgk\mright)  \Big((S(\basek))_j, (\catA(A_j,\mplaceholder))_j\Big) \\
		                                          & \cong \prod\nolimits_j \catA(A_j,\mplaceholder).
	\end{align*}
	Hence the colimit $\colim^\Wcoprod f$ exists if and only if there is an object $C \in \catA$ such that $\catA(C,\mplaceholder)$ is quasi-isomorphic to $\prod_j \catA(A_j,\mplaceholder)$, or equivalently $\catA(C,\mplaceholder) \cong \prod_j \catA(A_j,\mplaceholder)$ in $\Dom(\catA^\op)$. Such a colimit is called an \emph{h-coproduct} of $\{A_j\}_{j \in J}$.

	Similarly, let $\Vprod \colon J \slashedrightarrow \basek$ be the bimodule corresponding to $(S(\basek))_{j} \in \prod_J\Comdgk$. Then
	\[ (\DRan_\Vprod f_\equipsubStar)(\mplaceholder) \cong \prod\nolimits_j \catA(\mplaceholder,A_j) \]
	and the limit $\lim^\Vprod f$ exists if and only if there is an object $D \in \catA$ such that $\catA(\mplaceholder,D)$ is quasi-isomorphic to $\prod_j \catA(\mplaceholder,A_j)$, or equivalently $\catA(\mplaceholder,D) \cong \prod_j \catA(\mplaceholder,A_j)$ in $\Dom(\catA)$. Such a limit is called an \emph{h-product} of $\{A_j\}_{j\in J}$.
\end{example}

Notice that products in $\Com(\catA)$ becomes products in $\Dom(\catA)$.

\begin{proposition}\label{prop:characterization_of_h-coproduct}
	Let $\catA$ be a dg category and $\{A_j\}_j$ a family of objects.
	\begin{enumerate}
		\item $\catA$ has an h-coproduct of $\{A_j\}_j$ if and only if $H^0(\catA)$ has a coproduct of $\{A_j\}_j$ which is preserved by the functor $H^0(\catA)^\op=H^0(\catA^\op) \to \Dom(\catA^\op)$.
		\item $\catA$ has an h-product if and only if $H^0(\catA)$ has a product of $\{A_j\}_j$ which is preserved by the functor $H^0(\catA) \to \Dom(\catA)$.
	\end{enumerate}
\end{proposition}

\begin{proof}
	(1) If $C\in\catA$ is an h-coproduct of $\{A_j\}_j$, then we have $H^0(\catA)(C,\mplaceholder) = H^0(\catA(C,\mplaceholder))\cong  H^0\mleft( \prod_j\catA(A_j,\mplaceholder) \mright) \cong \prod_j H^0(\catA)(A_j,\mplaceholder)$, which shows that $C$ is a coproduct of $\{A_j\}_j$ in $H^0(\catA)$.
	The converse is immediate. (2) Similar.
\end{proof}

%% idempotent splitting

%\begin{example}[idempotent splitting]\label{example:h-idempotent_splitting}
%\end{example}

Before proceeding to homotopical shifts and cones, we recall ordinary shifts and cones for complexes.

\begin{remark}\label{rmk:def_of_S(k)_and _D(k)_plus_cone_as_weighted_limit}
	Recall $S(\basek)$ denotes the sphere complex.
	We define the \emph{disk complex} $D(\basek)$ to be such a complex that $D(\basek)^i=0$ for $i\neq -1,0$ and the non-trivial differential $d^{-1}\colon D(\basek)^{-1}\to D(\basek)^0$ is the identity $\id\colon \basek\to\basek$; these complexes look like:
	\begin{alignat*}{2}
		S(\basek)\colon \qquad & \cdots\to 0 \to 0 \to \basek \to 0 \to \cdots,                     & \qquad & \\
		D(\basek)\colon \qquad & \cdots \to 0 \to \basek \xrightarrow{\id} \basek \to 0 \to \cdots. &        &
	\end{alignat*}
	We have the canonical injection $\iota\colon S(\basek) \hookrightarrow D(\basek)$. Notice that $D(\basek)$ is equal to the cone $\Cone(\id\colon S(\basek)\to S(\basek))$ of the identity.

	Put $S^n(\basek)\coloneqq S(\basek)[-n]$. Then it is easy to see that there is an isomorphism of complexes
	\[ \Homcpx(S^n(\basek),X) = \Homcpx(S(\basek)[-n],X) \cong X[n] \]
	for a complex $X$ of $\basek$-modules.

	Also, write the poset $\{0<1\}$ and its free dg category by the same symbol $\warrcat$. Let $W\colon \warrcat\to \Comdgk$ be the dg functor which chooses the morphism $\iota\colon S(\basek) \hookrightarrow D(\basek)$ in $\Chk$. Then we can verify
	\[ \Fun_\dg(\warrcat,\Comdgk)(W,F) \cong \Cone(f)[-1] = \Cone(-f[-1]) \eqqcolon \Cocone(f) \]
	for a dg functor $F\colon \warrcat\to \Comdg$ which corresponds to a morphism $f\colon X\to Y$ in $\Chk$.
\end{remark}

%%% tensor product

\begin{example}[h-tensor and h-coshift]\label{example:h-tensor}
	Consider $\catJ=\basek$ and the quasi-functor $f=F_\equipsubAst\colon \basek\to\catA$ induced by an object $A \colon \basek \to \catA$ . Note that $f_\equipsubStar=F_\equipsubAst$ is h-projective since it corresponds to $\catA(\mplaceholder,A)$ under $\Comdg(\basek,\catA)\cong \Comdg(\catA)$, and hence we have $f^\equipsubStar=\bbL\IsbellL(F_\equipsubAst)\cong F^\equipsubAst$.

	Let $S\in\Comdgk$ be a complex and $W\colon \basek\slashedrightarrow \basek$ be the corresponding dg bimodule $\basek^\op\otimes \basek\cong \basek \xrightarrow{S} \Comdgk$. In this case, we have an isomorphism of left dg modules
	\begin{align*}
		(\DRift_W f^\equipsubStar)(\mplaceholder)
		 & = \Fun_\dg(\basek^\op,\Comdgk)\big((\hp W)(?), F^\equipsubAst(?,\mplaceholder)\big) \\
		 & \cong {\Homcpx}\big(\hp S, \catA(A,\mplaceholder)\big).
	\end{align*}
	Hence the colimit $\colim^W f$ exists if and only if there is an object $C \in \catA$ such that $\catA(C,\mplaceholder)$ is quasi-isomorphic to ${\Homcpx}\big(\hp S, \catA(A,\mplaceholder)\big)$, or equivalently $\catA(C,\mplaceholder) \cong {\Homcpx}\big(\hp S, \catA(A,\mplaceholder)\big)$ in $\Dom(\catA^\op)$. Such a colimit is called an \emph{h-tensor} of $A$ with $S$.

	In particular, let us take $S= S^n(\basek)\coloneqq S(\basek)[-n]$, the $(-n)$-shift of the sphere complex (see \cref{rmk:def_of_S(k)_and _D(k)_plus_cone_as_weighted_limit}).
	%; the corresponding weight will be written by $\Wcoshift{n}$.
	Then it is h-projective and we have an isomorphism of complexes
	\[ (\DRift_W f^\equipsubStar)(\mplaceholder)\cong {\Homcpx}\big(S^n(\basek), \catA(A,\mplaceholder)\big) \cong \catA(A,\mplaceholder)[n]. \]
	Therefore the h-tensor of $A$ with $S^n(\basek)$ is called a \emph{homotopical $n$-coshift} of $A$. 
	We write the weight for h-$n$-coshifts by $\Wcoshift{n}$.
	%We will write the dg bimodule corresponding to $S^n(\basek)$ by $\Wcoshift{n}\colon \basek\slashedrightarrow\basek$.

	%the h-tensor of $A$ with $S^n(\basek)$ exists if and only if there is an object $C\in \catA$ such that $\catA(C,\mplaceholder)$ is quasi-isomorphic to $\catA(A,\mplaceholder)[n]$, or equivalently $\catA(C,\mplaceholder) \cong \catA(A,\mplaceholder)[n]$ in $\Dom(\catA^\op)$. Such an h-tensor is called an \emph{h-$n$-coshift} of $A$.

	Similarly, let $V \colon \basek\slashedrightarrow \basek$ be the dg bimodule corresponding to $S\in\Comdgk$. Then
	\[ (\DRan_V f_\equipsubStar)(\mplaceholder) \cong {\Homcpx}\big(\hp S, \catA(\mplaceholder,A)\big) \]
	and the limit $\lim^V f$ exists if and only if there is an object $D \in \catA$ such that $\catA(\mplaceholder,D)$ is quasi-isomorphic to ${\Homcpx}\big(\hp S, \catA(\mplaceholder,A)\big)$, or equivalently $\catA(\mplaceholder,D) \cong {\Homcpx}\big(\hp S, \catA(\mplaceholder,A)\big)$ in $\Dom(\catA)$. Such a limit is called an \emph{h-cotensor} of $A$ with $S$. In particular, the h-cotensor of $A$ with $S^n(\basek)$ is called a \emph{homotopical $n$-shift} of $A$.
	We write the weight for h-$n$-shifts by $\Vshift{n}$.
	%We will also write the dg bimodule corresponding to $S^n(\basek)$ by $\Vshift{n}\colon \basek\slashedrightarrow\basek$.
\end{example}

Recall that $H^0(\catA)$ can be regarded as a (non-replete) full subcategory of the triangulated category $\Dom(\catA)$ via $H^0(\yoneda)\colon H^0(\catA) \hookrightarrow \Dom(\catA)$.

\begin{proposition}
	Let $\catA$ be a dg category.
	\begin{enumerate}
		\item $\catA$ has all homotopical $n$-coshifts if and only if the subcategory $H^0(\catA^\op) \subseteq \Dom(\catA^\op)$ is closed (up to isomorphism) under $n$-shifts.
		\item $\catA$ has all homotopical $n$-shifts if and only if the subcategory $H^0(\catA) \subseteq \Dom(\catA)$ is closed (up to isomorphism) under $n$-shifts.
	\end{enumerate}
\end{proposition}

\begin{proof}
	Straightforward.
\end{proof}

%%% cone

Recall that $\warrcat$ denotes the poset $\{ 0< 1 \}$ and is identified with its free dg category. Giving a dg functor $F\colon \warrcat \to \catA$ is equivalent to giving a morphism $a\colon A \to B$ of $Z^0(\catA)$. We now prepare a lemma.

\begin{lemma}\label{lem:weight_for_cone_is_h-projective}
	Let $W\colon \warrcat\to\Comdgk$ be the dg functor corresponding to the morphism $\iota\colon S(\basek) \to D(\basek)$ from the sphere complex to the disk complex (see \cref{rmk:def_of_S(k)_and _D(k)_plus_cone_as_weighted_limit}).
	\begin{enumerate}
		\item $W$ is h-projective in $\Fun_\dg(\warrcat,\Comdgk)=\Comdg(\warrcat^\op)$.
		\item $W$ is compact in $\Dom(\warrcat^\op)$.

		      %\item Let $\catA$ be a dg category and $a\colon A \to B$ a morphism of $Z^0(\catA)$. Then the dg functor $\warrcat \xrightarrow{a} \catA \hookrightarrow \Comdg(\catA)$ is h-projective in $\Fun_\dg(\warrcat,\Comdg(\catA))\cong \Comdg(\warrcat,\catA)$.\mymemo{（成り立たないかも）}
	\end{enumerate}
\end{lemma}

\begin{proof}
	(1) We observe that the representable dg functor $\Hom(0,-)\colon \warrcat\to\Comdgk$ for the object $0\in \warrcat$ corresponds to the identity morphism $\id\colon S(\basek) \to S(\basek)$, and $\Hom(1,-)\colon \warrcat\to\Comdgk$ corresponds to the zero morphism $0 \to S(\basek)$.
	It is then easy to see that $W$ is the cone of $\Hom(1,-)\to \Hom(0,-)$ in $\Fun_\dg(\warrcat,\Comdgk)$. Since representable dg functors are h-projective, so is $W$.

	(2) As seen in \cref{rmk:def_of_S(k)_and _D(k)_plus_cone_as_weighted_limit}, we have $\Fun_\dg(\warrcat,\Comdgk)(W,F) \cong \Cone(f)[-1]$ for a dg functor $F\colon \warrcat\to \Comdg$. Thus by using h-projectiveness of $W$, we have
	\begin{align*}
		\Hom_{\Dom(\warrcat^\op)}(W,-)
		 & \cong \Hom_{\Kom(\warrcat^\op)}(W,-)                 \\
		 & \cong H^0\Big( \Fun_\dg(\warrcat,\Comdgk)(W,-) \Big) \\
		 & \cong H^0(\Cone(-)[-1]),
	\end{align*}
	which preserves coproducts. Therefore $W$ is compact in $\Dom(\warrcat^\op)$.
	%Let $F$ be the dg functor $\warrcat \xrightarrow{a} \catA \hookrightarrow \Comdg(\catA)$. 	It can be checked by a direct calculation that 	\[ \Fun_\dg(\warrcat,\Comdg(\catA))(F,N) \cong  \]
\end{proof}

\begin{example}[h-cone]\label{example:h-cocone}
	Consider $\catJ=\warrcat$ and a morphism $a\colon A\to B$ in $Z^0(\catA)$. Let $f=a_\equipsubAst\colon \warrcat\to\catA$ be the quasi-functor induced by $a\colon \warrcat\to \catA$. 
	By \cref{cor:adjunction_of_dg_functors_becomes_those_of_quasi-functors}, $a^\equipsubAst$ is the right adjoint of $f_\equipsubStar=a_\equipsubAst$ in $\DBimod$, and hence $f^\equipsubStar=a^\equipsubAst$.
	%Also let $g^\equipsubStar=a^\equipsubAst\colon \catA\slashedrightarrow\warrcat$ be the left representable dg bimodule induced by $a\colon \warrcat\to \catA$ and $g_\equipsubStar\colon \warrcat\slashedrightarrow \catA$ be its left adjoint dg bimodule, which we view as a quasi-functor $g\colon \warrcat\to\catA$.

	Let $\Wcone\colon \basek \slashedrightarrow \warrcat$ be the dg bimodule $\warrcat^\op\otimes\basek \cong \warrcat^\op\to \Comdgk$ obtained from the morphism of complexes $\iota\colon S(\basek) \to D(\basek)$. In this case, $\Wcone$ is h-projective and for the quasi-functor $f$ we have an isomorphism of left dg modules
	\begin{align*}
		(\DRift_\Wcone f^\equipsubStar)(\mplaceholder)
		 & = \Fun_\dg(\warrcat^\op,\Comdgk)\big(\Wcone(?), a^\equipsubAst(?,\mplaceholder)\big)               \\
		 & \cong \Cone\big(a^\equipsubAst\colon \catA(B,\mplaceholder)\to\catA(A,\mplaceholder)\big)[-1]
		%= \Cone(-a^\equipsubAst[-1])
		\eqqcolon \Cocone(a^\equipsubAst).
	\end{align*}
	Hence the colimit $\colim^\Wcone g$ exists if and only if there is an object $C \in \catA$ such that $\catA(C,\mplaceholder)$ is quasi-isomorphic to $\Cocone(a^\equipsubAst)$, or equivalently $\catA(C,\mplaceholder) \cong \Cocone(a^\equipsubAst)$ in $\Dom(\catA^\op)$. Such a colimit is called an \emph{h-cone} of $a\colon A\to B$.
	In general, we refer to the colimit of a quasi-functor $f\colon \warrcat\to\catA$ weighted by the above dg bimodule $\Wcone$ as an \emph{h-cone} of $f$.

	Similarly, let $\Vcocone \colon \warrcat\slashedrightarrow \basek$ be the dg bimodule $\basek^\op\otimes\warrcat\cong\warrcat\to\Comdgk$ obtained from the morphism of complexes $\iota\colon S(\basek) \to D(\basek)$. Then
	\begin{align*}
		(\DRan_\Vcocone f_\equipsubStar)(\mplaceholder)
		 & = \Fun_\dg(\warrcat,\Comdgk)\big(\Vcocone(?), a_\equipsubAst(?,\mplaceholder)\big)                   \\
		 & \cong \Cone\big(a_\equipsubAst\colon \catA(\mplaceholder,A)\to\catA(\mplaceholder,B)\big)[-1]
		%= \Cone(-a_\equipsubAst[-1])
		\eqqcolon \Cocone(a_*),
	\end{align*}
	and hence the limit $\lim^\Vcocone f$ exists if and only if there is an object $D \in \catA$ such that $\catA(\mplaceholder,D)$ is quasi-isomorphic to $\Cocone(a_*)$, or equivalently $\catA(\mplaceholder,D) \cong \Cocone(a_*)$ in $\Dom(\catA)$. Such a limit is called an \emph{h-cocone} of $a\colon A \to B$.
	In general, we refer to the limit of a quasi-functor $f\colon \warrcat\to\catA$ weighted by the above dg bimodule $\Vcocone$ as an \emph{h-cocone} of $f$.
\end{example}

Take a quasi-functor $f\colon \warrcat\to \catA$. As it is right quasi-representable, the right dg $\catA$-modules $f_\equipsubStar(\mplaceholder,i)$ for $i=0,1$ are quasi-representable, i.e.\ isomorphic to $\catA(\mplaceholder,A_i)$ for some $A_i\in \catA$ in the derived category $\Dom(\catA)$. Then we refer to $f$ as a \emph{quasi-morphism} form $A_0$ to $A_1$. A quasi-morphism is just a morphism of $\qucl{\catA}$.

\begin{proposition}
	Let $\catA$ be a dg category.
	\begin{enumerate}
		\item $\catA$ has all h-cones if and only if the subcategory $H^0(\catA^\op) \subseteq \Dom(\catA^\op)$ is closed (up to isomorphism) under cocones.
		\item $\catA$ has all h-cocones if and only if the subcategory $H^0(\catA) \subseteq \Dom(\catA)$ is closed (up to isomorphism) under cocones.
	\end{enumerate}
\end{proposition}

\begin{proof}
	It is easy to check by definition.
\end{proof}

In particular, we obtain a formal characterization of being pretriangulated.

\begin{corollary}\label{cor:pretriangulated_iff_homotopy_cone_and_homotopy_shift}
	A dg category $\catA$ has all h-shifts and h-cocones if and only if it is pretriangulated, i.e.\ $H^0(\catA)$ is closed (up to isomorphism) under shifts and cones of $\Dom(\catA)$.
\end{corollary}

We have seen some examples of module-weighted colimits. 
We can reduce the existence of bimodule-weighted colimits to that of module-weighted ones.

\begin{proposition}\label{prop:reduction_of_bimodule-weighted_colimit_to_module-weighted_one}
	Let $W\colon \catM\slashedrightarrow\catJ$ be a dg bimodule and $f\colon \catJ \to \catA$ a quasi-functor. Then the colimit $\colim^W f$ exists if and only if the colimits $\colim^{W(\mplaceholder,M)} f$ exist for all $M \in \catM$.
\end{proposition}

\begin{proof}
	We may assume $W$ to be h-projective. Notice that $W(\mplaceholder,M) \cong W \procomp M_\equipsubAst \cong W \Dprocomp M_\equipsubAst$ in $\Dom(\catJ)$ since $M_\equipsubAst$ is h-projective. The colimit $\colim^W f$ exists if and only if $\DRift_W f^\equipsubStar$ is left quasi-representable, i.e.\ $(\DRift_W f^\equipsubStar)(M,\mplaceholder)$ is quasi-representable for each $M \in \catM$. Because $M^\equipsubAst=\DRift_{M_\equipsubAst} I_\catM$ is an absolute right Kan lifting in $\DBimod$ (see \cref{lem:A_ast_is_left_adjoint_to_A^star}), we have $M^\equipsubAst \Dprocomp \DRift_W f^\equipsubStar \cong  \DRift_{M_\equipsubAst} \DRift_W f^\equipsubStar \cong \DRift_{W\Dprocomp M_\equipsubAst} f^\equipsubStar \cong \DRift_{W(\mplaceholder,M)} f^\equipsubStar$. Thus $(\DRift_W f^\equipsubStar)(M,\mplaceholder)\cong M^\equipsubAst\Dprocomp \DRift_W f^\equipsubStar$ is quasi-representable if and only if the colimit $\colim^{W(\mplaceholder,M)} f$ exists, which proves the proposition.
\end{proof}

We can furthermore reduce homotopy colimits to h-coshifts, h-cones, and h-coproducts in the following way.

\begin{proposition}\label{prop:h-cocomplete_iff_having_coshifts_cones_coproducts}
	A dg category $\catA$ has all weighted colimits if and only if it has all h-coshifts, h-cones, and h-coproducts.
\end{proposition}

\begin{proof}
	The necessity is obvious. To prove the sufficiency, assume that $\catA$ has h-coshifts, h-cones, and h-coproducts.
	By \cref{prop:reduction_of_bimodule-weighted_colimit_to_module-weighted_one}, we only need to show that $\catA$ has the $W$-weighted colimits of a quasi-functor $f\colon \catJ\to\catA$ for all modules $W\colon \basek\slashedrightarrow\catJ$. Consider the full subcategory $\catC\subseteq \Dom(\basek,\catJ)=\Dom(\catJ)$ consisting of those $W$ for which the colimit $\colim^W f$ exists, i.e.,
	\[ \catC=\{ W\in\Dom(\catJ) \mid \text{$\DRift_W f^\equipsubStar\in\Dom(\catA,\basek)$ is left quasi-representable} \}. \]
	Then let us verify that $\catC$ is a localizing triangulated subcategory of $\Dom(\catJ)$ containing all representable modules.
	\begin{itemize}
		\item For $j\in\catJ$, recall from \cref{lem:A_ast_is_left_adjoint_to_A^star} that  $j_\equipsubAst=\catJ(\mplaceholder,j)$ is left adjoint to $j^\equipsubAst$ in $\DBimod$. Hence we have $\DRift_{j_\equipsubAst} f^\equipsubStar \cong j^\equipsubAst\Dprocomp f^\equipsubStar\cong f^\equipsubStar(j,\mplaceholder)$, which is quasi-representable because $f^\equipsubStar$ is left quasi-representable. Therefore $j_\equipsubAst\in \catC$.
		
		\item For $W\in \catC$, we show that $W[-n]\in \catC$ for $n\in \Z$. Noting that $W[-n] \cong X\Dprocomp \Wcoshift{n}$ in $\Dom(\basek,\catJ)$ from \cref{rmk:summury_on_h-shift_h-cone_h-coproduct} below, we have
		\[ \DRift_{W[-n]} f^\equipsubStar \cong \DRift_{W\Dprocomp \Wcoshift{n}} f^\equipsubStar \cong \DRift_{\Wcoshift{n}} \DRift_{W} f^\equipsubStar. \]
		Now, $\DRift_{W} f^\equipsubStar$ is left quasi-representable. Since $\catA$ has h-coshifts, $\DRift_{\Wcoshift{n}} \DRift_{W} f^\equipsubStar$, and hence $\DRift_{W[-n]} f^\equipsubStar$, is left quasi-representable. Thus $W[-n]\in\catC$.
		
		\item For a morphism $\theta\colon W\to W'$ in $\catC$, we show that $\Cone(\theta)\in \catC$. We may assume that $\theta$ is represented by a morphism $W\to W'$ of $\Comdg(\catJ)$, which we write by the same symbol $\theta$. Then $\theta$ determines a dg functor $\warrcat\to\Comdg(\catJ)$ and we identify it with the corresponding dg bimodule $\theta\colon \warrcat\slashedrightarrow\catJ$. Noting that $\Cone(\theta) \cong \theta \Dprocomp_{\warrcat} \Wcone$ in $\Dom(\basek,\catJ)$ from \cref{rmk:summury_on_h-shift_h-cone_h-coproduct} below, we have
		\[ \DRift_{\Cone(\theta)} f^\equipsubStar \cong \DRift_{\theta \Dprocomp_{\warrcat} \Wcone} f^\equipsubStar \cong \DRift_{\Wcone} \DRift_{\theta} f^\equipsubStar. \]
		Because $W=\theta(\mplaceholder,0)$ and $W'=\theta(\mplaceholder,1)$ are in $\catC$ and so $\DRift_{\theta(\mplaceholder,0)} f^\equipsubStar$ and $\DRift_{\theta(\mplaceholder,1)} f^\equipsubStar$ are left quasi-representable, $\DRift_{\theta} f^\equipsubStar$ is also left quasi-representable, as in the proof of \cref{prop:reduction_of_bimodule-weighted_colimit_to_module-weighted_one}. Since $\catA$ has h-cones, $\DRift_{\Wcone} \DRift_{\theta} f^\equipsubStar$, and hence $\DRift_{\Cone(\theta)} f^\equipsubStar$, is left quasi-representable. Thus $\Cone(\theta)\in\catC$.
		
		\item For a family $\{W_i\}_{i\in I}$ of objects of $\catC$, we show that $\coprod_i W_i \in \catC$. The family $\{W_i\}_i$ determines a dg functor $I \to \Comdg(\catJ)$ and we identify it with the corresponding dg bimodule $(W_i)_i \colon I \slashedrightarrow\catJ$. Noting that $\coprod_i W_i \cong (W_i)_i \Dprocomp_{I} \Wcoprod$ in $\Dom(\basek,\catJ)$ from \cref{rmk:summury_on_h-shift_h-cone_h-coproduct} below, we have
		\[ \DRift_{\coprod_i W_i} f^\equipsubStar \cong \DRift_{(W_i)_i \Dprocomp_{I} \Wcoprod} f^\equipsubStar \cong \DRift_{\Wcoprod} \DRift_{(W_i)_i} f^\equipsubStar. \]
		Now, by \cref{prop:reduction_of_bimodule-weighted_colimit_to_module-weighted_one}, $\DRift_{(W_i)_i} f^\equipsubStar$ is left quasi-representable, as $W_i\in \catC$. Since $\catA$ has h-coproducts, $\DRift_{\Wcoprod} \DRift_{(W_i)_i} f^\equipsubStar$, and hence $\DRift_{\coprod_i W_i} f^\equipsubStar$, is left quasi-representable. Thus $\coprod_i W_i\in\catC$.
	\end{itemize}
	Recall that $\Dom(\catJ)$ is a triangulated category compactly generated by the representable modules. Therefore it follows from \cref{prop:Brown_representability_theorem} that $\catC=\Dom(\catJ)$, which implies that $\catA$ has all weighted colimits of $f$.
\end{proof}

Note that a dg bimodule $W\colon \basek\slashedrightarrow\catJ$ from the unit dg category is h-projective if and only if it is right h-projective.

\begin{remark}\label{rmk:summury_on_h-shift_h-cone_h-coproduct}
	\begin{itemize}
		\item The weight $\Wcoshift{n}\colon \basek\slashedrightarrow\basek$ for homotopical $n$-coshifts is given by $S^n(\basek)=S(\basek)[-n]$, which is (right) h-projective. In particular $\Wcoshift{n}$ is right h-flat, so we have
		\[ X\Dprocomp_\basek \Wcoshift{n} \cong X\procomp_\basek \Wcoshift{n} = S^n(\basek)\otimes_\basek X \cong X[-n] \]
		in $\Dom(\catJ)$ for a dg bimodule $X\colon \basek\slashedrightarrow\catJ$.
		
		\item The weight $\Wcone\colon \basek\slashedrightarrow\warrcat$ for h-cones is given by the dg functor $\warrcat^\op\to\Comdg(\basek)$ choosing the inclusion $\iota\colon S(\basek)\hookrightarrow D(\basek)$, which is (right) h-projective from \cref{lem:weight_for_cone_is_h-projective}. In particular $\Wcone$ is right h-flat, so we have
		\[ X\Dprocomp_\warrcat \Wcone \cong X\procomp_\warrcat \Wcone = \Wcone\otimes_\warrcat X \cong \Cone(\theta) \]
		in $\Dom(\catJ)$ (see \cite[Example 4.2.1]{Loregian:2021coend_calculus} for the last isomorphism). Here $X\colon \warrcat\slashedrightarrow\catJ$ is a dg bimodule and $\theta\colon \warrcat\to\Comdg(\catJ)$ is the corresponding dg functor.
		
		\item The weight $\Wcoprod\colon\basek\slashedrightarrow I$ for h-coproducts is given by the dg functor $(S(\basek))_i\colon I\to \Comdg(\basek)$, which is (right) h-projective. In particular $\Wcoprod$ is right h-flat, so we have
		\[ X\Dprocomp_I \Wcoprod \cong X\procomp_I \Wcoprod = (S(\basek))_i \otimes_I X \cong {\textstyle \coprod_i X_i} \]
		in $\Dom(\catJ)$. Here $X\colon I \slashedrightarrow\catJ$ is a dg bimodule and $X_i$ is the value of the corresponding dg functor $I\to\Comdg(\catJ)$ at $i\in I$.
	\end{itemize}
\end{remark}

\subsection{Absolute (co)limits and homotopy Cauchy complete dg categories}\label{subsection:homotopy_Cauchy_complete_dg_categories}
%% absolute colimit and Cauchy completeness

We can also define the preservation of (co)limits.

\begin{definition}
	Let $W\colon \catM\slashedrightarrow\catJ$ be a dg bimodule and $f\colon \catJ \to \catA$ a quasi-functor. Suppose the colimit $\colim^W f$ exists. We say that a quasi-functor $g\colon \catA \to \catB$ \emph{preserves} $\colim^W f$ if $\DRift_W (gf)^\equipsubStar \cong (\DRift_W f^\equipsubStar) \Dprocomp g^\equipsubStar = (\colim^W f)^\equipsubStar \Dprocomp g^\equipsubStar$. Dually we define the preservation of limits by a quasi-functor.
\end{definition}

\begin{example}\label{example:preservation_of_h-coproducts}
	Let $g\colon \catA\to \catB$ be a quasi-functor.
	Suppose that $W=\Wcoprod$ is the weight for h-coproducts and $f\colon J \to \catA$ corresponds to a family $\{A_j\}_j$ of objects of $\catA$. Since $g$ is a quasi-functor, $\bbLT_{g_\equipsubStar}((A_j)_\equipsubAst) =g_\equipsubStar(\mplaceholder,A_j)$ is quasi-representable and so isomorphic to $(B_j)_\equipsubAst$ for some $B_j\in \catB$.
	If $C\in\catA$ is the h-coproduct of $\{A_j\}_j$, then $\DRift_\Wcoprod f^\equipsubStar\cong C^\ast$ and hence $(\DRift_\Wcoprod f^\equipsubStar) \Dprocomp g^\equipsubStar \cong C^\ast\Dprocomp g^\equipsubStar\cong C^\ast\procomp g^\equipsubStar\cong g^\equipsubStar(C,\mplaceholder)$. 
	Since $\bbLT_{g_\equipsubStar}(C_\equipsubAst) =g_\equipsubStar(\mplaceholder,C)$ is quasi-representable, there exists $D\in\catB$ such that $g_\equipsubStar(\mplaceholder,C)\cong D_\equipsubAst$ and $g^\equipsubStar(C,\mplaceholder) \cong D^\equipsubAst$.
	Therefore $g$ preserves the h-coproduct $C$ if and only if $D$ is the h-coproduct of $\{B_j\}_j$ in $\catB$, i.e., $D$ is the coproduct of $\{B_j\}_j$ in $H^0(\catB)$ which is preserved by the embedding $H^0(\catB)^\op \hookrightarrow \Dom(\catB^\op)$ (see \cref{prop:characterization_of_h-coproduct}). In particular it follows that if $g$ preserves h-coproducts, then $\res{\bbLT_{g_\equipsubStar}}{\qucl{\catA}}\colon \qucl{\catA}\to\qucl{\catB}$ preserves coproducts.
\end{example}

\begin{proposition}
	Left adjoint quasi-functors preserve colimits. Dually, right adjoint quasi-functors preserve limits.
\end{proposition}

\begin{proof}
	It follows from \cref{prop:LAPC_in_a_proarrow_equipment}.
\end{proof}

As a corollary, h-(co)limits are invariant under quasi-equivalence because, according to \cref{prop:quasi-equivalence_is_just_equivalence_in_DBimodrqr}, quasi-equivalences become equivalences in $\DBimodrqr$, and equivalences are both left and right adjoints.

\begin{proposition}\label{prop:right_compact_weight_is_absolute}
	All existing colimits weighted by left adjoint dg bimodules are preserved by any quasi-functors.
	Dually, all existing limits weighted by right adjoint dg bimodules are preserved by any quasi-functors. Here, adjointness is understood in $\DBimod$.
\end{proposition}

\begin{proof}
	This is a consequence of \cref{prop:left_adjoint_weight_is_absolute}; however, we provide the proof here for the reader's convenience.
	Let $W\colon \catM\slashedrightarrow\catJ$ and $V\colon \catJ \slashedrightarrow\catM$ be dg bimodules, and suppose they form an adjunction $W\dashv V$ in $\DBimod$. Take a quasi-functor $f\colon \catJ \to \catA$ and assume the colimit $\colim^W f$ exists. Then $\DRift_W f^\equipsubStar$ is left quasi-representable. Since $\DRift_W f^\equipsubStar \cong V\Dprocomp f^\equipsubStar$ by \cref{prop:Kan_extension/lifting_along_adjoint}, we have
	\[ \DRift_W (gf)^\equipsubStar \cong \DRift_W (f^\equipsubStar \Dprocomp g^\equipsubStar) \cong V\Dprocomp f^\equipsubStar\Dprocomp g^\equipsubStar \cong (\DRift_W f^\equipsubStar)\Dprocomp g^\equipsubStar \]
	for any quasi-functor $g\colon \catA \to \catB$. It follows that $\DRift_W (gf)^\equipsubStar$ is left quasi-representable and $g$ preserves $\colim^W f$.
\end{proof}

(Co)limits which are preserved by any quasi-functors are called \emph{absolute}. The proposition above means that colimits weighted by left adjoints and limits weighted by right adjoints are absolute.
Among h-(co)limits introduced in the previous subsection, h-initial and h-terminal objects, finite h-(co)products, h-(co)shifts, and h-(co)cones are absolute (see \cref{example:absolute_h-colimit}).

\begin{proposition}\label{prop:right-compact-weighted_colimit_is_equal_to_left-compact-weighted_limit}
	Suppose that a dg bimodule $W\colon \catM\slashedrightarrow \catJ$ is left adjoint to a dg bimodule $V\colon \catJ\slashedrightarrow \catM$ in $\DBimod$. For quasi-functors $f\colon \catJ\to \catA$ and $z\colon \catM\to \catA$, $z$ is the (absolute) $W$-weighted colimit of $f$ if and only if $z$ is the (absolute) $V$-weighted limit of $f$.
\end{proposition}

\begin{proof}
	The result follows from \cref{prop:left-adj-weighted_colimit_is_equal_to_right-adj-weighted_limit}, but we include the proof for clarity.
	By \cref{prop:Kan_extension/lifting_along_adjoint}, we have
	\[ \DRift_W f^\equipsubStar \cong V\Dprocomp f^\equipsubStar,\qquad \DRan_V f_\equipsubStar \cong f_\equipsubStar\Dprocomp W. \]
	In particular, we have an adjunction $\DRan_V f_\equipsubStar \dashv \DRift_W f^\equipsubStar$. Hence, $\DRift_W f^\equipsubStar$ is corepresentable by $z$ if and only if $\DRan_V f_\equipsubStar$ is representable by $z$.
\end{proof}

Roughly speaking, \cref{prop:right-compact-weighted_colimit_is_equal_to_left-compact-weighted_limit} says that $W$-weighted colimits and $V$-weighted limits coincide, where $W$ is left adjoint to $V$ in $\DBimod$. A dg category admits $W$-weighted colimits if it admits $V$-weighted limits, and vice versa.

\begin{example}\label{example:absolute_h-colimit}
	Let $X\colon \basek\slashedrightarrow\catJ$ be an h-projective dg bimodule from the unit dg category. We recall that $\DIsbellL(X)\cong \IsbellL(X)\colon \catJ\slashedrightarrow \basek$ is given by the dg functor
	\[ \catJ\to \Comdgk,\quad j \mapsto \Fun_\dg(\catJ^\op,\Comdgk)(X(\mplaceholder), \catJ(\mplaceholder,j)). \]
	When $X$ is (right) compact, we have $X\dashv \DIsbellL(X)$ in $\DBimod$ from \cref{thm:right_compact_bimodules_are_right_adjoint}.
	\begin{enumerate}
		\item Take $\catJ=\emptyset$ and $X=\Winit\colon \basek\slashedrightarrow\emptyset$, as in \cref{example:h-initial_object}. Then since $\Winit$ is h-projective and compact, it holds that $\DIsbellL(\Winit)=\Vterm\colon \emptyset\slashedrightarrow\basek$ and so we have $\Winit\dashv \Vterm$. Hence, by \cref{prop:right-compact-weighted_colimit_is_equal_to_left-compact-weighted_limit}, h-initial objects and h-terminal objects are equivalent, and they are absolute.
		
		\item Take $\catJ$ as a finite set $J$ and $X=\Wcoprod=(S(\basek))_i\colon\basek\slashedrightarrow J$, as in \cref{example:h-coproduct}. Then $\Wcoprod$ is h-projective and 
		\begin{align*}
			\DIsbellL(\Wcoprod)(j) &= \left(\prod_{i\in J} \Comdgk \right) \big( (S(\basek))_i, \catJ(i,j) \big) \\
			&= \Comdgk \big( S(\basek), \catJ(j,j) \big) = S(\basek),
		\end{align*}
		which implies that $\DIsbellL(\Wcoprod) = (S(\basek))_j = \Vprod$. Since $J$ is finite, $\Wcoprod$ is right compact and hence we have $\Wcoprod\dashv \Vprod$.
		Thus, by \cref{prop:right-compact-weighted_colimit_is_equal_to_left-compact-weighted_limit}, finite h-coproducts and finite h-products are equivalent, and they are absolute.
		
		\item Take $\catJ=\basek$ and $X=\Wcoshift{n}=S^n(\basek)\colon \basek\slashedrightarrow \basek$ for $n \in \Z$, as in \cref{example:h-tensor}. Then $\Wcoshift{n}$ is h-projective and 
		\begin{align*}
			\DIsbellL(\Wcoshift{n})(\ast) &= \Fun_\dg(\basek,\Comdgk) \big( \Wcoshift{n} , \basek(\mplaceholder,\ast) \big) \\
			&= \Comdgk \big( S^n(\basek), S(\basek) \big) = S^{-n}(\basek),
		\end{align*}
		which implies that $\DIsbellL(\Wcoshift{n})=S^{-n}(\basek) = \Vshift{(-n)}$. Since $\Wcoshift{n}$ is right compact, we have $\Wcoshift{n}\dashv \Vshift{(-n)}$.
		Thus, by \cref{prop:right-compact-weighted_colimit_is_equal_to_left-compact-weighted_limit}, homotopical $n$-coshifts and $(-n)$-shifts are equivalent, and they are absolute.
		From now on, we will identify h-$(-n)$-coshifts with h-$n$-shift.
		
		\item Take $\catJ=\warrcat$ and $X=\Wcone\colon \basek\slashedrightarrow\warrcat$, as in \cref{example:h-cocone}. Then $\Wcone$ is h-projective and 
		\begin{align*}
			\DIsbellL(\Wcone)(0) &= \Fun_\dg(\warrcat^\op,\Comdgk) \big( \Wcone , \warrcat(\mplaceholder,0) \big) \\
			&= \Cocone(0\to S(\basek)) = S(\basek)[-1],\\
			\DIsbellL(\Wcone)(1) &= \Fun_\dg(\warrcat^\op,\Comdgk) \big( \Wcone , \warrcat(\mplaceholder,1) \big) \\
			&= \Cocone(S(\basek)\xrightarrow{\id} S(\basek)) = D(\basek)[-1].
		\end{align*}
		Therefore we see that
		\begin{align*}
			\DIsbellL(\Wcone)&=(\iota[-1]\colon S(\basek)[-1]\to D(\basek)[-1])\\
			&= \Vcocone[-1]=\Vshift{1}\Dprocomp_\basek \Vcocone.
		\end{align*}
		Since $\Wcone$ is right compact, we have $\Wcone\dashv \Vshift{1}\Dprocomp_\basek \Vcocone$.
		Thus, by \cref{prop:right-compact-weighted_colimit_is_equal_to_left-compact-weighted_limit}, h-cones and $1$-shifts of h-cocones are equivalent. Similarly, $(-1)$-shifts of h-cones and h-cocones are equivalent.
		Also, h-cones and h-cocones are absolute.
	\end{enumerate}
\end{example}

The absoluteness of the h-cone and h-shift ensures that any quasi-functor between pretriangulated dg categories preserves these (co)limits, and hence it yields a triangulated functor on the $H^0$-categories.

The following is a homotopical analog of Cauchy completeness.

\begin{definition}
	A dg category $\catA$ is called \emph{h-Cauchy complete} if all left adjoint dg bimodules in $\DBimod$ with codomain $\catA$ are right quasi-representable.
\end{definition}

\begin{proposition}\label{prop:characterization_of_h-Cauchy_complete}
	For a dg category $\catA$, the following are equivalent.
	\begin{enumerate}[label=\equivitem]
		\item $\catA$ is h-Cauchy complete.
		\item $\catA$ has all limits weighted by right adjoint dg bimodules.
		\item $\catA$ has all colimits weighted by left adjoint dg bimodules.
		\item The inclusion $H^0(\catA)\hookrightarrow \Perf(\catA)\coloneqq \Dom(\catA)^\cpt$ is an equivalence of categories.

		      %\item $\catA$ has homotopical shifts, h-cones and h-idempotent splittings.
	\end{enumerate}
\end{proposition}

\begin{proof}
	(i) $\Rightarrow$ (ii): Let $W\dashv V$ be an adjunction in $\DBimod$ and $f$ a quasi-functor.
	By the proof of \cref{prop:right-compact-weighted_colimit_is_equal_to_left-compact-weighted_limit}, $\DRan_V f_\equipsubStar$ is a left adjoint, and hence it is right quasi-representable. Thus $\catA$ has all limits weighted by right adjoint dg bimodules.

	(ii) $\Leftrightarrow$ (iii): By \cref{prop:right-compact-weighted_colimit_is_equal_to_left-compact-weighted_limit}.

	(iii) $\Rightarrow$ (i): Let $W\colon \catD \slashedrightarrow \catA$ be a left adjoint in $\DBimod$. Then the colimit $\colim^W \id_\catA$ exists, which means that $\DRift_W I_\catA$ is left quasi-representable. Since $\DIsbellL(W)=\DRift_W I_\catA$ is right adjoint to $W$, it follows that $W$ is right quasi-representable.

	(i) $\Leftrightarrow$ (iv): It follows from \cref{thm:right_compact_bimodules_are_right_adjoint}.
	%(iv) $\Leftrightarrow$ (v): It easily follows, because $\Perf(\catA)= \Dom(\catA)^\cpt$ is the thick closure of $H^0(\catA)$ in $\Dom(\catA)$ by \cref{prop:Brown_representability_theorem}.
\end{proof}

\subsection{Application to the gluing}\label{subsection:application}

In \cite{Kuznetsov-Lunts:2015}, the gluing construction of two dg categories along a dg bimodule is introduced. That is, from a dg bimodule $\varphi\colon \catA\slashedrightarrow\catB$ we can make a new dg category $\int\varphi$ together with canonical dg functors $p\colon \int\varphi\to \catA$ and $q\colon \int\varphi\to \catB$. It is essentially observed in that paper that the gluing has a weak universality in the following sense.

\begin{proposition}[{\cite[Proposition A.7 (i)]{Kuznetsov-Lunts:2015}}]
	\label{prop:gluing_is_weak_tabulator}
	Let $\varphi\colon \catA \slashedrightarrow\catB$ be a dg bimodule. Then, for a dg category $\catC$, the functor
	\[ \Dom(\catC,\textstyle\int\varphi) \to \Dom(\catC,\catB)\downarrow \Dom(\catC,\varphi) \]
	is \emph{weakly smothering}, i.e.\ essentially surjective, full, and conservative. Here $\Dom(\catC,\catB)\downarrow \Dom(\catC,\varphi)$ denotes the comma category of functors
	\[ \Dom(\catC,\catB) \xrightarrow{\id} \Dom(\catC,\catB) \xleftarrow{\varphi\Dprocomp-} \Dom(\catC,\catA). \]
	Moreover, a dg bimodule $\alpha\colon \catC \slashedrightarrow\int\varphi$ is a quasi-functor if and only if both $p_{*}\Dprocomp\alpha$ and $q_*\Dprocomp\alpha$ are quasi-functors.
\end{proposition}

\begin{proof}
	The original statement is a bijective correspondance between isoclasses of objects of $\Dom(\catC,\textstyle\int\varphi)$ and of $\Dom(\catC,\catB)\downarrow \Dom(\catC,\varphi)$. This is owing to \cite[Lemma 2.5]{Kuznetsov-Lunts:2015}. By refining the proof of \cite[Lemma 2.5]{Kuznetsov-Lunts:2015}, we obtain the assertion.
\end{proof}

The notion of \emph{smothering} functors was first introduced in \cite{Riehl-Verity:2015The_2-category}.
The notion of weakly smothering functors is due to \cite[``\href{https://ncatlab.org/nlab/show/smothering+functor}{smothering functor}'']{nLab} or \cite{Arlin:2020phd_thesis}.

As shown in \cite[Lemma 4.3]{Kuznetsov-Lunts:2015}, the gluing of pretriangulated dg categories is again pretriangulated. The next proposition gives a slight generalization and an alternative, more conceptual proof of this result.

\begin{proposition}\label{prop:gluing_has_right-adjoint-weighted_limits}
	Let $\varphi\colon \catA \slashedrightarrow\catB$ be a dg bimodule. Let $V\colon \catJ\slashedrightarrow\catM$ be a right adjoint dg bimodule and suppose that $\catA$ and $\catB$ have $V$-weighted limits. Then the gluing $\int \varphi$ also has $V$-weighted limits.
\end{proposition}

\begin{proof}
	Take a quasi-functor $f\colon \catJ\to \int\varphi$ and put $\alpha= \DRan_V f_\equipsubStar \colon \catM \slashedrightarrow \int\varphi$, which is an absolute right Kan extension as $V$ is right adjoint. We want to prove that $\alpha$ is a quasi-functor. By \cref{prop:gluing_is_weak_tabulator}, it suffices to show that $p_*\Dprocomp\alpha$ and $q_*\Dprocomp \alpha$ are quasi-functors. Since $\alpha= \DRan_V f_\equipsubStar$ is absolute, $p_*\Dprocomp\alpha$ is a right Kan extension $\DRan_V (p_*\Dprocomp f_\equipsubStar)$. Now that $\catA$ has $V$-weighted limits, $p_*\Dprocomp\alpha$ is a quasi-functor. The same arguments imply that $q_*\Dprocomp \alpha$ is also a quasi-functor.
\end{proof}

\subsection{Reflection of adjoints and colimits}\label{subsection:reflection_of_adjoints_and_colimits}

We have already observed that adjunctions of quasi-functors induce ordinary adjunctions between $H^0$-categories, and that $H^0$-categories of dg categories with h-coproducts admit coproducts in the $1$-categorical sense.
We study when such adjunctions and colimits are reflected under the $H^0$-construction.

The following is an analog of the fact that a natural transformation is invertible if its components are so.

\begin{proposition}\label{prop:termwise_criterion_of_invertible_2-cell_of_DBimod}
	Let $\alpha\colon X\Rightarrow Y$ be a morphism in $\Dom(\catA,\catB)$.
	\begin{enumerate}
		\item If $\alpha A_\equipsubAst \coloneqq \alpha \Dprocomp A_\equipsubAst \colon X\Dprocomp A_\equipsubAst \Rightarrow Y\Dprocomp A_\equipsubAst$ is invertible in $\Dom(\basek,\catB)=\Dom(\catB)$ for each $A\in \catA$, then $\alpha$ is invertible in $\Dom(\catA,\catB)$.
		
		\item If $B^\equipsubAst \alpha \coloneqq B^\equipsubAst \Dprocomp \alpha \colon B^\equipsubAst \Dprocomp X \Rightarrow B^\equipsubAst \Dprocomp Y$ is invertible in $\Dom(\catA,\basek)=\Dom(\catA^\op)$ for each $B\in \catB$, then $\alpha$ is invertible in $\Dom(\catA,\catB)$.
	\end{enumerate}
\end{proposition}

\begin{proof}
	(1) We remark that since $A_\equipsubAst$ is right h-flat, we have $X\Dprocomp A_\equipsubAst \cong X \procomp A_\equipsubAst = X(\mplaceholder,A)$ in $\Dom(\catB)$. Then, from \cref{cor:bbLtil_is_locally_conservative}, we can see
	\begin{align*}
		& \text{$\alpha A_\equipsubAst$ is invertible in $\Dom(\basek,\catB)$} \\
		&\iff \text{$\bbLtil(\alpha A_\equipsubAst)\colon \bbLT_{X(\mplaceholder,A)} \Rightarrow \bbLT_{Y(\mplaceholder,A)}\colon \Dom(\basek)\to \Dom(\catB)$ is an isomorphism,} \\
		&\iff \text{$\bbLtil(\alpha) \bbLT_{A_\equipsubAst}$ is an isomorphism,} \\
		&\iff \text{$\bbLtil(\alpha)_A\colon X(\mplaceholder,A)\to Y(\mplaceholder,A)$ is an isomorphism.}
	\end{align*}
	Thus
	\[ \text{$\alpha A_\equipsubAst$ is invertible in $\Dom(\catB)$ for each $A\in \catA$} \iff \text{$\bbLtil(\alpha)$ is an isomorphism.} \]
	Therefore \cref{cor:bbLtil_is_locally_conservative} implies that $\alpha$ is invertible.
	
	(2) Under the equation $\Dom(\catA,\catB)=\Dom(\catA^\op\otimes\catB)=\Dom( (\catB^\op)^\op,\catA^\op)=\Dom(\catB^\op,\catA^\op)$, we can identify $B^\equipsubAst \Dprocomp_\catB \alpha = \alpha \Dprocomp_{\catB^\op} B_\equipsubAst$. Therefore we obtain the assertion from (1).
\end{proof}

In the following, via $\DBimod(\basek,\basek)=\Dom(\basek)$, we regard an object $S\in \Dom(\basek)$ and morphism $\alpha\colon S\to T$ in $\Dom(\basek)$ as a morhism $S\colon \basek\slashedrightarrow\basek$ and a $2$-cell $\alpha\colon S\Rightarrow T \colon \basek\slashedrightarrow\basek$ in $\DBimod$, respectively. Also we recall that there are the cohomology functors $H^n\colon \Dom(\basek) \to \Ab$ from the derived category of complexes.

\begin{lemma}\label{lem:isom_in_DBimod(kk)}
	Let $\alpha$ be as above and $n\in \Z$ be an integer. Then $H^n(\alpha)\colon H^n(S)\to H^n(T)$ is an isomorphism if and only if the post-composition map
	\[ \alpha\circ\mplaceholder \colon \DBimod(\basek,\basek)(S^n(\basek),S) \to \DBimod(\basek,\basek)(S^n(\basek),T) \]
	is a bijection.
\end{lemma}

\begin{proof}
	For $S\in \Dom(\basek)$, using the h-projectivity of the sphere complex $S^n(\basek)$, we have
	\begin{align*}
		\DBimod(\basek,\basek)(S^n(\basek),S)
		&= \Hom_{\Dom(\basek)}(S^n(\basek),S) \\
		&\cong \Hom_{\Kom(\basek)}(S^n(\basek),S) \\
		&= H^0\big(\Homcpx(S^n(\basek),S)\big) \\
		&\cong H^0(S[n]) \cong H^n(S).
	\end{align*}
	Hence $\DBimod(\basek,\basek)(S^n(\basek),\mplaceholder) \cong H^n(\mplaceholder)$, which shows the assertion.
\end{proof}

\begin{lemma}\label{lem:right_quasi-representable_can_be_checked_pointwise}
	Let $X\colon \catA\slashedrightarrow\catB$ be a dg bimodule.
	\begin{enumerate}
		\item $X$ is right quasi-representable if and only if for all $A\in\catA$, $X\Dprocomp A_\equipsubAst\colon \basek\slashedrightarrow \catB$ is quasi-representable.
		\item $X$ is left quasi-representable if and only if for all $B\in\catB$, $B^\equipsubAst \Dprocomp X \colon \catA\slashedrightarrow \basek$ is quasi-corepresentable.
	\end{enumerate}
\end{lemma}

\begin{proof}
	By \cref{lem:when_the_constraint_is_invertible}, $X\Dprocomp A_\equipsubAst\cong X\procomp A_\equipsubAst = X(\mplaceholder,A)$ and $B^\equipsubAst \Dprocomp X \cong B^\equipsubAst \procomp X = X(B,\mplaceholder)$. Hence the lemma follows from the definitions.
\end{proof}

We can now prove the reflection theorem for left adjoints of quasi-functors.

\begin{theorem}\label{thm:reflection_theorem_of_left_adjoint_for_quasi-functor}
	Let $f\colon \catA\to\catB$ be a quasi-functor. Suppose that $\catA$ and $\catB$ have h-shifts. Then $f$ admits a left adjoint as a quasi-functor if and only if the induced functor $\Htil(f)=\res{\bbLT_{f_\equipsubStar}}{\qucl{\catA}} \colon \qucl{\catA}\to \qucl{\catB}$ has a left adjoint.
\end{theorem}

\begin{proof}
	The only-if part has been proved. We will show the if part. Put $X=f_\equipsubStar\colon \catA\slashedrightarrow\catB$. By \cref{prop:characterization_of_adjointness_for_quasi-functors}, it suffices to show that $X$ is left quasi-representable, i.e., that $B^\equipsubAst \Dprocomp X \colon \catA\slashedrightarrow \basek$ is quasi-corepresentable in $\Dom(\catA,\basek)=\Dom(\catA^\op)$ for all $B\in\catB$ (see \cref{lem:right_quasi-representable_can_be_checked_pointwise}). Let $F\coloneqq \Htil(f)=\res{\bbLT_{f_\equipsubStar}}{\qucl{\catA}}$. Now $H^0(\catA)\simeq \qucl{\catA} \xrightarrow{F} \qucl{\catB}$ has a left adjoint, and hence, for each $B\in\catB$, there exists a $C\in \catA$ together with a natural isomorphism
	\[ \Hom_{\qucl{\catB}}(B_\equipsubAst, F(\mplaceholder)) \cong \Hom_{H^0(\catA)}(C,\mplaceholder) \colon H^0(\catA)\to \Ab.  \]
	We calculate the left hand side as
	\begin{align*}
		\Hom_{\qucl{\catB}}(B_\equipsubAst, F(\mplaceholder)) &= \Hom_{\Dom(\catB)}(\catB(?,B) , \bbLT_X(\mplaceholder)(?)) \\
		&\cong \res{H^0(\bbLT_X(\mplaceholder)(B))}{H^0(\catA)} \\
		&= H^0(X(B,\mplaceholder)).
	\end{align*}
	Therefore the above isomorphism states that
	\begin{equation}
		\Hom_{H^0(\catA)}(C,\mplaceholder) = H^0(\catA(C,\mplaceholder)) \cong H^0(X(B,\mplaceholder)). \label{equ:H^0(X(B,-))_is_representable}
	\end{equation}
	It follows from the Yoneda lemma that
	\begin{align*}
		\Nat(\Hom_{H^0(\catA)}(C,\mplaceholder), H^0(X(B,\mplaceholder))) &\cong H^0(X(B,C)) \\
		&\cong \Hom_{\Dom(\catA^\op)}(\catA(C,\mplaceholder), X(B,\mplaceholder)).
	\end{align*}
	Thus, we obtain a morphism in $\Dom(\catA^\op)$
	\[ \alpha\colon \catA(C,\mplaceholder) \to X(B,\mplaceholder) \]
	which corresponds to the isomorphism~\eqref{equ:H^0(X(B,-))_is_representable}. Regarding $\alpha$ as a $2$-cell $C^\equipsubAst\Rightarrow B^\equipsubAst\Dprocomp X\colon \catA\slashedrightarrow \basek$ in $\DBimod$, we can observe that
	\[ H^0(\alpha\Dprocomp A_\equipsubAst)\colon H^0(C^\equipsubAst\Dprocomp A_\equipsubAst) \to H^0(B^\equipsubAst \Dprocomp X \Dprocomp A_\equipsubAst) \]
	is an isomorphism for all $A\in \catA$.
	
	In order to show $\alpha$ is invertible in $\Dom(\catA^\op)$, it is sufficient by \cref{prop:termwise_criterion_of_invertible_2-cell_of_DBimod} to prove that
	\[ \alpha A_\equipsubAst \coloneqq \alpha\Dprocomp A_\equipsubAst \colon C^\equipsubAst\Dprocomp A_\equipsubAst \Rightarrow B^\equipsubAst\Dprocomp X \Dprocomp A_\equipsubAst\colon \basek\slashedrightarrow\basek \]
	is invertible in $\Dom(\basek)$ for all $A\in \catA$. To do so, we only have to verify that $H^n(\alpha A_\equipsubAst)$ is an isomorphism for all $n \in \Z$, which is equivalent by \cref{lem:isom_in_DBimod(kk)} to showing that
	\[ \alpha A_\equipsubAst \circ \mplaceholder \colon \DBimod(\basek,\basek)(S^n(\basek), C^\equipsubAst\Dprocomp A_\equipsubAst) \to \DBimod(\basek,\basek)(S^n(\basek), B^\equipsubAst\Dprocomp X \Dprocomp A_\equipsubAst) \]
	is bijective.
	
	Since $\catA$ has h-shifts, it has the h-cotensor product $S\pitchfork A$ with the sphere complex $S\coloneqq S^n(\basek)\in \Dom(\basek)$: in other words, it holds that $(S\pitchfork A)_\equipsubAst \cong \DRan_{S} A_\equipsubAst$ in $\DBimod$. Since any quasi-functor preserves h-shifts (see \cref{example:absolute_h-colimit}), we have $\DRan_{S}(X \Dprocomp A_\equipsubAst) \cong X \Dprocomp (S\pitchfork A)_\equipsubAst$. Since right adjoints commute with right Kan extensions, $\DRan_{S}(C^\equipsubAst\Dprocomp A_\equipsubAst) \cong C^\equipsubAst\Dprocomp (S\pitchfork A)_\equipsubAst$ and $\DRan_{S}(B^\equipsubAst\Dprocomp X \Dprocomp A_\equipsubAst) \cong B^\equipsubAst\Dprocomp \DRan_{S}(X \Dprocomp A_\equipsubAst) \cong B^\equipsubAst\Dprocomp X \Dprocomp (S\pitchfork A)_\equipsubAst$. As a consequence, we can see that $\DRan_{S}(\alpha A_\equipsubAst) \cong \alpha \Dprocomp (S\pitchfork A)_\equipsubAst \eqqcolon \alpha (S\pitchfork A)_\equipsubAst$. Therefore, recalling $S^n(\basek) \cong S^0(\basek)\Dprocomp S$, we obtain the commutative diagram
	\[
	\begin{tikzcd}[column sep=5.4em]
		\DBimod(\basek,\basek)(S^n(\basek), C^\equipsubAst A_\equipsubAst) \arrow{r}{\alpha A_\equipsubAst \circ \mplaceholder} \arrow[equal]{d} & \DBimod(\basek,\basek)(S^n(\basek), B^\equipsubAst X A_\equipsubAst) \arrow[equal]{d} \\
		\DBimod(\basek,\basek)(S^0(\basek), \DRan_{S}(C^\equipsubAst A_\equipsubAst)) \arrow{r}{\DRan_{S}(\alpha A_\equipsubAst) \circ \mplaceholder} \arrow[equal]{d} & \DBimod(\basek,\basek)(S^0(\basek), \DRan_{S}(B^\equipsubAst X A_\equipsubAst)) \arrow[equal]{d} \\
		\DBimod(\basek,\basek)(S^0(\basek), C^\equipsubAst (S\pitchfork A)_\equipsubAst) \arrow{r}{\alpha(S\pitchfork A)_\equipsubAst \circ \mplaceholder} \arrow[equal]{d} & \DBimod(\basek,\basek)(S^0(\basek), B^\equipsubAst X (S\pitchfork A)_\equipsubAst) \arrow[equal]{d} \\
		H^0(C^\equipsubAst (S\pitchfork A)_\equipsubAst) \arrow{r}{H^0(\alpha(S\pitchfork A)_\equipsubAst)} & H^0(B^\equipsubAst X (S\pitchfork A)_\equipsubAst).
	\end{tikzcd}
	\]
	Because the bottom arrow $H^0(\alpha (S\pitchfork A)_\equipsubAst)$ is an isomorphism from the definition of $\alpha$, it follows that the top arrow $\alpha A_\equipsubAst \circ \mplaceholder$ is so. Thus $\alpha$ is invertibe in $\Dom(\catA^\op)$, and hence $X=f_\equipsubStar$ is left quasi-representable.
\end{proof}

In \cite[Remark 3.9]{Lowen-RamosGonzalez:2022tensor_product_of_well_generated}, it is shown that a \textit{dg Bousfiled localization} (a quasi-fully faithful functor between pretriangulated dg categories which induces a Bousfield localization at the $H^0$-level) admits a left adjoint quasi-functor.
\Cref{thm:reflection_theorem_of_left_adjoint_for_quasi-functor} above generalizes the arguments of \cite[Remark 3.9]{Lowen-RamosGonzalez:2022tensor_product_of_well_generated}.

Using the same idea as in the proof of \cref{thm:reflection_theorem_of_left_adjoint_for_quasi-functor}, we also obtain the reflection theorem for h-coproducts.

\begin{theorem}\label{thm:reflection_theorem_of_h-coproduct}
	Let $\catA$ be a dg category with h-shifts and $\{A_j\}_{j\in J}$ a family of objects of $\catA$. Then $\catA$ admits an h-coproduct of $\{A_j\}_j$ if and only if $H^0(\catA)$ has a coproduct of $\{A_j\}_j$.
\end{theorem}

\begin{proof}
	The only-if part is immediate. We will show the if part, i.e., that the product $\prod_j \catA(A_j,\mplaceholder)$ in $\Dom(\catA^\op)$ is quasi-corepresentable. Now $H^0(\catA)$ has a coproduct of $\{A_j\}_j$, and hence there exists a $C\in H^0(\catA)$ together with a natural isomorphism
	\begin{equation}
		\Hom_{H^0(\catA)}(C,\mplaceholder)=H^0(\catA(C,\mplaceholder)) \cong {\textstyle \prod_j \Hom_{H^0(\catA)}(A_j,\mplaceholder)} = H^0({\textstyle \prod_j \catA(A_j,\mplaceholder)}). \label{equ:H^0(prod-A(A_j,-))_is_representable}
	\end{equation}
	It follows from the Yoneda lemma that
	\begin{align*}
		\Nat\big(\Hom_{H^0(\catA)}(C,\mplaceholder), H^0({\textstyle \prod_j \catA(A_j,\mplaceholder)})\big) &\cong H^0({\textstyle \prod_j \catA(A_j,C)}) \\
		&\cong \Hom_{\Dom(\catA^\op)}(\catA(C,\mplaceholder), {\textstyle \prod_j \catA(A_j,\mplaceholder)}).
	\end{align*}
	Thus, we obtain a morphism in $\Dom(\catA^\op)$
	\[ \alpha\colon \catA(C,\mplaceholder) \to {\textstyle\prod_j \catA(A_j,\mplaceholder)} \]
	which corresponds to the isomorphism \eqref{equ:H^0(prod-A(A_j,-))_is_representable}. Regarding $\alpha$ as a $2$-cell $C^\equipsubAst\Rightarrow {\textstyle\prod_j A_j^\equipsubAst}\colon \catA\slashedrightarrow\basek$ in $\DBimod$, we can observe that
	\[ H^0(\alpha\Dprocomp A_\equipsubAst)\colon H^0(C^\equipsubAst\Dprocomp A_\equipsubAst) \to H^0({(\textstyle\prod_j A_j^\equipsubAst)} \Dprocomp A_\equipsubAst) \]
	is an isomorphism for all $A\in \catA$.
	
	In order to show $\alpha$ is invertible in $\Dom(\catA^\op)$, it is sufficient by \cref{prop:termwise_criterion_of_invertible_2-cell_of_DBimod} to prove that
	\[ \alpha A_\equipsubAst \coloneqq \alpha\Dprocomp A_\equipsubAst \colon C^\equipsubAst\Dprocomp A_\equipsubAst \Rightarrow {(\textstyle\prod_j A_j^\equipsubAst)} \Dprocomp A_\equipsubAst \colon \basek\slashedrightarrow\basek \]
	is invertible in $\Dom(\basek)$ for all $A\in \catA$. To do so, we only have to verify that $H^n(\alpha A_\equipsubAst)$ is an isomorphism for all $n \in \Z$, which is equivalent by \cref{lem:isom_in_DBimod(kk)} to showing that
	\[ \alpha A_\equipsubAst \circ \mplaceholder \colon \DBimod(\basek,\basek)(S^n(\basek), C^\equipsubAst\Dprocomp A_\equipsubAst) \to \DBimod(\basek,\basek)(S^n(\basek), {(\textstyle\prod_j A_j^\equipsubAst)} \Dprocomp A_\equipsubAst) \]
	is bijective.
	
	Since $\catA$ has h-shifts, it has the h-cotensor product $S\pitchfork A$ with the sphere complex $S\coloneqq S^n(\basek)\in \Dom(\basek)$: in other words, it holds that $(S\pitchfork A)_\equipsubAst \cong \DRan_{S} A_\equipsubAst$ in $\DBimod$.
	Since right adjoints commute with right Kan extensions, $\DRan_{S}(C^\equipsubAst\Dprocomp A_\equipsubAst) \cong C^\equipsubAst\Dprocomp (S\pitchfork A)_\equipsubAst$. 
	Since $A_\equipsubAst\colon \basek\slashedrightarrow\catA$ is left adjoint in $\DBimod$ (see \cref{lem:A_ast_is_left_adjoint_to_A^star}), $\mplaceholder \Dprocomp A_\equipsubAst$ is right adjoint, and hence preserves limits. Therefore it holds that
	\( {(\textstyle\prod_j A_j^\equipsubAst)} \Dprocomp A_\equipsubAst \cong {\textstyle\prod_j (A_j^\equipsubAst \Dprocomp A_\equipsubAst)}. \)
	By using this and the fact that the right Kan extension commutes with limits, we have
	\begin{align*}
		\DRan_S({(\textstyle\prod_j A_j^\equipsubAst)\Dprocomp A_\equipsubAst})
		&\cong \DRan_S ({\textstyle\prod_j (A_j^\equipsubAst \Dprocomp A_\equipsubAst)}) \\
		&\cong {\textstyle\prod_j \DRan_S (A_j^\equipsubAst \Dprocomp A_\equipsubAst)} \\
		&\cong {\textstyle\prod_j (A_j^\equipsubAst \Dprocomp (S\pitchfork A)_\equipsubAst)} \\
		&\cong {(\textstyle\prod_j A_j^\equipsubAst)} \Dprocomp (S\pitchfork A)_\equipsubAst.
	\end{align*}
	As a consequence, we can see that $\DRan_{S}(\alpha A_\equipsubAst) \cong \alpha \Dprocomp (S\pitchfork A)_\equipsubAst \eqqcolon \alpha (S\pitchfork A)_\equipsubAst$. Therefore, recalling $S^n(\basek) \cong S^0(\basek)\Dprocomp S$, we obtain the commutative diagram
	\[
	\begin{tikzcd}[column sep=5.4em]
		\DBimod(\basek,\basek)(S^n(\basek), C^\equipsubAst A_\equipsubAst) \arrow{r}{\alpha A_\equipsubAst \circ \mplaceholder} \arrow[equal]{d} & \DBimod(\basek,\basek)(S^n(\basek), {(\textstyle\prod_j A_j^\equipsubAst)} A_\equipsubAst) \arrow[equal]{d} \\
		\DBimod(\basek,\basek)(S^0(\basek), \DRan_{S}(C^\equipsubAst A_\equipsubAst)) \arrow{r}{\DRan_{S}(\alpha A_\equipsubAst) \circ \mplaceholder} \arrow[equal]{d} & \DBimod(\basek,\basek)(S^0(\basek), \DRan_{S}({(\textstyle\prod_j A_j^\equipsubAst)} A_\equipsubAst)) \arrow[equal]{d} \\
		\DBimod(\basek,\basek)(S^0(\basek), C^\equipsubAst (S\pitchfork A)_\equipsubAst) \arrow{r}{\alpha(S\pitchfork A)_\equipsubAst \circ \mplaceholder} \arrow[equal]{d} & \DBimod(\basek,\basek)(S^0(\basek), {(\textstyle\prod_j A_j^\equipsubAst)} (S\pitchfork A)_\equipsubAst) \arrow[equal]{d} \\
		H^0(C^\equipsubAst (S\pitchfork A)_\equipsubAst) \arrow{r}{H^0(\alpha(S\pitchfork A)_\equipsubAst)} & H^0({(\textstyle\prod_j A_j^\equipsubAst)} (S\pitchfork A)_\equipsubAst).
	\end{tikzcd}
	\]
	Because the bottom arrow $H^0(\alpha (S\pitchfork A)_\equipsubAst)$ is an isomorphism from the definition of $\alpha$, it follows that the top arrow $\alpha A_\equipsubAst \circ \mplaceholder$ is so. Thus $\alpha$ is invertibe in $\Dom(\catA^\op)$, and hence $\prod_j \catA(A_j,\mplaceholder)$ is quasi-corepresentable.
\end{proof}

Let $\kappa$ be a cardinal.
In \cite[Definition 6.1]{Porta:2010} and \cite[Remark 3.9]{Lowen-RamosGonzalez:2022tensor_product_of_well_generated}, a dg category $\catA$ is called \emph{homotopically $\kappa$-cocomplete} if $H^0(\catA)$ has all $\kappa$-small coproducts.
This notion relates to our one, as follows.

\begin{corollary}\label{cor:cocompleteness_for_pretri_dg_cat}
	For a pretriangulated dg category $\catA$, it has all weighted h-colimits if and only if $H^0(\catA)$ has all coproducts.
\end{corollary}

\begin{proof}
	It follows from \cref{prop:h-cocomplete_iff_having_coshifts_cones_coproducts} and \cref{thm:reflection_theorem_of_h-coproduct}.
\end{proof}

\begin{corollary}\label{cor:cocontinuity_for_quasi-functor_bw_pretri_dg_cat}
	A quasi-functor $f\colon \catA\to \catB$ between pretriangulated dg categories preserves all h-colimits if and only if the induced functor $\Htil(f)=\res{\bbLT_{f_\equipsubStar}}{\qucl{\catA}}\colon \qucl{\catA}\to\qucl{\catB}$ preserves all coproducts.
\end{corollary}

\begin{proof}
	It follows from \cref{prop:h-cocomplete_iff_having_coshifts_cones_coproducts} and \cref{thm:reflection_theorem_of_h-coproduct} together with \cref{example:preservation_of_h-coproducts}.
\end{proof}

This corollary provides some justification for the terminology used in \cite[Definition 2.5]{Lowen-RamosGonzalez:2022tensor_product_of_well_generated}.

%% compact iff A(A,-) preserves coproduct

\section{A universal property of the homotopy category theory}\label{section:universality_of_htpy_cat_theory}

We will continue to assume that $\basek$ is a field. We finally observe that the proarrow equipment $\equipStar\colon \DBimodrqr\to \DBimod$ has a universal property among proarrow equipments.

We have seen in \cref{section:quasi-functors} that there are a normal lax functor $\Gamma\colon\Bimod\to\DBimod$ and a pseudo-functor $\gamma\colon \dgCat\to \DBimodrqr$ such that the diagram
\[
\begin{tikzcd}
	\dgCat \arrow[hook]{r}{\equipAst} \arrow{d}[swap]{\gamma} & \Bimod \arrow{d}{\Gamma} \\
	\DBimodrqr \arrow[hook]{r}{\equipStar} & \DBimod\vspace{-1ex}
\end{tikzcd}
\]
commutes.
%Moreover we can see that the pair $(\Gamma,\gamma)$ behaves like a ``morphism of proarrow equipments'' in a sense, by using the following proposition.
We have also seen in \cref{cor:adjunction_of_dg_functors_becomes_those_of_quasi-functors} that $\Gamma$ preserves adjunctions of the form $F_\equipsubAst \dashv F^\equipsubAst$ for some dg functor $F$.
These indicate that the pair $(\Gamma,\gamma)$ forms an \emph{equipment morphism} in the sense of \cite[Proposition 1.5.13]{Verity:2011phd_thesis_reprints};
but we will not address that notion in this paper.
%but we do not explore this structure further.

\begin{theorem}\label{thm:universality_1}
	Let $\Theta\colon \Bimod\to \bicatM$ be a lax functor which sends quasi-isomorphism $2$-cells to invertible ones. Then there is a lax functor $\widetilde{\Theta}\colon \DBimod\to\bicatM$ such that the following commutes:
	\[\begin{tikzcd}
		\Bimod \arrow{r}{\Theta} \arrow{d}[swap]{\Gamma} & \bicatM\rlap{.} \\
		\DBimod \arrow[bend right=15]{ru}[swap]{\widetilde{\Theta}}
	\end{tikzcd}\]
	Moreover, $\widetilde{\Theta}$ is normal if so is $\Theta$.
\end{theorem}

\begin{proof}
	On objects, define $\widetilde{\Theta}(\catA)=\Theta(\catA)$. As the functor $\Gamma_{\catA,\catB}\colon \Com(\catA,\catB)\to \Dom(\catA,\catB)$ is a localization with respect to quasi-isomorphisms, $\Theta_{\catA,\catB}$ factors as
	\[\begin{tikzcd}
		\Com(\catA,\catB) \arrow{r}{\Theta_{\catA,\catB}} \arrow{d}[swap]{\Gamma_{\catA,\catB}} & \bicatM(\Theta(\catA),\Theta(\catB))\rlap{.} \\
		\Dom(\catA,\catB) \arrow[bend right=15,dashed]{ru}[swap]{\widetilde{\Theta}_{\catA,\catB}}
	\end{tikzcd}\]
	These data determine a lax functor $\widetilde{\Theta}\colon \DBimod\to\bicatM$ whose constraints are 
	\begin{align*}
		\widetilde{\Theta}(\Gamma(Y)) \circ \widetilde{\Theta}(\Gamma(X)) \cong \Theta(\hp Y) \circ \Theta(\hp X) &\to \Theta(\hp Y \procomp \hp X) =\widetilde{\Theta}(\Gamma(Y\Dprocomp X))\\
		\intertext{and}
		\id_{\widetilde{\Theta}(\catA)}= \id_{\Theta(\catA)} &\to \Theta(I_\catA) = \widetilde{\Theta}(\Gamma(I_\catA)).
	\end{align*}
	From the construction we can verify $\widetilde{\Theta}\circ \Gamma=\Theta$.
\end{proof}

Applying \cref{thm:universality_1} to the lax functor $\bbL\colon \Bimod\to \CAT$, $X\mapsto \bbLT_X$, we get $\bbLtil\colon \DBimod\to \CAT$, which is obtained in \cref{subsection:the_tensor_and_Hom_functors_associated_with_dg_bimodules}.

\begin{lemma}\label{lem:roof_lemma_for_quasi-functor}
	For any right quasi-representable dg bimodule $X\colon \catA \slashedrightarrow \catB$ between (locally h-projective) dg categories, there are a dg functor $F\colon \catC \to\catB$ and a quasi-equivalence $W\colon \catC\to \catA$ such that $X \Dprocomp W_\equipsubAst \cong F_\equipsubAst$ in $\Dom(\catC,\catB)$.
\end{lemma}

\begin{proof}
	This statement is attributed to \cite{Toen:2007}, in which it is shown that there exists a bijection between the set of morphisms $\catC\to\catB$ in $\HodgCat$ and the set of isomorphism classes of right quasi-representable dg bimodules in $\Dom(\catB \otimes^\bbL \catC^\op)=\Dom(\cofQ(\catC),\catB)$ (see \cite[Corollary 4.8]{Toen:2007} or \cite{Canonaco-Stellari:2015Internal_Homs}). When $\catC$ is cofibrant, morphisms $\catC\to\catB$ in $\HodgCat$ is homotopy classes of dg functors. In such a case, we can thus verify that any right quasi-representable dg bimodule $\catC\slashedrightarrow\catB$ is isomorphic to $F_\equipsubAst$ for some dg functor $F$.
	
	Take a cofibrant resolution $W\colon \cofQ(\catA)\to\catA$ of $\catA$ and set $\catC=\cofQ(\catA)$. Then $X\Dprocomp W_\equipsubAst$ is right quasi-representable. Since $\cofQ(\catA)$ is cofibrant, it is isomorphic to $F_\equipsubAst$ in $\Dom(\catC,\catB)$ for some $F\colon \catC\to\catB$.
\end{proof}

\begin{theorem}\label{thm:universality_2}
	Consider another proarrow equipment $\equipBullet\colon \bicatK\hookrightarrow\bicatM$. Let $\Theta\colon \Bimod\to\bicatM$ be a normal lax functor that preserves adjunctions of the form $F_\equipsubAst \dashv F^\equipsubAst$ for dg functors $F$ (i.e.\ it satisfies the equivalent conditions in \cref{lem:condition_lax_functor_preserves_adjunctions} for $F_\equipsubAst \dashv F^\equipsubAst$), and $\theta\colon \dgCat\to\bicatK$ be a pseudo-functor such that $\Theta\circ \equipAst=\equipBullet\circ \theta$.
	If $\Theta$ sends quasi-isomorphism $2$-cells to invertible ones and $\theta$ sends quasi-equivalences to equivalences, then $(\Theta,\theta)$ \emph{factors through} the pair $(\Gamma,\gamma)$: namely, there exist a normal lax functor $\widetilde{\Theta}\colon \DBimod\to\bicatM$ and a pseudo-functor $\widetilde{\theta}\colon \dgCat\to \bicatK$ such that $\widetilde{\Theta}\circ \equipStar=\equipBullet\circ \widetilde{\theta}$, $\widetilde{\Theta}\circ \Gamma = \Theta$, and $\widetilde{\theta}\circ \gamma = \theta$.
	\[\begin{tikzcd}[column sep=small,row sep=small]
		\dgCat \arrow{dd}{\gamma} \arrow{rr}{\theta} \arrow[hookrightarrow]{rd}{\equipAst} &[-2ex] &[-3ex] \bicatK \arrow[hookrightarrow]{rd}{\equipBullet} &[3ex] \\
		& \phantom{bimod} \arrow{dd}{\Gamma} \arrow{rr}[swap]{\Theta} & & \bicatM\rlap{.} \\
		\DBimodrqr \arrow[hookrightarrow]{rd}[swap]{\equipStar}
		\arrow[bend right=10,dashed]{rruu}[swap,very near start]{\widetilde{\theta}}[description,pos=0.55]{\textstyle\Bimod} 
		%\arrow[bend right=10,dash,dashed,shorten >=8ex]{rruu}[swap,near start]{\widetilde{\theta}} \arrow[bend right=10,dashed,shorten <=8ex]{rruu}
		 & & & \\
		& \DBimod \arrow[bend right=10,dashed]{rruu}[swap,near start]{\widetilde{\Theta}} & &
	\end{tikzcd}\]
	%\dgCat \arrow{dd}{\gamma} \arrow{rr}{\theta} \arrow[hookrightarrow]{rd}{\equipAst} &[-2ex] &[-3ex] \bicatK \arrow[hookrightarrow]{rd}{\equipBullet} &[3ex] \\
	%& \Bimod \arrow{dd}{\Gamma} \arrow{rr}[swap]{\Theta} & & \bicatM \\
	%\DBimodrqr \arrow[hookrightarrow]{rd}[swap]{\equipStar} \arrow[bend right=10,dashed]{rruu}[swap,very near start]{\widetilde{\theta}} & & & \\
	%& \DBimod \arrow[bend right=10,dashed]{rruu}[swap]{\widetilde{\Theta}} & &
\end{theorem}

\begin{proof}
	By \cref{thm:universality_1} we have a normal lax functor $\widetilde{\Theta}\colon \DBimod\to \bicatM$ through which $\Theta$ factors. 
	First let us prove that $\widetilde{\Theta}$ preserves the adjunction $\Gamma(F_\equipsubAst) \dashv \Gamma(F^\equipsubAst)$ for a dg functor $F$. It suffices by \cref{lem:condition_lax_functor_preserves_adjunctions} to show that the constraint $\widetilde{\Theta}(X)\circ  \widetilde{\Theta}(\Gamma(F_\equipsubAst)) \to \widetilde{\Theta}(X\Dprocomp \Gamma(F_\equipsubAst))$ is invertible for any composable morphism $X$ in $\DBimod$. It is given by
	\[ \widetilde{\Theta}(\Gamma(X))\circ  \widetilde{\Theta}(\Gamma(F_\equipsubAst)) = \Theta(X) \circ \Theta(F_\equipsubAst) \to \Theta(X\procomp F_\equipsubAst) = \widetilde{\Theta}(\Gamma(X\procomp F_\equipsubAst)) \cong \widetilde{\Theta}(\Gamma(X)\Dprocomp \Gamma(F_\equipsubAst)), \]
	where the last isomorphism is due to \cref{lem:when_the_constraint_is_invertible}. Hence it is invertible because so is $\Theta(X) \circ \Theta(F_\equipsubAst) \to \Theta(X\procomp F_\equipsubAst)$ by the assumption.
	
	Secondly take a quasi-functor $f$.
	By \cref{lem:roof_lemma_for_quasi-functor}, we have $f_\equipsubStar \Dprocomp \Gamma(W_\equipsubAst) \cong \Gamma(F_\equipsubAst)$ in $\DBimod$ for some dg functor $F$ and some quasi-equivalence $W$. Hence we have 
	\[ \widetilde{\Theta}(f_\equipsubStar \Dprocomp \Gamma(W_\equipsubAst)) \cong \widetilde{\Theta}(\Gamma(F_\equipsubAst)) =\Theta(F_\equipsubAst)=\theta(F)_\equipsubBullet. \]
	The above argument implies that
	\[ \widetilde{\Theta}(\Gamma(f_\equipsubStar) \Dprocomp \Gamma(W_\equipsubAst)) \cong \widetilde{\Theta}(\Gamma(f_\equipsubStar)) \circ \widetilde{\Theta}(\Gamma(W_\equipsubAst)) = \Theta(f_\equipsubStar) \circ \Theta(W_\equipsubAst) = \widetilde{\Theta}(f_\equipsubStar) \circ \theta(W)_\equipsubBullet. \]
	Therefore we have $\widetilde{\Theta}(f_\equipsubStar) \circ \theta(W)_\equipsubBullet \cong \theta(F)_\equipsubBullet$. Now we see from the assumption that $\theta(W)$ is an equivalence in $\bicatK$ and $\theta(W)_\equipsubBullet$ has the inverse $\theta(W)^\equipsubBullet=(\theta(W)^{-1})_\equipsubBullet$. Thus we obtain $\widetilde{\Theta}(f_\equipsubStar) \cong \theta(F)_\equipsubBullet\circ (\theta(W)^{-1})_\equipsubBullet \cong (\theta(F) \circ \theta(W)^{-1})_\equipsubBullet$, which shows that the lax functor $\widetilde{\Theta}\circ \equipStar$ factors through $\bicatK$ and induces a lax functor $\widetilde{\theta}\colon \DBimodrqr \to \bicatK$.
	
	We finally check that $\widetilde{\theta}$ is in fact a pseudo-functor. For this purpose, we only need to prove that the constraint $\widetilde{\Theta}(X) \circ \widetilde{\Theta}(f_\equipsubStar) \to \widetilde{\Theta}(X\Dprocomp f_\equipsubStar)$ is invertible for any quasi-functor $f$ and any dg bimodule $X$. Take again a dg functor $F$ and a quasi-equivalence $W$ such that $f_\equipsubStar \Dprocomp \Gamma(W_\equipsubAst) \cong \Gamma(F_\equipsubAst)$. Since $\widetilde{\Theta}(\Gamma(W_\equipsubAst))=\Theta(W_\equipsubAst)=\theta(W)_\equipsubBullet$ is an equivalence in $\bicatM$, it is sufficient to show that \[ \alpha\colon \widetilde{\Theta}(X) \circ \widetilde{\Theta}(f_\equipsubStar) \circ \widetilde{\Theta}(\Gamma(W_\equipsubAst)) \to \widetilde{\Theta}(X\Dprocomp f_\equipsubStar) \circ \widetilde{\Theta}(\Gamma(W_\equipsubAst)) \]
	is invertible. By the associativity axiom for the lax functor $\widetilde{\Theta}$, 
	\[\begin{tikzcd}
		\widetilde{\Theta}(X) \circ \widetilde{\Theta}(f_\equipsubStar) \circ \widetilde{\Theta}(\Gamma(W_\equipsubAst)) \arrow{r} \arrow{d}[swap]{\alpha} & \widetilde{\Theta}(X) \circ \widetilde{\Theta}(f_\equipsubStar \Dprocomp \Gamma(W_\equipsubAst)) \arrow[phantom]{r}{\cong} &[-3ex] \widetilde{\Theta}(X) \circ \widetilde{\Theta}(\Gamma(F_\equipsubAst)) \arrow{d} \\
		\widetilde{\Theta}(X\Dprocomp f_\equipsubStar) \circ \widetilde{\Theta}(\Gamma(W_\equipsubAst)) \arrow{r} & \widetilde{\Theta}(X\Dprocomp f_\equipsubStar \Dprocomp \Gamma(W_\equipsubAst)) \arrow[phantom]{r}{\cong} & \widetilde{\Theta}(X\Dprocomp \Gamma(F_\equipsubAst))\rlap{.}
	\end{tikzcd}\]
	commutes and the three morphisms there except $\alpha$ are invertible by the discussion in the first paragraph. Therefore so is $\alpha$, which is the desired conclusion.
\end{proof}

\begin{example}
	The $0$-th cohomology functor $H^0\colon \Ch(\basek)\to\Mod(\basek)$ becomes a lax monoidal functor. This induces not only a $2$-functor $H^0\colon \dgCat=\Ch(\basek)\text{-}\Cat \to \Mod(\basek)\text{-}\Cat$, but also a normal lax functor $\overline{H}^0 \colon \Bimod=\Ch(\basek)\text{-}\Prof \to \Mod(\basek)\text{-}\Prof$, which sends a dg bimodule $X\colon \catB^\op\otimes \catA \to \Ch_\dg(\basek)$ to
	\[ H^0(\catB)^\op\otimes H^0(\catA) \to H^0(\catB^\op\otimes\catA) \xrightarrow{H^0(X)} H^0(\Ch_\dg(\basek)) \xrightarrow{H^0} \Mod(\basek). \]
	It is immediate to see that $\overline{H}^0$ sends quasi-isomorphisms to isomorphisms and $H^0$ sends quasi-equivalences to equivalences. Moreover, since $X\procomp F_\equipsubAst$ is isomorphic to the dg functor $\catC^\op\otimes\catA \xrightarrow{\id\otimes F} \catC^\op \otimes\catB \xrightarrow{X} \Ch_\dg(\basek)$, we can verify that $\overline{H}^0(X\procomp F_\equipsubAst) \cong \overline{H}^0(X) \procomp H^0(F)_\equipsubAst \cong \overline{H}^0(X) \procomp \overline{H}^0(F_\equipsubAst)$. Therefore, by \cref{thm:universality_2}, $(\overline{H}^0,H^0)$ factors through $(\Gamma,\gamma)$.
\end{example}

\begin{remark}
	The universal property of the proarrow equipment $\equipStar\colon \DBimodrqr\hookrightarrow\DBimod$ as in \cref{thm:universality_2} is like that of localizations of categories. It is a future work to establish the theory of localizations for proarrow equipments and show that $\DBimodrqr\hookrightarrow\DBimod$ is a localization of $\dgCat\hookrightarrow\Bimod$.
\end{remark}

%%% Appendix用にカウンターの変更 %%%
\setcounter{section}{0}
\renewcommand{\thesection}{\Alph{section}}%
\renewcommand{\thesubsection}{\Alph{section}.\arabic{subsection}}

\section{Reviews on proarrow equipments}\label[appendix]{section:proarrow_equipment}

Proarrow equipments are one of the frameworks for formal category theory. The purpose of formal category theory is to abstract and synthesize the theory of ordinary categories.
With the structure of a proarrow equipment, we can interpret many concepts in category theory, such as (co)limits, Cauchy completeness, and pointwise Kan extensions, in the $2$-categorical way.
In this appendix, we review the theory of proarrow equipments from \cite{Wood:1982proarrow1}.

\subsection{Basic bicategory theory}\label{subsection:basic_bicategory_theory}

In this subsection, we collect some basic facts and fix some notation on bicategories.

A \emph{bicategory} is a notion of weak $2$-category. Namely, a bicategory is like a $2$-category such that the associativity and unity axioms hold \emph{only} up to coherent isomorphism.
Basic references of $2$-categories and bicategories are \cite{Benabou:1967bicategories} and \cite{Johnson-Yau:2021}.
In the following, we will omit the coherent isomorphisms for simplicity.

Let $\bicatK$ be a bicategory.
Objects of Hom categories $\bicatK(A,B)$ are called \emph{$1$-morphisms} or simply \emph{morphisms} of $\bicatK$, and morphisms of Hom categories are called \emph{$2$-morphisms} or \emph{$2$-cells}.
We let $\bicatK^\op$ denote the bicategory obtained by reversing morphisms of $\bicatK$ and $\bicatK^\co$ denote the bicategory obtained by reversing $2$-cells of $\bicatK$. Put $\bicatK^\coop=(\bicatK^\op)^\co=(\bicatK^\co)^\op$, which is obtained by reversing both $1$-morphisms and $2$-morphisms of $\bicatK$.

% 垂直合成の記号、水平合成の記号
As in a 2-category, we can define equivalences and adjunctions in a bicategory as well. 

\begin{definition}
	Let $\bicatK$ be a bicategory.
	\begin{enumerate}
		\item An \emph{equivalence} in $\bicatK$ is a pair of morphisms $f\colon A \to B$ and $u\colon B \to A$ such that $u\circ f \cong \id_A$ and $u\circ f \cong \id_B$.
		
		\item An \emph{adjunction} in $\bicatK$ consists of a pair of morphisms $f\colon A \to B$ and $u\colon B \to A$ together with $2$-cells $\eta\colon \id_A\Rightarrow u\circ f$ and $\varepsilon\colon f\circ u \Rightarrow \id_B$ that satisfy the following triangle identities:
		\[\begin{tikzcd}[column sep=small]
			& B \arrow{rr}{\id_B}[name=B2,below,pos=0.5]{} \arrow{rd}[description]{u} & & B \\
			A \arrow{rr}[swap]{\id_A}[name=A1,above,pos=0.5]{} \arrow{ru}{f} & & A \arrow{ru}[swap]{f} & 
			\arrow[Rightarrow,from=A1,to=1-2,shorten <=0ex,shorten >=1.5ex,"\eta",pos=0.3]
			\arrow[Rightarrow,from=2-3,to=B2,shorten <=1.3ex,shorten >=0ex,"\varepsilon",swap,pos=0.6]
		\end{tikzcd} = \id_f, \qquad\quad
		\begin{tikzcd}[column sep=small]
			B \arrow{rr}{\id_B}[name=B2,below,pos=0.5]{} \arrow{rd}[swap]{u} & & B \arrow{rd}{u}& \\
			& A \arrow{rr}[swap]{\id_A}[name=A1,above,pos=0.5]{} \arrow{ru}[description]{f} & & A 
			\arrow[Rightarrow,from=A1,to=1-3,shorten <=0ex,shorten >=1.5ex,"\eta",pos=0.3]
			\arrow[Rightarrow,from=2-2,to=B2,shorten <=1.3ex,shorten >=0ex,"\varepsilon",pos=0.6]
		\end{tikzcd} = \id_u. \]
		In this case, we call $f$ the \emph{left adjoint}, $u$ the \emph{right adjoint}, $\eta$ the \emph{unit}, and $\varepsilon$ the \emph{counit}.
	\end{enumerate}
\end{definition}

\begin{example}
	\begin{enumerate}
		\item In the $2$-category $\Cat$ of categories, an  equivalence is just an equivalence of categories, and an adjunction is just a pair of adjoint functors.
		\item Consider the $2$-category $\VCat$ of enriched categories over a cosmos $\moncatV$\footnotemark. An adjunction in $\VCat$ is equivalently a pair of $\moncatV$-functors $F\colon\catA\to\catB$ and $G\colon \catB \to \catA$ together with a natural isomorphism $\catB(FA,B) \cong \catA(A,GB)$.
	\end{enumerate}
\end{example}
\footnotetext{A \emph{(B\'enabou) cosmos} means a complete and cocomplete symmetric monoidal closed category.}

Now we introduce the notion of Kan extensions and Kan liftings.

\begin{definition}
	Let $\bicatK$ be a bicategory.
	\begin{enumerate}
		\item Let $f\colon A \to B$ and $k\colon A\to D$ be morphisms of $\bicatK$. A \emph{left Kan extension} of $f$ along $k$ is a morphism $\Lan_k f\colon D \to B$ together with a $2$-cell $\theta\colon f\Rightarrow \Lan_k f\circ k$ such that the map
		      \[ \bicatK(D,B)(\Lan_k f,h) \to \bicatK(A,B)(f,h\circ k), \quad \chi \mapsto (\chi k) \circ \theta \]
		      is bijective for all morphisms $h\colon D \to B$. A left Kan extension in $\bicatK^\co$, $\bicatK^\op$, and $\bicatK^\coop$ are called a \emph{right Kan extension}, a \emph{left Kan lifting}, and a \emph{right Kan lifting}, and denoted by $\Ran$, $\Lift$, and $\Rift$, respectively.

		\item Let $g\colon B\to C$ be another morphism of $\bicatK$. Then we say that $g$ \emph{commutes with} the left Kan extension $\Lan_k f$ (resp.\ the right Kan extension $\Ran_k f$) if $g\circ \Lan_k f \cong \Lan_k(g\circ f)$ (resp.\ $g\circ \Ran_k f \cong \Ran_k(g\circ f)$).

		\item Also we say that $g$ \emph{commutes with} the left Kan liftings $\Lift_k f$ (resp.\ the right Kan liftings $\Rift_k f$) if $(\Lift_k f) \circ g \cong \Lift_k (f\circ g)$ (resp.\ $(\Rift_k f) \circ g \cong \Rift_k (f\circ g)$).

		\item A left Kan extension is called \emph{absolute} if it is commuted with by all morphisms. We define absolute right Kan extensions and absolute Kan liftings in a similar way.
	\end{enumerate}
\end{definition}

\begin{proposition}
	Consider morphisms $f\colon A \to B$, $k\colon A\to D$, and $l\colon D\to E$ in a bicategory $\bicatK$. If the left Kan extension $\Lan_k f$ exists, then we have an isomorphism
	\[ \Lan_l \Lan_k f \cong \Lan_{lk} f, \]
	either side existing if the other does.
\end{proposition}

\begin{proposition}\label{prop:adjunction_as_Kan_extension/lifting}
	Consider morphisms $f\colon A\to B$, $g\colon B \to A$ in a bicategory $\bicatK$.
	\begin{enumerate}
		\item For a $2$-cell $\eta \colon \id_A \Rightarrow g\circ f$, the following are equivalent:
		      \begin{enumerate}[label=\equivitem]
			      \item $\eta$ is the unit of the adjunction $f\dashv g$.
			      \item The pair $(f,\eta)$ is the absolute left Kan lifting $\Lift_g \id_A$.
			      \item The pair $(f,\eta)$ is the left Kan lifting $\Lift_g \id_A$, and $g$ commutes with it.
			      \item The pair $(g,\eta)$ is the absolute left Kan extension $\Lan_f \id_A$.
			      \item The pair $(g,\eta)$ is the left Kan extension $\Lan_f \id_A$, and $f$ commutes with it.
		      \end{enumerate}

		\item For a $2$-cell $\varepsilon\colon f\circ g \Rightarrow \id_B$, the following are equivalent:
		      \begin{enumerate}[label=\equivitem]
			      \item $\varepsilon$ is the counit of the adjunction $f\dashv g$.
			      \item The pair $(g,\varepsilon)$ is the absolute right Kan lifting $\Rift_f \id_B$.
			      \item The pair $(g,\varepsilon)$ is the right Kan lifting $\Rift_f \id_B$, and $f$ commutes with it.
			      \item The pair $(f,\varepsilon)$ is the absolute right Kan extension $\Ran_g \id_B$.
			      \item The pair $(f,\varepsilon)$ is the right Kan extension $\Ran_g \id_B$, and $g$ commutes with it.
		      \end{enumerate}
	\end{enumerate}
\end{proposition}

\begin{proposition}\label{prop:Kan_extension/lifting_along_adjoint}
	Let $\bicatK$ be a bicategory.
	\begin{enumerate}
		\item A right Kan lifting along a left adjoint is a post-composite with the right adjoint: that is, for morphisms $h\colon A\to C$ and $f\colon B\to C$ where $f$ has a right adjoint $g$, we have
		      \[ \Rift_f h \cong g\circ h. \]
		      In particular, such a right Kan lifting is absolute.
		\item A right Kan extension along a right adjoint is a pre-composite with the left adjoint: that is, for morphisms $f\colon A\to B$ and $h\colon A \to C$ where $f$ has a left adjoint $g$, we have
		      \[ \Ran_f h \cong h\circ g. \]
		      In particular, such a right Kan extension is absolute.
	\end{enumerate}
\end{proposition}

\begin{proof}
	This follows from \cref{prop:adjunction_as_Kan_extension/lifting}.
\end{proof}

\begin{proposition}\label{prop:adjoint_respects_Kan_extension/lifting}
	Left adjoints commute with all right Kan liftings. Dually, right adjoints commute with all right Kan extensions.
\end{proposition}

\begin{proof}
	Easy exercises on diagram chasing.
\end{proof}

\begin{definition}
	A bicategory $\bicatM$ is said to be \emph{closed} if both the pre-composition functor and the post-composition functor have right adjoints.
	In other words, for morphisms $X\colon A\to B$, $Y\colon B \to C$, and $Z\colon A \to C$ in $\bicatM$, there exist morphisms $Y_\ddag Z$ and $X^\ddag Z$ together with natural bijections
	\begin{align*}
		\bicatM(A,C)(Y\circ X, Z) & \cong \bicatM(A,B)(X, Y_\ddag Z), \\
		\bicatM(A,C)(Y\circ X, Z) & \cong \bicatM(B,C)(Y, X^\ddag Z).
	\end{align*}
\end{definition}

\begin{proposition}\label{prop:closed_bicat_has_Ran_and_Rift}
	Let $X\colon A\to B$, $Y\colon B \to C$, $Z\colon A \to C$ be morphisms in a bicategory $\bicatM$ where $Y\circ\mplaceholder$ and $\mplaceholder\circ X$ have right adjoints $Y_\ddag(\mplaceholder)$ and $X^\ddag(\mplaceholder)$ respectively. Then the counits of these adjunctions
	\[
		\begin{tikzcd}
			& B \arrow{d}{Y}
			\arrow[Rightarrow,d,"",shift right=3.4ex,shorten <=1.7ex,shorten >=-0.3ex] \\
			A \arrow[bend left=15]{ru}{Y_\ddag Z} \arrow{r}[swap]{Z} & C\rlap{,}
		\end{tikzcd}
		\qquad\qquad
		\begin{tikzcd}
			B \arrow[bend left=15]{rd}{X^\ddag Z}
			\arrow[Rightarrow,shift left=3.4ex,shorten <=1.7ex,shorten >=-0.3ex]{d} & \\
			A \arrow{u}{X} \arrow{r}[swap]{Z} & C
		\end{tikzcd}
	\]
	are respectively a right Kan lifting and a right Kan extension in $\bicatM$. Furthermore, $\bicatM$ being closed is equivalent to the existence of all right Kan liftings and right Kan extensions in $\bicatM$.
\end{proposition}

\begin{proof}
	This follows from the definitions.
\end{proof}

\begin{example}[The closed bicategory of profunctors]
	For small categories $\catA$ and $\catB$, a \emph{profunctor} $X\colon \catA\slashedrightarrow\catB$ is simply an ordinary functor $X\colon \catB^\op\times\catA \to \Set$. For profunctors $X\colon \catA\slashedrightarrow\catB$ and $Y\colon \catB \slashedrightarrow \catC$, their ``composition'' $Y\odot X$ is given by the coend
	\[ (Y\odot X)(c,a) = \int^{b\in\catB} Y(c,b)\times X(b,a). \]
	This composition makes profunctors into a bicategory $\Prof$. The Hom set functor $\Hom_\catA(\mplaceholder,\mplaceholder)\colon \catA^\op\times\catA\to \Set$ regarded as a profunctor $I_\catA\colon \catA\slashedrightarrow\catA$ performs the identity morphism of $\catA$ in $\Prof$.

	Furthermore, the bicategory $\Prof$ is closed. For profunctors $X\colon \catA\slashedrightarrow\catB$, $Y\colon \catB\slashedrightarrow\catC$, and $Z\colon \catA\slashedrightarrow\catC$, the profunctor $X^\ddag Z\colon \catB\slashedrightarrow\catC$ is defined by
	\[ X^\ddag Z(C,B) = \Fun(\catA,\Set)(X(B,\mplaceholder), Z(C,\mplaceholder)) = \int_{A\in\catA} {\Hom}\big(X(B,A),Z(C,A)\big), \]
	and the profunctor $Y_\ddag Z\colon \catA\slashedrightarrow\catB$ is defined by
	\[ Y_\ddag Z (B,A) = \Fun(\catC^\op,\Set)(Y(\mplaceholder,B),Z(\mplaceholder,A)) = \int_{C\in\catC} {\Hom}\big(Y(C,B),Z(C,A)\big). \]
	Then there are natural bijections
	\begin{align*}
		\Prof(\catA,\catC)(Y\odot X, Z) & \cong \Prof(\catB,\catC)(Y,X^\ddag Z), \\
		\Prof(\catA,\catC)(Y\odot X,Z)  & \cong \Prof(\catA,\catB)(X,Y_\ddag Z).
	\end{align*}
\end{example}

\begin{proposition}
	Let $\bicatM$ be a closed bicategory. Then for a morphism $X\colon A\to B$, we have
	\begin{align*}
		\text{$X$ is a left adjoint}  & \iff \text{$X$ commutes with all right Kan liftings,}   \\
		\text{$X$ is a right adjoint} & \iff \text{$X$ commutes with all right Kan extensions.}
	\end{align*}
\end{proposition}

\begin{proof}
	We show the first statement; the second is dual. The ($\Rightarrow$) direction follows from \cref{prop:adjoint_respects_Kan_extension/lifting}. For ($\Leftarrow$), since $\bicatM$ is closed, the statement follows from \cref{prop:adjunction_as_Kan_extension/lifting}.
\end{proof}

We also recall the notion of morphisms between bicategories.

\begin{definition}\label{def:proarrow_equipment}
	Let $\bicatK,\bicatL$ be bicategories. A \emph{lax functor} $\Phi\colon \catK\to \catL$ consists of
	\begin{itemize}
		\item a map $\Phi\colon \obj(\bicatK)\to \obj(\bicatL)$,
		\item for each pair $A,B\in \bicatK$ of objects, a functor $\Phi=\Phi_{AB}\colon \bicatK(A,B)\to\bicatL(\Phi(A),\Phi(B))$,
		\item for each object $A\in\bicatK$, a $2$-cell $\epsilon^A\colon \id_{\Phi(A)} \Rightarrow \Phi(\id_A)$, and
		\item for each triple $A,B,C\in \bicatK$ of objects, a $2$-cell $\mu^{A,B,C}_{g,f} \colon \Phi(g)\circ \Phi(f) \Rightarrow \Phi(g\circ f)$ natural in $f\in \bicatK(A,B)$ and $g\in \catK(B,C)$
	\end{itemize}
	such that these data satisfy the associativity and unity axioms (see \cite[Definition 4.1.2]{Johnson-Yau:2021} for the precise definition). We call $\epsilon^A$ the \emph{lax unity constraint} and $\mu^{A,B,C}_{g,f}$ the \emph{lax functoriality constraint}.
	
	A lax functor $\Phi\colon \catK\to \catL$ is called \emph{normal} (or \emph{unitary}) if all $\epsilon^A$ are invertible, and a \emph{pseudo-functor} if all $\epsilon^A$ and $\mu^{A,B,C}_{g,f}$ are invertible.
\end{definition}

It is immediate to see that any pseudo-functor preserves equivalences and adjunctions, but this does not hold for a lax functor in general.

\subsection{Proarrow equipments}\label{subsection:proarrow_equipments}

Let $\bicatK,\bicatM$ be bicategories.

\begin{definition}[{\cite{Wood:1982proarrow1}, \cite{Wood:1985proarrow2}}]
	A pseudo-functor $\equip\colon \bicatK\to\bicatM$ is called a \emph{proarrow equipment} if it satisfies:
	\begin{enumerate}
		\item $\equip$ is bijective on objects,
		\item $\equip$ is locally fully faithful, and
		\item for any $1$-morphism $f$ in $\bicatK$, $f_\ast$ has a right adjoint $f^*$ in $\bicatM$.
	\end{enumerate}
\end{definition}

Note that in recent years, an equipment often refers to a special double category (\cite{Shulman:2008framed}, \cite{Cruttwell-Shulman:2010}).

Since a proarrow equipment $\equip$ is bijective on objects, we henceforth identify the objects of $\bicatK$ and $\bicatM$, so $\obj(\bicatK)=\obj(\bicatM)$.

For a proarrow equipment $\equip\colon \bicatK\to\bicatM$, given a morphism $f\colon A \to B$ in $\bicatK$, there exists an adjunction $f_\ast \dashv f^*$ in $\bicatM$. We write the unit of this adjunction by $\overline{f}\colon \id_A \Rightarrow f^*\circ f_\ast$. Since the unit is a left Kan extension, for any $2$-morphism $\tau \colon f\Rightarrow g \colon A \to B$ in $\bicatK$, one obtains a unique $2$-morphism $\tau^*\colon g^*\Rightarrow f^* \colon B \to A$ in $\bicatM$ such that
\[
	\begin{tikzcd}[column sep=large]
		& B \arrow{d}{f^*} \\
		A \arrow{ru}[swap,pos=0.25,outer sep=-2pt]{f_*}[name=A1,above,pos=0.4]{}
		\arrow[shift left=1.5ex,bend left=30]{ru}[pos=0.25]{g_*}[name=B1,below,pos=0.45]{}
		\arrow[Rightarrow,from=A1,to=B1,shorten <=-2pt,shorten >=-2pt,"\tau_*",swap,pos=0.3]
		\arrow{r}[swap]{\id_A}
		& A \arrow[Rightarrow,u,shift left=3.3ex,shorten <=-0.8ex,shorten >=2ex,"\overline{f}",swap,pos=0.2]
	\end{tikzcd}
	\qquad=\qquad
	\begin{tikzcd}[column sep=large]
		B \arrow{rd}[swap,pos=0.75,outer sep=-2pt]{g^*}[name=A2,above,pos=0.6]{}
		\arrow[shift left=1.5ex,bend left=30]{rd}[pos=0.75]{f^*}[name=B2,below,pos=0.55]{}
		\arrow[Rightarrow,from=A2,to=B2,dotted,shorten <=-2pt,shorten >=-2pt,"\tau^*",pos=0.3] & \\
		A \arrow{u}{g_*} \arrow{r}[swap]{\id_A}
		\arrow[Rightarrow,u,shift right=3.3ex,shorten <=-0.8ex,shorten >=2ex,"\overline{g}",pos=0.2,outer sep=1pt] & A\rlap{.}
	\end{tikzcd}
\]
This correspondence defines a pseudo-functor
\[ (\mplaceholder)^* \colon \bicatK^\coop \to \bicatM \]
which is locally fully faithful since $\equip$ is.

A morphism $X$ in $\bicatM$ is called \emph{representable} if it can be written as $X\cong f_*$ for some morphism $f$ in $\bicatK$, and \emph{corepresentable} if it can be written as $X\cong f^*$.

\begin{example}
	The following are typical examples of proarrow equipments.
	\begin{enumerate}
		\item Let $\Cat$ be the $2$-category of small categories and $\Prof$ the bicategory of profunctors. For a functor $F\colon \catA \to \catB$, we define the profunctor $F_*\colon \catA\slashedrightarrow\catB$ by $F_*=\Hom_\catB(\mplaceholder,F(\mplaceholder))$. It is known that $F_*$ has a right adjoint $F^*=\Hom_\catB(F(\mplaceholder),\mplaceholder)\colon \catB\slashedrightarrow\catA$ in $\Prof$ (see \cite[Proposition 7.9.1]{Borceux:1994HoCA1} or \cite[Remark 5.2.1]{Loregian:2021coend_calculus} for example). Hence the mapping $F \mapsto F_*$ yields a proarrow equipment
		      \[ \equip\colon\Cat\to\Prof. \]

		\item More generally, if $\moncatV$ is a cosmos, then $\moncatV$-enriched profunctors $X\colon \catA\slashedrightarrow\catB$ between $\moncatV$-enriched small categories are defined similarly, and there exists a proarrow equipment
		      \[ \equip\colon\VCat\to\VProf,\quad F \mapsto \catB(\mplaceholder,F(\mplaceholder)). \]
	\end{enumerate}
\end{example}

\begin{example}
	(Other examples)
	\begin{enumerate}
		\item Let $\catC$ be a category with pullbacks, and let $\Span(\catC)$ denote the bicategory of spans in $\catC$. Then the assignment $(f\colon c\to d)\mapsto (c\xleftarrow{\id_c}c \xrightarrow{f}d)$ defines a proarrow equipment
		      \[ \equip\colon\catC\to\Span(\catC). \]
		      %% \item equipment of relation

		\item If $\catS$ is a finitely complete category, there exists a proarrow equipment
		      \[ \equip\colon \Cat(\catS) \to \Prof(\catS) \]
		      relating $\catS$-internal categories and $\catS$-internal profunctors.

		\item Let $\catname{TopGeom}$ denote the $2$-category of elementary toposes and geometric morphisms (with morphisms reversed from the direction of left adjoints). Let $\catname{TopLex}$ denote the $2$-category of elementary toposes and left exact functors. Then there exists a proarrow equipment
		      \[ \equip\colon \catname{TopGeom}^\op \to \catname{TopLex}^\co \]
		      given by taking left adjoints of geometric morphisms.

		\item Similarly, let $\catname{AbelGeom}$ denote the $2$-category of Abelian categories and geometric morphisms, and $\catname{AbelLex}$ the $2$-category of Abelian categories and left exact functors. Then there exists a proarrow equipment
		      \[ \equip\colon \catname{AbelGeom}^\op \to \catname{AbelLex}^\co \]
		      given by taking left adjoints of geometric morphisms.
	\end{enumerate}
\end{example}

\begin{proposition}[Yoneda {\cite[Proposition 3]{Wood:1982proarrow1}}]\label{prop:Yoneda_lemma_in_a_proarrow_equipment}
	Let $\equip\colon \bicatK\to\bicatM$ be a proarrow equipment.
	\begin{enumerate}
		\item For morphisms $f\colon B\to C$ in $\bicatK$ and $Z\colon A\to C$ in $\bicatM$, we have $\Rift_{f_*} Z \cong f^*\circ Z$.
		\item For morphisms $f\colon B\to A$ in $\bicatK$ and $Z\colon A\to C$ in $\bicatM$, we have $\Ran_{f^*} Z \cong Z\circ f_*$.
	\end{enumerate}
\end{proposition}

\begin{proof}
	It is immediate from \cref{prop:Kan_extension/lifting_along_adjoint}.
\end{proof}

\begin{corollary}
	Let $\equip\colon \bicatK\to\bicatM$ be a proarrow equipment. For morphisms $f\colon B\to C, g\colon A\to C$ in $\bicatK$, we have $\Rift_{f_*} g_* \cong f^*\circ g_* \cong \Ran_{g^*} f^*$.
\end{corollary}

\begin{comment}
Let us apply \cref{prop:Yoneda_lemma_in_a_proarrow_equipment} to the proarrow equipment $\equip\colon\Cat\to\Prof$ of profunctors.

\mymemo{？}
\end{comment}

\subsection{The notion of (co)limits in a proarrow equipment}\label{subsection:limits_in_a_proarrow_equipment}

Let $\equip\colon \bicatK\to\bicatM$ be a proarrow equipment.

\begin{definition}[{\cite[\S2]{Wood:1982proarrow1}}]\label{def:weighted_colimits}
	For morphisms $f\colon J\to A$ in $\bicatK$ and $W\colon M\to J$ in $\bicatM$, the \emph{$W$-weighted colimit} of $f$ is a morphism $\colim^W f = W\star f\colon M \to A$ in $\bicatK$ together with a right Kan lifting
	\[
		\begin{tikzcd}
			& M \arrow{d}{W}
			\arrow[Rightarrow,d,"\iota",pos=0.7,shift right=3.4ex,shorten <=1.7ex,shorten >=-0.3ex] \\
			A \arrow[bend left=15]{ru}{(W\star f)^*} \arrow{r}[swap]{f^*} & J
		\end{tikzcd}
	\]
	in $\bicatM$. If the $W$-weighted colimit of $f$ exists, then $(W\star f)^*=\Rift_{W} f^*$. In other words, $W\star f$ exists if and only if the right Kan lift $\Rift_{W} f^*$ exists and is corepresentable.

	Dually, for morphisms $f\colon J\to A$ in $\bicatK$ and $V\colon J\to M$ in $\bicatM$, the \emph{$V$-weighted limit} of $f$ is a morphism $\lim^V f = \{V,f\}\colon M \to A$ in $\bicatK$ together with a right Kan extension
	\[
		\begin{tikzcd}
			M \arrow[bend left=15]{rd}{\{V,f\}_*}
			\arrow[Rightarrow,shift left=3.4ex,shorten <=1.7ex,shorten >=-0.3ex,"\pi",swap,pos=0.7]{d} & \\
			J \arrow{u}{V} \arrow{r}[swap]{f_*} & A
		\end{tikzcd}
	\]
	in $\bicatM$. If the $V$-weighted limit of $f$ exists, then $\{V,f\}_*=\Ran_{V} f_*$. In other words, $\{V,f\}$ exists if and only if the right Kan extension $\Ran_{V} f_*$ exists and is representable.
\end{definition}

\begin{example}
	Consider the proarrow equipment $\VCat\to\VProf$ of enriched categories. Let $M=\catI$ be the unit $\catV$-category.
	For a $\catV$-functor $F\colon \catJ \to \catA$ and a $\catV$-profunctor $W\colon \catI\slashedrightarrow\catJ$, the right Kan lifting $\Rift_W F^*\colon\catA\slashedrightarrow\catI$ is given by
	\[ \Rift_W F^*(\ast,A) = \Fun(\catJ^\op,\catV)(W(\mplaceholder,\ast),\catA(F\mplaceholder,A)). \]
	Hence the $W$-weighted colimit of $F$ in the sense of \cref{def:weighted_colimits} is precisely the enriched colimit of $F$ weighted by the presheaf $W\colon \catJ^\op\cong \catJ^\op\otimes\catI\to\catV$ in the sense of \cite[\S3.1]{Kelly:1982Basic}.
\end{example}

\begin{proposition}[{\cite[Proposition 6]{Wood:1982proarrow1}}]
	For morphisms $f\colon J \to A$ in $\bicatK$ and $V\colon N \to M$, $W\colon M \to A$ in $\bicatM$, if the $W$-weighted colimit $W\star f$ exists, then we have an isomorphism
	\[ V\star (W\star f) \cong (W\circ V)\star f, \]
	either side existing if the other does.

	Dually, for suitable $U,W,f$, if the $W$-weighted limit $\{W,f\}$ exists, then we have an isomorphism
	\[ \{U,\{W,f\}\} \cong \{U\circ W,f\}, \]
	either side existing if the other does.
\end{proposition}

\begin{proof}
	We prove the first half. In general, $\Rift_V (\Rift_W f^*) \cong \Rift_{W\circ V} f^*$	so we have
	\[ (V\star (W\star f))^* \cong \Rift_V (W\star f)^* \cong  \Rift_V (\Rift_W f^*) \cong \Rift_{W\circ V} f^* \cong ((W\circ V)\star f)^* \]
	and since $\coequip$ is locally fully faithful, we get $V\star (W\star f) \cong (W\circ V)\star f$.
\end{proof}

\begin{proposition}[{\cite[Proposition 7]{Wood:1982proarrow1}}]\label{prop:colimit_weighted_by_corepresentables}
	For morphisms $f\colon J\to A$, $w\colon M\to J$ in $\bicatK$, we have
	\[ w_* \star f \cong f\circ w \cong \{ w^*,f \}. \]
\end{proposition}

\begin{proof}
	By \cref{prop:Yoneda_lemma_in_a_proarrow_equipment}, we have
	\[ (f\circ w)^* \cong w^*\circ f^* \cong \Rift_{w_*} f^* = (w_* \star f)^*, \]
	which implies $ w_* \star f \cong f\circ w$. The other isomorphism is proved similarly.
\end{proof}

\begin{definition}
	A morphism $g\colon A\to B$ in $\bicatK$ is said to \emph{preserve} a weighted colimit $W\star f$ if $g^*$ commutes with the right Kan lifting $(W\star f)^*=\Rift_{W} f^*$.

	Similarly, $g\colon A\to B$ is said to \emph{preserve} a weighted limit $\{V,f\}$ if $g_*$ commutes with the right Kan extension $\{V,f\}_*=\Ran_{V} f_*$.
\end{definition}

\begin{proposition}[{\cite[Proposition 8]{Wood:1982proarrow1}}]\label{prop:LAPC_in_a_proarrow_equipment}
	Left adjoints preserve all weighted colimits. Dually, right adjoints preserve all weighted limits.
\end{proposition}

\begin{proof}
	If a morphism $f$ in $\bicatK$ has a right adjoint $u$, then we get the adjunction $f^*\dashv u^*$ in $\bicatM$, so by \cref{prop:adjoint_respects_Kan_extension/lifting}, $f^*$ commutes with all right Kan liftings. Hence $f$ preserves all weighted colimits in particular.
\end{proof}

\begin{proposition}[Formal criterion for representability {\cite[Proposition 9]{Wood:1982proarrow1}}]\label{prop:formal_criterion_for_representability}
	For morphisms $f\colon A\to B$ in $\bicatK$ and $X\colon A\to B$ in $\bicatM$, the following are equivalent:
	\begin{enumerate}[label=\equivitem]
		\item $X\cong f_*$.
		\item The weighted colimit $X\star \id_B$ exists, $X\star \id_B \cong f$, and $X$ commutes with all right Kan liftings.
		\item The weighted colimit $X\star \id_B$ exists, $X\star \id_B \cong f$, and $X$ commutes with the right Kan lifting $(X\star \id_B)^*=\Rift_{X} \id_B^*$.
	\end{enumerate}
\end{proposition}

\begin{proof}
	(i) $\Rightarrow$ (ii): By \cref{prop:colimit_weighted_by_corepresentables}, we have
	\[ X\star \id_B \cong f_*\star \id_B \cong \id_B \circ f \cong f. \]
	Also, since $X\cong f_*$ has a right adjoint, it commutes with all right Kan liftings.

	(ii) $\Rightarrow$ (iii): Clear.

	(iii) $\Rightarrow$ (i): From $X\star \id_B \cong f$, we have
	\[ f^* \cong (X\star \id_B)^* \cong \Rift_X \id_B^* \cong \Rift_X \id_B. \]
	Since $X$ commutes with this Kan lifting, \cref{prop:adjunction_as_Kan_extension/lifting} shows that there is an adjunction $X\dashv f^*$. As $f_*\dashv f^*$, we get $X\cong f_*$.
\end{proof}

Note that for morphisms $f\colon A \to B$, $u\colon B \to A$ in $\bicatK$, $f\dashv u$ holds in $\bicatK$ if and only if $f^*\cong u_*$ holds in $\bicatM$.

\begin{corollary}[Formal adjoint arrow theorem {\cite[Corollary 10]{Wood:1982proarrow1}}]\label{cor:formal_adjoint_arrow_theorem}
	For morphisms $f\colon A \to B$, $u\colon B \to A$ in $\bicatK$, the following are equivalent:
	\begin{enumerate}[label=\equivitem]
		\item $f\dashv u$ holds in $\bicatK$.
		\item The weighted colimit $f^*\star \id_A$ exists, $f^*\star \id_A \cong u$, and $f$ preserves all weighted colimits.
		\item The weighted colimit $f^*\star \id_A$ exists, $f^*\star \id_A \cong u$, and $f$ preserves the weighted colimit $f^*\star \id_A$.
	\end{enumerate}
\end{corollary}

\begin{proof}
	This follows from \cref{prop:formal_criterion_for_representability} by taking $X=f^*\colon B\to A$.
\end{proof}

\begin{definition}\label{def:fully_faithful_morphism}
	A morphism $f\colon A \to B$ in $\bicatK$ is called \emph{fully faithful} if the unit of the adjunction $f_* \dashv f^*$ is an isomorphism.
\end{definition}

\begin{example}
	In the proarrow equipment $\VCat\to\VProf$ of enriched categories over a cosmos $\moncatV$, a $\moncatV$-functor $F\colon \catA\to\catB$ is fully faithful in the above sense if and only if it is fully faithful in the enriched sense; that is, all $F_{AA'}\colon \catA(A,A)\to \catB(FA,FA')$ are isomorphisms for $A,A'\in \catA$.
\end{example}

\begin{proposition}[{\cite[Proposition 13]{Wood:1982proarrow1}}]
	For a left adjoint $f\colon A\to B$ in $\bicatK$ with unit $\eta$, the following are equivalent.
	\begin{enumerate}[label=\equivitem]
		\item $f$ is fully faithful.
		\item $\eta$ is an isomorphism.
	\end{enumerate}
\end{proposition}

\begin{proof}
	If $f$ has a right adjoint with unit $\eta$, then $\eta_*$ is the unit of the adjunction $f_* \dashv f^*$. Thus the assertion follows, since $(-)_*$ is locally fully faithful.
\end{proof}

%% ここまでは Wood: Abstract Proarrow I の内容

\subsection{Absolute limits and Cauchy completeness}\label{subsection:absolute_limits_and_Cauchy_completeness}
%% cf Street:1983, Absolute colimits in enriched categories
% \cite{Street:1983}

Most of the contents of this subsection are not included in \cite{Wood:1982proarrow1}, but they are considered part of the folklore.

Let $\equip\colon \bicatK\to\bicatM$ be a proarrow equipment.

\begin{definition}
	For morphisms $f\colon J\to A$ in $\bicatK$ and $W\colon M\to J$ in $\bicatM$, the $W$-weighted colimit $W\star f\colon M \to A$ of $f$ is said to be \emph{absolute} if it is preserved by every morphism $g\colon A\to B$ in $\bicatK$.
	Similarly, the $V$-weighted limit $\{V,f\}\colon M \to A$ of $f$ is said to be \emph{absolute} if it is preserved by every morphism $g\colon A\to B$ in $\bicatK$.
\end{definition}

\begin{proposition}\label{prop:left_adjoint_weight_is_absolute}
	Let $W\colon M\to J$ be a morphism in $\bicatM$. If $W$ has a right adjoint $V\colon J\to M$, then all $W$-weighted colimits are absolute.
	Similarly, if $W\colon J\to M$ has a left adjoint $U\colon M\to J$, then all $W$-weighted limits are absolute.
\end{proposition}

\begin{proof}
	Suppose $f\colon J\to A$ is a morphism in $\bicatK$ and the $W$-weighted colimit $W\star f\colon M\to A$ exists. Then $(W\star f)^* = \Rift_W f^*$. Since $W$ is a left adjoint, \cref{prop:Kan_extension/lifting_along_adjoint} shows $\Rift_W f^* \cong V\circ f^*$. Hence for any $g\colon A \to B$ in $\bicatK$, we have
	\[ \Rift_W f^* \circ g^* \cong V\circ f^*\circ g^* \cong V\circ (gf)^* \cong \Rift_W (gf)^*. \]
	Thus $g^*$ commutes with the right Kan lifting $\Rift_W f^*$, and $g$ preserves the colimit $W\star f$. The case of weighted limits is similar.
\end{proof}

%\mymemo{Question: Conversely, when does "an absolute weight has a right adjoint" hold?} It holds for the proarrow equipment $\equip\colon \VCat\to\VProf$ of enriched categories and enriched profunctors. \mymemo{What about in general?}

%% Garner: Diagrammatic characterisation of enriched absolute colimits
% \cite{Garner:2014diagrammatic}
\begin{proposition}[A generalization of ~\cite{Garner:2014diagrammatic}]\label{prop:left-adj-weighted_colimit_is_equal_to_right-adj-weighted_limit}
	Let $W\colon M\to J$ be a morphism in $\bicatM$ with a right adjoint $V\colon J\to M$. For morphisms $f\colon J\to A$ and $z\colon M\to A$ in $\bicatK$, $z$ is the (absolute) $W$-weighted colimit of $f$ if and only if $z$ is the (absolute) $V$-weighted limit of $f$.
\end{proposition}

\begin{proof}
	Since $W\dashv V$, we see from \cref{prop:Kan_extension/lifting_along_adjoint}
	\[ \Rift_W f^* \cong V\circ f^*,\qquad \Ran_V f_* \cong f_*\circ W. \]
	In particular, we get an adjunction $\Ran_V f_* \dashv \Rift_W f^*$. Therefore, $\Rift_W f^*$ being corepresentable by $z$ is equivalent to $\Ran_V f_*$ being representable by $z$.
\end{proof}

\begin{definition}
	An object $A\in \bicatK$ is said to be \emph{Cauchy complete} if every left adjoint morphism $\Phi\colon D\to A$ in $\bicatM$ is representable.
\end{definition}

\begin{proposition}\label{prop:Cauchy_complete_is_having_left-adjoint-weighted_colimits}
	For an object $A\in \bicatK$, the following are equivalent:
	\begin{enumerate}[label=\equivitem]
		\item $A$ is Cauchy complete.
		\item $A$ has all colimits weighted by left adjoints.
		\item $A$ has all limits weighted by right adjoints.
	\end{enumerate}
\end{proposition}

\begin{proof}
	(i) $\Rightarrow$ (ii): Let $W\colon D \to J$ be a left adjoint in $\bicatM$ with right adjoint $V\colon J \to D$. For any morphism $f\colon J \to A$ in $\bicatK$, by \cref{prop:left-adj-weighted_colimit_is_equal_to_right-adj-weighted_limit}, the existence of $\colim^W f$ is equivalent to the existence of $\lim^V f$, which holds if and only if $\Ran_V f_* \cong f_*\circ W$ is representable. Since $f_*\circ W$ is a left adjoint in $\bicatM$, it is representable by the Cauchy completeness of $A$.

	(ii) $\Rightarrow$ (i): Let $W\colon D \to A$ be a left adjoint in $\bicatM$ with right adjoint $V\colon A \to D$. Since $A$ has $W$-weighted colimits, in particular $g\coloneqq \colim^W \id_A$ exists. Then $g^*$ is the right Kan lifting $\Rift_W \id_A^*$. Now $W\dashv V$ implies $V\cong g^*$, and hence $g_*\cong W$.

	(ii) $\Leftrightarrow$ (iii): This follows from \cref{prop:left-adj-weighted_colimit_is_equal_to_right-adj-weighted_limit}.
\end{proof}

Therefore, by \cref{prop:left_adjoint_weight_is_absolute}, an object is Cauchy complete if it has all absolute weighted colimits.
%\mymemo{Question: When does the converse hold?} It holds for the proarrow equipment $\equip\colon \VCat\to\VProf$ of enriched categories. \mymemo{What about in general?}

\begin{remark}
	For the proarrow equipment $\equip\colon \VCat\to\VProf$ of enriched categories and enriched profunctors, the converse of \cref{prop:left_adjoint_weight_is_absolute} holds (\cite{Street:1983}). That is, weights of absolute colimits have right adjoints; and hence an enriched category is Cauchy complete if and only if it has all absolute weighted colimits.
	It is not clear to the author whether this is the case for general proarrow equipments.
\end{remark}

\printbibliography %% biblatex 

\end{document}